\documentclass[reqno,12pt]{amsart}
\usepackage{amsmath, amssymb}
\usepackage{amsthm, amsbsy}
\usepackage{graphicx}
\usepackage{float}
\usepackage{color}   
\usepackage{hyperref}
\hypersetup{
    colorlinks=true, 
    citecolor=blue,
    filecolor=black,
    urlcolor=black
    linktoc=all,     
    linkcolor=blue,  
}
\usepackage[ruled]{caption}

\theoremstyle{plain}
\newtheorem{theorem}{\indent\bf Theorem}[section]
\newtheorem{lemma}[theorem]{\indent\bf Lemma}
\newtheorem{corollary}[theorem]{\indent\bf Corollary}
\newtheorem{proposition}[theorem]{\indent\bf Proposition}
\theoremstyle{definition}
\newtheorem{definition}[theorem]{\indent\bf Definition}
\newtheorem{remark}[theorem]{\indent\bf Remark}

\begin{document}
\title[Resultants and irreducibility criteria]{Irreducibility criteria for pairs of polynomials whose resultant is a prime number}
\author[N.C. Bonciocat]{Nicolae Ciprian Bonciocat}
\address{Simion Stoilow Institute of Mathematics of the Romanian Academy, Research Unit 7, P.O. Box 1-764, Bucharest 014700, Romania}
\email{Nicolae.Bonciocat@imar.ro}
\keywords{Resultant, irreducible polynomials, prime number.}
\subjclass[2020]{Primary 11C08, Secondary 11R09}

\dedicatory{Dedicated to the memory of Professors Lauren\c tiu Panaitopol and Doru \c Stef\u anescu}

\begin{abstract}
We obtain various irreducibility criteria for pairs of polynomials $(f(X),g(X))$ with integer coefficients whose resultant $Res(f,g)$ is a prime number, or is divisible by a sufficiently large prime number, and also for some of their linear combinations $Mf(X)+Ng(X)$ with integer scalars $M$ and $N$. In particular, we find irreducibility conditions for polynomials with coefficients obtained by representing primes by certain quadratic forms.  
The irreducibility criteria will appear as corollaries of more general results providing upper bounds for the number of irreducible factors of each one of $f$ and $g$, counting multiplicities, that depend on the prime factorization of $Res(f,g)$, and on the distances between the roots of $f$ and those of $g$. Similar results will be also obtained for pairs of bivariate polynomials $(f(X,Y),g(X,Y))$ over an arbitrary field $K$, using information on the canonical decomposition of their resultant $Res_Y(f,g)$, and on the location of their roots in an algebraic closure of $K(X)$, studied in a non-Archimedean setting.
\end{abstract}

\maketitle

\tableofcontents

\section{Introduction}

Many of the classical or more recent irreducibility criteria in the literature rely on the prime factorization of the value that a given polynomial assumes at a specified integer argument. For instance, one may find such results in works by St\"ackel \cite{Stackel}, Weisner \cite{Weisner}, Ore \cite{Ore}, P\'olya and G. Szeg\"o \cite{PolyaSzego}, Dorwart \cite{Dorwart}, Brillhart, Filaseta, Odlyzko, Gross, Cole, Dunn \cite{Brillhart}, \cite{Filaseta1}, \cite{Filaseta2}, \cite{FilasetaGross}, \cite{CDF}, Ram Murty \cite{RamMurty}, Girstmair \cite{Girstmair}, Guersenzvaig \cite{Guersenzvaig}, Singh, Kumar and Garg \cite{SinghKumar}, \cite{SinghGarg}, to name just a few. For other related results we refer the reader to Prasolov \cite{Prasolov}, and Mignotte and \c Stef\u  anescu \cite{MignotteStefanescu}, and for results on the prime values assumed by reducible polynomials to Chen, Kun, Pete and Ruzsa \cite{ChenKunPeteRuzsa}.

Given an integer $m$ and a polynomial $f\in\mathbb{Z}[X]$ we may view $f(m)$ as $Res(f(X),X-m)$, the resultant of $f(X)$ and $X-m$, so one natural way to obtain new irreducibility criteria is to replace the condition that $f$ takes a prime value at some integer argument $m$ by a more general one, that asks the resultant of $f$ and another polynomial $g\in\mathbb{Z}[X]$ to be a prime number. As we shall see, together with some extra conditions on the location of their roots, this will ensure the irreducibility of both polynomials $f$ and $g$. 

There are several irreducibility criteria in the literature that rely on the properties of resultants. One such result, that uses a resultant condition to generalize the irreducibility criteria of Sch\"onemann \cite{Schonemann} and Akira \cite{Akira}, was obtained by Panaitopol and \c Stef\u anescu in \cite{PanStef}. 
\medskip

\emph{Let $(K,v)$ be a discrete valued field, $A$ the valuation ring of $v$, $\overline{K}$ its residue field, $p\in K$ such that $v(p)=1$ and $X$ an indeterminate over $K$. Let $F\in A[X]$ be a monic polynomial such that $F=f_1^{m_1}\cdots f_n^{m_n}+pg$, where $f_1,\dots ,f_n,g\in A[X]$, $f_1,\dots ,f_n$ are monic, $\overline{f}_1,\dots ,\overline{f}_n$ are nonconstant and irreducible in $\overline{K}[X]$ and $\overline{f}_i$ does not divide $\overline{g}$ if $m_i\geq 2$. Let $h=f_1^{m_1}\cdots f_n^{m_n}$. If, for every nontrivial factorization $h=h_1h_2$, there is $j\in\{1,2\}$ and $i\in\{1,\dots,n\}$ such that $f_i$ divides $h_j$ and, for all divisors $d$ of $Res(f_i,g)$, $Res(\overline{f}_i,\overline{h/h_j})\neq \overline{d}$, then the polynomial $F$ is irreducible in $K[X]$.}
\medskip

In a series of papers, Cavachi, M. V\^aj\^aitu and Zaharescu \cite{Cavachi}, \cite{CVZ1}, \cite{CVZ2} used some properties of the resultant to obtain irreducibility criteria for linear combinations of relatively prime polynomials. Their method was later adapted by various authors to find irreducibility conditions for compositions and multiplicative convolutions of polynomials, in both univariate and multivariate cases \cite{BB1}, \cite{BBC}, \cite{AZ1}, \cite{AZ2} and \cite{BBCM}. Some recent results on sums of relatively prime polynomials that use similar techniques relying on resultant properties  have been obtained by Zhang, Yuan and Zhou \cite{Zhang}.
For a standard reference on resultants and their properties, we refer the reader to \cite{GKZ}, for instance.

In this paper we will first obtain irreducibility conditions for pairs of polynomials with integer coefficients whose resultant is a prime number, or is divisible by a suitable prime number, and also for some of their linear combinations. We will then prove similar results for pairs of multivariate polynomials over a field $K$, whose resultant with respect to one of the indeterminate is an irreducible polynomial, or has a suitable irreducible factor. To get a glimpse on the irreducibility conditions that we will obtain, we will only mention here the simplest ones, as follows.
\medskip

{\bf Theorem A} {\em Two polynomials $f,g\in\mathbb{Z}[X]$ with $f(0)g(0)\neq 0$ and $|Res(f,g)|$ a prime number are both irreducible in each one of the following cases:

i) \thinspace \thinspace $|\theta-\xi|>1$ for every root $\theta $ of $f$ and every root $\xi $ of $g$;

ii) $|\frac{1}{\theta}-\frac{1}{\xi}|>1$ for every root $\theta $ of $f$ and every root $\xi $ of $g$.
}
\medskip

In particular, when $g$ is an irreducible quadratic polynomial, the irreducibility conditions for $f$ reduce to finding a suitable prime number represented by a certain quadratic form, as in the following two results:
\medskip

{\bf Theorem B} {\em If we write a prime number as $(a_{0}-a_{2}+a_{4}-\dots
)^{2}+(a_{1}-a_{3}+a_{5}-\dots )^{2}$ with $a_{i}\in \mathbb{Z}$ and $|a_{0}|>\sum_{i> 0}|a_{i}|\sqrt{2^i}$, then the polynomial $\sum_{i}a_{i}X^{i}$ is irreducible over $\mathbb{Q}$.}
\medskip

{\bf Theorem C} {\em Let $a_{0},a_{1},\dots,a_{n}$ be integers with $a_0a_n\neq 0$ and $|a_0|>\sum_{i=1}^{n}|a_{i}|2^i$, and let
$S_{j}=\sum _{i\equiv j\thinspace ({\rm mod}\thinspace 3)}a_{i}$ for $j\in\{0,1,2\}$.
If $S_0^2+S_1^2+S_2^2-S_0S_1-S_0S_2-S_1S_2$ is a prime number, then the polynomial $\sum_{i=0}^{n}a_{i}X^{i}$ is irreducible over $\mathbb{Q}$.
}
\medskip

A bivariate analogue of Theorem A that we will prove is the following result.
\medskip

{\bf Theorem D} {\em Let $K$ be a field, $f(X,Y)=\sum\nolimits_{i=0}^{n}a_{i}Y^{i},%
\ g(X,Y)=\sum\nolimits_{i=0}^{m}b_{i}Y^{i}\in K[X,Y]$, with $a_{0},\dots
,a_{n},b_{0},\dots ,b_{m}\in K[X]$, $a_{0}a_{n}b_{0}b_{m}\neq 0$, and assume that $f$ and $g$ have no nonconstant factors in $K[X]$, and that $Res_{Y}(f,g)$ is irreducible over $K$. If 
\[
\deg a_{0}>\max \{\deg a_{1},\ldots ,\deg a_{n}\}\quad and\quad \deg
b_{m}\geq \max \{\deg b_{0},\ldots ,\deg b_{m-1}\},
\]
then both $f$ and $g$ are irreducible in $K[X,Y]$. The same conclusion will also hold if we interchange the signs $>$ and $\geq$ in these two inequalities.
}
\medskip

The reader may naturally wonder how sharp the above results are. For instance, the condition that $|\theta-\xi|>1$ for every root $\theta $ of $f$ and every root $\xi $ of $g$ in the statement of Theorem A is best possible, in the sense that there exist polynomials $f,g\in \mathbb{Z}[X]$ whose roots $\theta _{i}$ and 
$\xi _{j}$ satisfy $\min_{i,j}|\theta _{i}-\xi _{j}|=1$, $|Res(f,g)|$ is a prime number, while both $f$ and $g$ are reducible over $\mathbb{Q}$. In this respect, let us
consider the polynomials $f(X)=X^{2}-8X+15$ and $g(X)=X^{2}-10X+24$, with
roots $\theta _{1}=3$, $\theta _{2}=5$, $\xi _{1}=4$, $\xi _{2}=6$. Here $|Res(f,g)|=3$, $\min_{i,j}|\theta _{i}-\xi _{j}|=1$, while $f$ and $g$ are obviously reducible over $\mathbb{Q}$. The condition that $|\frac{1}{\theta}-\frac{1}{\xi}|>1$ for every root $\theta $ of $f$ and every root $\xi $ of $g$ is also best possible, and to see this, one may consider the reciprocals of $f$ and $g$, that is $\tilde{f}(X)=15X^{2}-8X+1$ and $\tilde{g}(X)=24X^{2}-10X+1$.

The irreducibility criteria that we will obtain will be deduced from some more general results that provide upper bounds for the number of irreducible factors of our polynomials, multiplicities counted, and which will be stated and proved in the following section. In Section \ref{magnitudeoffandg} we will prove some irreducibility criteria that depend on the magnitude of the coefficients of $f$ and $g$. Some irreducibility criteria for the case when $g$ is linear and for the case when $g$ is quadratic will be given in Section \ref{glinear} and Section \ref{gquadratic}, respectively. 
In Section \ref{LinearCombinations} we will show that in some cases the irreducibility of a pair $(f(X),g(X))$ extends locally to linear combinations $Mf(X)+Ng(X)$ with small integer scalars $M,N$. Some analogous results for multivariate polynomials over arbitrary fields will be provided in Section \ref{multivariate}. The paper will end with a series of examples of polynomials that are proved to be irreducible by using some of these irreducibility criteria. 

\section{General irreducibility criteria for pairs $(f,g)$ of polynomials}\label{ggeneral}
We will first fix some definitions and notations. 

\begin{definition}\label{PrimaDefinitie}
i) For a nonzero integer $n$ we will denote by $\Omega(n)$ the total number of prime factors of $n$, counting multiplicities ($\Omega(\pm 1)=0$). 

ii) For an integer $n\geq 1$, and each positive integer $k$ we define
\[
d_k(n):= \max\{d:d\mid n\ {\rm and}\ d\leq \sqrt[k+1]{n}\thinspace \}.
\]

We may view $d_k(n)$ as {\em the best approximation from below with divisors of $n$, of the $(k+1)$th root of $n$}.
We note that $d_k(n)$ is a decreasing, eventually constant sequence. More precisely, $d_k(n)$ will eventually become equal to $1$.
Throughout this paper, we will only use this definition in the case that $n=|Res(f,g)|$, with $(f,g)$ a pair of nonconstant, relatively prime polynomials with integer coefficients, and instead of $d_k(|Res(f,g)|)$, we will simply write $d_k$.
\end{definition}

We will first use the properties of $d_k$ to study the irreducible factors of $f$ and $g$ in the case when the distances between the roots of $f$ and those of $g$ are sufficiently large. Later on we will also consider the case when the roots of $f$ are close to those of $g$, and the corresponding results will rely on the properties of $\frac{Res(f,g)}{d_k}$. 

The irreducibility criteria that we will first prove will be derived from two more general results that provide upper bounds for the total number of irreducible factors simultaneously for $f$ and for $g$. The statements of these results are symmetric with respect to $f$ and $g$, and require no information on their coefficients.

\begin{theorem}
\label{thm0}Let the roots of two relatively prime polynomials $f,g\in \mathbb{Z}[X]$ be $\theta _{1},\dots ,\theta _{n}$ and $\xi _{1},\dots ,\xi _{m}$, respectively, and assume that $\theta _{1}\cdots \theta _{n}\xi _{1}\cdots \xi _{m}\neq 0$. If for some positive divisor $d$ of $Res(f,g)$ we have 
\[
\min_{i,j}|\theta _{i}-\xi _{j}|>d^{\smallskip 1/\min (m,n)}\quad or \quad \min_{i,j}\left|\frac{1}{\theta _{i}}-\frac{1}{\xi _{j}}\right|>d^{1/\min (m,n)},
\]
then each one of $f$ and $g $ is a product of at most $\Omega (\frac{Res(f,g)}{d})$ irreducible factors over $\mathbb{Q}$. 
\end{theorem}

\begin{theorem}
\label{thm0dk}Let the roots of two relatively prime polynomials $f,g\in \mathbb{Z}[X]$ be $\theta _{1},\dots ,\theta _{n}$ and $\xi _{1},\dots ,\xi _{m}$, respectively, and assume that $\theta _{1}\cdots \theta _{n}\xi _{1}\cdots \xi _{m}\neq 0$. If for some integer $k>0$ we have 
\[
\min_{i,j}|\theta _{i}-\xi _{j}|>d_k^{\smallskip 1/\min (m,n)}\quad or \quad \min_{i,j}\left|\frac{1}{\theta _{i}}-\frac{1}{\xi _{j}}\right|>d_k^{1/\min (m,n)},
\]
then each one of $f$ and $g $ is a product of at most $k$ irreducible factors over $\mathbb{Q}$. 
\end{theorem}
As we shall see in Proposition \ref{EquivalentConditions}, for $d=d_k$, the conclusion in Theorem \ref{thm0} coincides with that in Theorem \ref{thm0dk} for $k\leq \frac{\max (m,n)-1}{2}$. Two immediate consequences of Theorem \ref{thm0} and Theorem \ref{thm0dk} are the following irreducibility criteria for pairs of polynomials, respectively. 
\begin{corollary}
\label{coro1thm0}Let the roots of $f,g\in \mathbb{Z}[X]$ be $\theta
_{1},\dots ,\theta _{n}$ and $\xi _{1},\dots ,\xi _{m}$, respectively, and
assume that $\theta _{1}\cdots \theta _{n}\xi _{1}\cdots \xi _{m}\neq 0$, and that $|Res(f,g)|=p\cdot q$, where $p$ is a prime number and $q$ is a
positive integer. If 
\[
\min_{i,j}|\theta _{i}-\xi _{j}|>q^{1/\min(m,n)}\quad or \quad \min_{i,j}\left|\frac{1}{\theta _{i}}-\frac{1}{\xi _{j}}\right|>q^{1/\min (m,n)},
\]
then both $f$ and $g$ are irreducible over $\mathbb{Q}$.
\end{corollary}

\begin{corollary}
\label{corothm0d1}Let the roots of $f,g\in \mathbb{Z}[X]$ be $\theta
_{1},\dots ,\theta _{n}$ and $\xi _{1},\dots ,\xi _{m}$, respectively, and
assume that $\theta _{1}\cdots \theta _{n}\xi _{1}\cdots \xi _{m}\neq 0$. If 
\[
\min_{i,j}|\theta _{i}-\xi _{j}|>d_1^{1/\min(m,n)}\quad or \quad \min_{i,j}\left|\frac{1}{\theta _{i}}-\frac{1}{\xi _{j}}\right|>d_1^{1/\min (m,n)},
\]
then both $f$ and $g$ are irreducible over $\mathbb{Q}$.
\end{corollary}
We will also prove that the conditions in these two irreducibility criteria are equivalent if at least one of $f$ and $g$ has degree at least $3$.
\begin{proposition}\label{EquivalentConditions} i) For $d=d_k$, the conclusions in Theorem \ref{thm0} and Theorem \ref{thm0dk} coincide for $k\leq \frac{\max (m,n)-1}{2}$;\hspace{1mm}  
ii) Corollary \ref{coro1thm0} and Corollary \ref{corothm0d1} are equivalent if $\max (m,n)\geq 3$.
\end{proposition}
The simplest instance of both Corollary \ref{coro1thm0} and Corollary \ref{corothm0d1} is our Theorem A, which may be obtained by letting $q=1$ (implying that $d_1=1$ as well):

\begin{corollary}
\label{coro2thm0} 
Two polynomials $f,g\in\mathbb{Z}[X]$ with $f(0)g(0)\neq 0$ and $|Res(f,g)|$ a prime number are both irreducible in each one of the following cases:

i) \thinspace \thinspace $|\theta-\xi|>1$ for every root $\theta $ of $f$ and every root $\xi $ of $g$;

ii) $|\frac{1}{\theta}-\frac{1}{\xi}|>1$ for every root $\theta $ of $f$ and every root $\xi $ of $g$.
\end{corollary}

For the proof of Theorem \ref{thm0} and Theorem \ref{thm0dk} we will need the following two lemmas, that are of independent interest in studying the irreducible factors of a polynomial. In both lemmas we will consider the general situation in which $f$ is assumed to have no nonconstant factors of degrees less than some positive integer constant $r$.

\begin{lemma}
\label{thm-1}Let $f,g\in \mathbb{Z}[X]$ with $f(0)\neq 0$ be relatively prime polynomials, let $d$ be a positive divisor of $Res(f,g)$, and assume that $f$ has no nonconstant factors of degree less than a positive integer $r<n$. If $|g(\theta )|>\sqrt[r]{d}$ for each root $\theta $ of $f$, or if $|\frac{g(\theta )}{\theta ^{\deg g}}|>\sqrt[r]{d}$ for each root $\theta $ of $f$, then $f$ is a product of at most $\Omega (\frac{Res(f,g)}{d})$ irreducible factors over $\mathbb{Q}$.
\end{lemma}
\begin{proof}
First of all we should note that our assumption that $f$ has no nonconstant factor of degree smaller than $r$ already provides us with the naive upper bound $\lfloor\frac{n}{r}\rfloor$ for the total number of irreducible factors of $f$ counted with multiplicities. Assume now that $f$ and $g$ factorize as 
\begin{eqnarray*}
f(X) &=&a_{n}(X-\theta _{1})\cdots (X-\theta _{n}), \\
g(X) &=&b_{m}(X-\xi _{1})\cdots (X-\xi _{m}),
\end{eqnarray*}
with $a_{n},b_{m}\in \mathbb{Z}$, $a_{n}b_{m}\neq 0$ and $\theta _{1},\ldots
,\theta _{n},\xi _{1},\ldots ,\xi _{m}\in \mathbb{C}$. Suppose that our polynomial $f$
decomposes as $f(X)=f_{1}(X)\cdots f_{k}(X)$, with $f_{1},\ldots ,f_{k}\in 
\mathbb{Z}[X]$, $\deg f_{1}\geq r,\ldots ,\deg f_{k}\geq r$ and $k>\Omega
(Res(f,g)/d)$. Then, since 
\begin{equation}\label{descompunere}
d\cdot \frac{Res(f,g)}{d}=Res(f_{1},g)\cdots Res(f_{k},g),
\end{equation}
and in the right side of this equality we have more terms than the total number of prime factors of $\frac{Res(f,g)}{d}$ counted with their multiplicities, we deduce that at least one of the integers $Res(f_{1},g),\ldots
,Res(f_{k},g)$, say $Res(f_{1},g)$, will not be affected after divison by $\frac{Res(f,g)}{d}$ in both sides of (\ref{descompunere}). This shows that $Res(f_{1},g)$ must divide $d$, so in particular we must have $|Res(f_{1},g)|\leq d$. Now, since $f_{1}$ is a factor of $f$, it will
factorize over $\mathbb{C}$ as $f_{1}(X)=c_{t}(X-\theta _{1})\cdots
(X-\theta _{t})$, say, with $t\geq r$ and $|c_{t}|\geq 1$, so \vspace{-1mm} 
\begin{equation}\label{InegRez}
|Res(f_{1},g)|=|c_{t}|^{m}|b_{m}|^{t}\prod\limits_{i=1}^{t}\prod%
\limits_{j=1}^{m}|\theta _{i}-\xi _{j}|=|c_{t}|^{m}|g(\theta _{1})\cdots
g(\theta _{t})|\leq d.
\end{equation}
Assume first that each one of $|g(\theta _1)|,\dots ,|g(\theta _t)|$ exceeds $\sqrt[r]{d}$. Then, since $|c_{t}|\geq 1$ we must actually have $|c_{t}|^{m}|g(\theta _{1})\cdots g(\theta _{t})|>(\sqrt[r]{d})^t\geq d$, as $t\geq r$, which contradicts (\ref{InegRez}) and completes the proof in this case.

Note that since $f(0)\neq 0$, we may also write $|c_{t}|^{m}|g(\theta _{1})\cdots
g(\theta _{t})|=|c_{0}|^{m}|\frac{g(\theta _{1})}{\theta _1^m}|\cdots |\frac{g(\theta _{t})}{\theta _t^m}|$, with $c_0\neq 0$ the free term of $f_1$, so (\ref{InegRez}) may be also written as
\begin{equation}\label{InegRezReciprocal}
|Res(f_{1},g)|=|c_{0}|^{m}\left|\frac{g(\theta _{1})}{\theta _1^m}\right|\cdots \left|\frac{g(\theta _{t})}{\theta _t^m}\right|\leq d.
\end{equation}
Thus, if we assume that each one of $|\frac{g(\theta _{1})}{\theta _1^m}|,\dots ,|\frac{g(\theta _{n})}{\theta _n^m}|$ exceeds $\sqrt[r]{d}$ and use the fact that
$|c_{0}|\geq 1$, we deduce again that $|Res(f_{1},g)|>(\sqrt[r]{d})^t\geq d$, which contradicts (\ref{InegRezReciprocal}) and completes the proof. 

We mention here that if we know that $|\theta _i|\geq 1$ for each $i$, we should use condition $|g(\theta _i)|>\sqrt[r]{d}$, $i=1,\dots ,n$, while if we know that $|\theta _i|<1$ for each $i$, we should instead choose condition $|\frac{g(\theta _{i})}{\theta _i^m}|>\sqrt[r]{d}$, $i=1,\dots ,n$.
\end{proof}

\begin{lemma}
\label{thm-1dk}Let $f,g\in \mathbb{Z}[X]$ with $f(0)\neq 0$ be relatively prime polynomials, and assume that $f$ has no nonconstant factors of degree less than a positive integer $r<n$. If for some positive integer $k$ we have $|g(\theta )|>\sqrt[r]{d_k}$ for each root $\theta $ of $f$, or we have $|\frac{g(\theta )}{\theta ^{\deg g}}|>\sqrt[r]{d_k}$ for each root $\theta $ of $f$, then $f$ is a product of at most $k$ irreducible factors over $\mathbb{Q}$.
\end{lemma}
\begin{proof}
We will use the same notations as in the proof of Lemma \ref{thm-1}. Let us assume that $f$ is a product of more than $k$ irreducible factors over $\mathbb{Q}$, say
\[
f(X)=f_1(X)\cdots f_\ell(X)
\]
with $\ell>k$ and $f_i\in\mathbb{Z}[X]$, $f_i$ irreducible over $\mathbb{Q}$ for $i=1,\dots ,\ell$. Then we have
\[
Res(f,g)=Res(f_1,g)\cdots Res(f_\ell,g),
\]
with $|Res(f_1,g)|,\dots ,|Res(f_\ell,g)|$ divisors of $|Res(f,g)|$, of which at least one must not exceed $\sqrt[\ell]{|Res(f,g)|}$, say $|Res(f_1,g)|\leq\sqrt[\ell]{|Res(f,g)|}$. In particular, we have
$|Res(f_1,g)|\leq d_{\ell-1}$, so this time instead of (\ref{InegRez}) we obtain
\begin{equation}\label{InegRezNou}
|Res(f_{1},g)|=|c_{t}|^{m}|g(\theta _{1})\cdots
g(\theta _{t})|\leq d_{\ell-1}.
\end{equation}
On the other hand, since $|c_{t}|\geq 1$, if we assume that each one of $|g(\theta _1)|,\dots ,|g(\theta _t)|$ exceeds $\sqrt[r]{d_k}$, we must actually have 
\[
|c_{t}|^{m}|g(\theta _{1})\cdots
g(\theta _{t})|>(\sqrt[r]{d_k})^t\geq d_k\geq d_{\ell-1},
\]
as $t\geq r$, $\ell-1\geq k$, and $d_i$ is a decreasing sequence.
This contradicts (\ref{InegRezNou}) and completes the proof in this case. 

Next, as $f(0)\neq 0$, we may use again the inverses of $\theta _i$, to deduce this time that
\begin{equation}\label{InegRezReciprocaldk}
|Res(f_{1},g)|=|c_{0}|^{m}\left|\frac{g(\theta _{1})}{\theta _1^m}\right|\cdots \left|\frac{g(\theta _{t})}{\theta _t^m}\right|\leq d_{\ell-1},
\end{equation}
so if we assume that each one of $|\frac{g(\theta _{1})}{\theta _1^m}|,\dots ,|\frac{g(\theta _{n})}{\theta _n^m}|$ exceeds $\sqrt[r]{d_k}$ and use the fact that
$|c_{0}|\geq 1$, we deduce again that $|Res(f_{1},g)|>(\sqrt[r]{d_k})^t\geq d_{\ell-1}$, which contradicts (\ref{InegRezReciprocaldk}). This completes the proof of the lemma. 
\end{proof}

When no information on the degrees of the factors of $f$ is available, one may still consider the cases $r=1$ and $r=2$, and use instead the following simpler instances of Lemma \ref{thm-1} and Lemma \ref{thm-1dk}:

\begin{lemma}
\label{corobanalthm-1}Let $f,g\in \mathbb{Z}[X]$ with $f(0)\neq 0$ be relatively prime polynomials, and let $d>0$ be a divisor of $Res(f,g)$. If $|g(\theta )|>d$ for each root $\theta$ of $f$, or if $|\frac{g(\theta )}{\theta ^{\deg g}}|>d$ for each root $\theta$ of $f$, then $f$ is a product of at most $\Omega (\frac{Res(f,g)}{d})$ irreducible factors over $\mathbb{Q}$.
Moreover, if $f$ has no rational roots, the same conclusion holds if $|g(\theta )|>\sqrt{d}$ for each root $\theta$ of $f$, or if $|\frac{g(\theta )}{\theta ^{\deg g}}|>\sqrt{d}$ for each root $\theta$ of $f$.
\end{lemma}

\begin{lemma}
\label{corobanalthm-1dk} Let $f,g\in \mathbb{Z}[X]$ with $f(0)\neq 0$ be relatively prime polynomials. If for some positive integer $k$ we have $|g(\theta )|>d_k$ for each root $\theta $ of $f$, or we have $|\frac{g(\theta )}{\theta ^{\deg g}}|>d_k$ for each root $\theta $ of $f$, then $f$ is a product of at most $k$ irreducible factors over $\mathbb{Q}$. Moreover, if $f$ has no rational roots, the same conclusion holds if $|g(\theta )|>\sqrt{d_k}$ for each root $\theta $ of $f$, or if $|\frac{g(\theta )}{\theta ^{\deg g}}|>\sqrt{d_k}$ for each root $\theta $ of $f$.
\end{lemma}

Lemma \ref{corobanalthm-1} and Lemma \ref{corobanalthm-1dk} have various corollaries, of which we will only state the following two irreducibility criteria. 

\begin{corollary}
\label{corothm-1}Let $f,g\in \mathbb{Z}[X]$ with $f(0)\neq 0$, and suppose that $|Res(f,g)|=p\cdot q$, where $p$ is a prime number and $q$ is a positive integer. 

i) If $|g(\theta )|>q$ for each root $\theta $ of $f$, or if $|\frac{g(\theta )}{\theta ^{\deg g}}|>q$ for each root $\theta $ of $f$, then $f$ is irreducible over $\mathbb{Q}$. 

ii) The same conclusion holds if $|g(\theta )|>\sqrt{q}$ for each root $\theta $ of $f$, or if $|\frac{g(\theta )}{\theta ^{\deg g}}|>\sqrt{q}$ for each root $\theta $ of $f$, provided that $f$ has no rational roots.
\end{corollary}

\begin{corollary}
\label{corothm-1d1} Let $f,g\in \mathbb{Z}[X]$ be two relatively prime polynomials, with $f(0)\neq 0$. 

i) If $|g(\theta )|>d_1$ for each root $\theta $ of $f$, or if $|\frac{g(\theta )}{\theta ^{\deg g}}|>d_1$ for each root $\theta $ of $f$, then $f$ is irreducible over $\mathbb{Q}$.

ii) The same conclusion holds if $|g(\theta )|>\sqrt{d_1}$ for each root $\theta $ of $f$, or if $|\frac{g(\theta )}{\theta ^{\deg g}}|>\sqrt{d_1}$ for each root $\theta $ of $f$, provided that $f$ has no rational roots.
\end{corollary}

\begin{proposition}\label{ADouaEchivalenta}
Corollary \ref{corothm-1} i) and Corollary \ref{corothm-1d1} i) are equivalent if $n\geq 3$, and Corollary \ref{corothm-1} ii) and Corollary \ref{corothm-1d1} ii) are equivalent if $n\geq 6$.
\end{proposition}

We may now proceed with the proofs of Theorem \ref{thm0} and Theorem \ref{thm0dk}.
\medskip

{\it Proof of Theorem \ref{thm0}.} With the notations used so far, we have
\[
g(\theta _i)=b_{m}(\theta _i-\xi _{1})\cdots (\theta _i-\xi _{m}),\qquad i=1,\dots ,n,
\]
so if we first assume that $\min_{i,j}|\theta _i-\xi _j|>d^{1/\min (m,n)}$, we obtain
\begin{equation}\label{gmaimarecad1}
|g(\theta _i)|\geq |(\theta _i-\xi _{1})\cdots (\theta _i-\xi _{m})|>(d^{1/\min (m,n)})^m\geq d,\qquad i=1,\dots ,n.
\end{equation}
Since $f(0)g(0)\neq 0$, we may also write $g(\theta _i)=b_{0}\frac{(\theta _i-\xi _{1})}{\xi _1}\cdots \frac{(\theta _i-\xi _{m})}{\xi _m}$, and thus
\[
\frac{g(\theta _i)}{\theta _i^{m}}=b_{0}\left(\frac{1}{\xi _1}-\frac{1}{\theta _i}\right)\cdots \left(\frac{1}{\xi _m}-\frac{1}{\theta _i}\right),\qquad i=1,\dots ,n,
\]
so if we assume this time that $\min_{i,j}|\theta _i^{-1}-\xi _j^{-1}|>d^{1/\min (m,n)}$, we obtain 
\begin{equation}\label{gmaimarecad2}
\left|\frac{g(\theta _i)}{\theta _i^{m}}\right|\geq \left|\frac{1}{\theta _i}-\frac{1}{\xi _{1}}\right|\cdots \left|\frac{1}{\theta _i}-\frac{1}{\xi _{m}}\right|>(d^{1/\min (m,n)})^m\geq d,\qquad i=1,\dots ,n.
\end{equation}
In view of (\ref{gmaimarecad1}) and (\ref{gmaimarecad2}) we conclude by Lemma \ref{thm-1} with $r=1$ that $f$ can be expressed as a product of at most $\Omega (Res(f,g)/d)$ irreducible factors over $\mathbb{Q}$. Now, since we have $|Res(f,g)|=|Res(g,f)|$, the hypotheses in the statement of our theorem are symmetric with respect to $f$ and $g$, so $g$ too can be expressed as a product of at most $\Omega (Res(f,g)/d)$ irreducible factors over $\mathbb{Q}$. This completes the proof. \hfill $\square $
\medskip 

{\it Proof of Theorem \ref{thm0dk}.} Using the same notations, we have
\[
g(\theta _i)=b_{m}(\theta _i-\xi _{1})\cdots (\theta _i-\xi _{m}),\qquad i=1,\dots ,n,
\]
so if $\min_{i,j}|\theta _{i}-\xi _{j}|>d_k^{1/\min (m,n)}$, we deduce this time that
\begin{equation}\label{gmaimarecadk1}
|g(\theta _i)|\geq |(\theta _i-\xi _{1})\cdots (\theta _i-\xi _{m})|>(d_k^{1/\min (m,n)})^m\geq d_k,\qquad i=1,\dots ,n.
\end{equation}
As $f(0)\neq 0$, $g(0)\neq 0$, and $\min_{i,j}|\frac{1}{\theta _{i}}-\frac{1}{\xi _{j}}|>d_k^{1/\min (m,n)}$, instead of (\ref{gmaimarecad2}) we obtain here
\begin{equation}\label{gmaimarecadk2}
\left|\frac{g(\theta _i)}{\theta _i^{m}}\right|\geq \left|\frac{1}{\theta _i}-\frac{1}{\xi _{1}}\right|\cdots \left|\frac{1}{\theta _i}-\frac{1}{\xi _{m}}\right|>(d_k^{1/\min (m,n)})^m\geq d_k,\qquad i=1,\dots ,n.
\end{equation}
By Lemma \ref{thm-1dk} with $r=1$ we conclude in view of (\ref{gmaimarecadk1}) and (\ref{gmaimarecadk2}) that $f$ can be expressed as a product of at most $k$ irreducible factors over $\mathbb{Q}$. Observe now that due to the fact that $|Res(f,g)|=|Res(g,f)|$, the sequence $d_i$ remains the same if we interchange the roles of $f$ and $g$, so the hypotheses in the statement of our theorem are symmetric with respect to $f$ and $g$. Therefore $g$ too can be expressed as a product of at most $k$ irreducible factors over $\mathbb{Q}$, and this completes the proof of the theorem. \hfill $\square $

\medskip

{\it Proof of Proposition \ref{EquivalentConditions}.}\ i) Let us denote by $a_n$ and $b_m$ the leading coefficients of $f$ and $g$, respectively, and by $a_0$ and $b_0$ their constant terms, respectively. 
Assume that $f$ and $g$ satisfy the hypotheses in Theorem \ref{thm0} with $d=d_k$. The condition that $|\theta-\xi|>d_k^{1/\min (m,n)}$ for every root $\theta $ of $f$ and every root $\xi $ of $g$ shows that
\[
|Res(f,g)|=|a_n|^m|b_m|^n\prod\limits_{i=1}^n\prod\limits _{j=1}^m|\theta _i-\xi_j|>|a_n|^m|b_m|^nd_k^{mn/\min (m,n)}\geq d_k^{\max (m,n)},
\]
while condition that $|\frac{1}{\theta}-\frac{1}{\xi}|>d_k^{1/\min (m,n)}$ for every root $\theta $ of $f$ and every root $\xi $ of $g$ together with Vieta's formulas implies that
\[
|Res(f,g)|=|a_0|^m|b_0|^n\prod\limits_{i=1}^n\prod\limits _{j=1}^m\Big|\frac{1}{\theta _i}-\frac{1}{\xi_j}\Big|>|a_0|^m|b_0|^nd_k^{mn/\min (m,n)}\geq d_k^{\max (m,n)},
\]
so each one of these conditions implies the inequality
\begin{equation}\label{dkmic}
d_k<|Res(f,g)|^{\frac{1}{\max (m,n)}}.
\end{equation} 
Note that by (\ref{dkmic}), $|Res(f,g)|$ must be greater than $1$. We will prove now that $\Omega(\frac{|Res(f,g)|}{d_k})\leq k$. Let us assume to the contrary that $\frac{|Res(f,g)|}{d_k}$ has at least $k+1$ prime factors, counting multiplicities, and let $p$ be its least prime factor. Then $p^{k+1}\leq \frac{|Res(f,g)|}{d_k}$, so $p\sqrt[k+1]{d_k}\leq \sqrt[k+1]{|Res(f,g)|}$, which by the definition of $d_k$ shows that $p\sqrt[k+1]{d_k}\leq d_k$, that is $p\leq d_k^{\frac{k}{k+1}}$. In view of (\ref{dkmic}) this further implies that
\begin{equation}\label{pmic}
p<|Res(f,g)|^{\frac{k}{(k+1)\max (m,n)}}.
\end{equation} 
On the other hand, using again the definition of $d_k$, since $pd_k$ is a divisor of $|Res(f,g)|$, it must be larger than $\sqrt[k+1]{|Res(f,g)|}$, which by (\ref{dkmic}) leads to $\sqrt[k+1]{|Res(f,g)|}<pd_k\leq p|Res(f,g)|^{1/\max (m,n)}$, so we must have
\begin{equation}\label{pmare}
p>|Res(f,g)|^{\frac{1}{k+1}-\frac{1}{\max (m,n)}}.
\end{equation}
Combining now (\ref{pmic}) and (\ref{pmare}) we conclude that $\frac{1}{k+1}-\frac{1}{\max (m,n)}<\frac{k}{(k+1)\max (m,n)}$, which is equivalent to $\max (m,n)<2k+1$. We therefore reach a contradiction for $k\leq \frac{\max (m,n)-1}{2}$, which will force $\Omega(\frac{|Res(f,g)|}{d_k})$ to be at most $k$, as claimed. Thus each one of $f$ and $g$ is a product of at most $k$ irreducible factors over $\mathbb{Q}$.

ii) The hypotheses in Corollary \ref{coro1thm0} that $|Res(f,g)|=pq$ and that $|\theta-\xi|>q^{1/\min (m,n)}$ for every root $\theta $ of $f$ and every root $\xi $ of $g$, together imply that 
\[
pq=|a_n|^m|b_m|^n\prod\limits_{i=1}^n\prod\limits _{j=1}^m|\theta _i-\xi_j|>|a_n|^m|b_m|^nq^{mn/\min (m,n)}\geq q^{\max (m,n)},
\]
that is $p>|a_n|^m|b_m|^nq^{\max (m,n)-1}$. On the other hand, if we assume that $|\frac{1}{\theta}-\frac{1}{\xi}|>q^{1/\min (m,n)}$ for every root $\theta $ of $f$ and every root $\xi $ of $g$ and use Vieta's formulas, we see that 
\[
pq=|a_0|^m|b_0|^n\prod\limits_{i=1}^n\prod\limits _{j=1}^m\Big|\frac{1}{\theta _i}-\frac{1}{\xi_j}\Big|>|a_0|^m|b_0|^nq^{mn/\min (m,n)}\geq q^{\max (m,n)},
\]
implying that $p>|a_0|^m|b_0|^nq^{\max (m,n)-1}$. Thus, unless $m=n=1$, in both cases we may conclude that $p>q$, so $p\nmid q$, which implies that
\[
d_1:= \max\{d:d\ {\rm divides}\ |Res(f,g)|\ {\rm and}\ d\leq \sqrt{|Res(f,g)|}\thinspace \}=q
\]
and shows that $f$ and $g$ satisfy the hypotheses in Corollary \ref{corothm0d1} with $d_1=q$.

Conversely, assume that $f$ and $g$ satisfy the hypotheses in Corollary \ref{corothm0d1}. Reasoning as above with $d_1$ instead of $d_k$, we deduce that $\Omega(\frac{|Res(f,g)|}{d_1})=1$ if $1\leq \frac{\max (m,n)-1}{2}$, that is $\frac{|Res(f,g)|}{d_1}$ is a prime number, say $p$, if $\max (m,n)\geq 3$. Denoting $d_1$ by $q$, we see that $|Res(f,g)|=pq$ with $p$ prime, so $f$ and $g$ satisfy the hypotheses in Corollary \ref{coro1thm0}. This completes the proof.
\hfill $\square $
\medskip

{\it Proof of Proposition \ref{ADouaEchivalenta}:}\ We will use the same notations as in the proof of Proposition \ref{EquivalentConditions}.
The condition in Corollary \ref{corothm-1} i) that $|g(\theta )|>q$ for each root $\theta $ of $f$ implies that $p>q^{n-1}|a_n|^m$, while the condition that $|\frac{g(\theta )}{\theta ^{\deg g}}|>q$ for each root $\theta $ of $f$ implies that $p>q^{n-1}|a_0|^m$, so for $n\geq 2$ we must have $p>q$, which in particular implies that $d_1=q$.

Conversely, if we assume that $|g(\theta )|>d_1$ for each root $\theta $ of $f$, or that $|\frac{g(\theta )}{\theta ^{\deg g}}|>d_1$ for each root $\theta $ of $f$, we deduce that $|Res(f,g)|=|a_n|^m\prod_{i=1}^n|g(\theta _i)|=|a_0|^m\prod_{i=1}^n|\frac{g(\theta_i)}{\theta_i^{m}}|>d_1^n$, so instead of (\ref{dkmic}) we obtain 
$d_1<|Res(f,g)|^{\frac{1}{n}}$. The fact that $\frac{|Res(f,g)|}{d_1}$ is a prime number follows in the same way as in the proof of Proposition \ref{EquivalentConditions} with $k=1$ and $n$ instead of $\max (m,n)$, so the desired equivalence holds for $n\geq 3$.

If we know that $f$ has no rational roots, the hypotheses in Corollary \ref{corothm-1} ii) that $|g(\theta )|>\sqrt{q}$ for each root $\theta $ of $f$ implies that $p>q^{\frac{n}{2}-1}|a_n|^m$, while the condition that $|\frac{g(\theta )}{\theta ^{m}}|>\sqrt{q}$ for each root $\theta $ of $f$ implies that $p>q^{\frac{n}{2}-1}|a_0|^m$, so for $n\geq 4$ we have $p>q$, so that $d_1=q$.

Conversely, if $|g(\theta )|>\sqrt{d_1}$ for each root $\theta $ of $f$, or if $|\frac{g(\theta )}{\theta ^{m}}|>\sqrt{d_1}$ for each root $\theta $ of $f$, instead of (\ref{dkmic}) with $k=1$ we obtain $d_1<|Res(f,g)|^{\frac{2}{n}}$. The fact that $\frac{|Res(f,g)|}{d_1}$ is a prime number follows as in the proof of Proposition \ref{EquivalentConditions} with $k=1$ and $\frac{n}{2}$ instead of $\max (m,n)$, which shows that in this case the equivalence holds for $n\geq 6$.
\hfill $\square $
\medskip

One way to be sure that the resultant of two polynomials $f$ and $g$ is divisible by a prime number $p$, without explicitly computing the resultant, is to find an integer $d$ such that both $f(d)$ and $g(d)$ are divisible by $p$. If so, by using Hadamard's inequality to bound the resultant of $f$ and $g$, for instance, one may find
some suitable conditions on their roots that guarantee the irreducibility of both $f$ and $g$, as in the following result. For other bounds on resultants we refer the interested reader to \cite{BisLif}, where bounds for formal resultants are also derived.

\begin{theorem}
\label{thm10,5} Let $f(X)=\sum_{i=0}^{n}a_{i}X^{i},\ g(X)=\sum_{i=0}^{m}b_{i}X^{i}\in \mathbb{Z}[X]$ with $a_{0}a_{n}b_{0}b_{m}\neq 0$, and assume that for an integer $d$, $f(d)$ and $g(d)$ are divisible by a prime number $p$. Let
\[
k=\left( \frac{||f||^m||g||^n}{p}  \right) ^\frac{1}{\rm{min}(m,n)},
\]
with $||f||=\sqrt{a_{0}^{2}+\cdots +a_{n}^{2}}$ and $||g||=\sqrt{b_{0}^{2}+\cdots +b_{m}^{2}}$.
If $|\theta-\xi|>k$ for
every root $\theta $ of $f$ and every root $\xi $ of $g$, or if $|\frac{1}{\theta}-\frac{1}{\xi }|>k$ for every root $\theta $ of $f$ and every root $\xi $ of $g$, then $f$ and $g$ are both irreducible over $\mathbb{Q}$.
\end{theorem}

\begin{proof}
\ Given two polynomials with integer coefficients $f$ and $g$, there exist two polynomials $A, B$ with integer coefficients too, $\deg A<\deg g$, $\deg B <\deg f$, and such that 
\[
A(X)f(X)+B(X)g(X)=Res(f,g).
\]
Taking $X=d$ in this equality, and using the fact that $f(d)$ and $g(d)$ are divisible by the same prime number $p$, we deduce that $Res(f,g)$ too must be a multiple of $p$, say $|Res(f,g)|=pq$ for some integer $q$. Our assumption that 
\begin{equation}\label{differ}
|\theta _{i}-\xi_{j}|>\left( \frac{||f||^m||g||^n}{p}  \right) ^\frac{1}{\rm{min}(m,n)}
\end{equation}
for each root $\theta_i $ of $f$ and each root $\xi_j $ of $g$, shows in particular  that the polynomials $f$ and $g$ can have no roots in common, so they must be relatively prime, and hence 
$q\neq 0$. On the other hand, since by Hadamard's inequality we have $|Res(f,g)|\leq ||f||^m||g||^n$, inequality (\ref{differ}) yields
\[
|\theta _{i}-\xi_{j}|>\left( \frac{|Res(f,g)|}{p}  \right) ^\frac{1}{\rm{min}(m,n)}=q^\frac{1}{\rm{min}(m,n)}.
\]
One may apply now Corollary \ref{coro1thm0} to conclude that the polynomials $f$ and $g$ are both irreducible. The case when one considers the inverses of the roots of $f$ and $g$ follows in a similar way.  
\end{proof}

The remaining results in this section will consider the case when the roots of $f$ are close to those of $g$. Recall that for a fixed pair of relatively prime polynomials $f,g\in \mathbb{Z}[X]$, and each positive integer $k\geq 2$ we defined
\[
d_k:=\max\{d: d\ \text{divides}\ |Res(f,g)|\ \text{and} \  d\leq \sqrt[k+1]{|Res(f,g)|}\}.
\]

Our first result in this respect is:

\begin{lemma}
\label{LemaGeneralaCuRez/dk}Let $f,g\in \mathbb{Z}[X]$ be relatively prime polynomials of degrees $n$ and $m$, respectively, and suppose that $f$ factorizes as $f(X)=a_n(X-\theta _{1})\cdots(X-\theta _{n})$ with $\theta_1\cdots \theta _n\neq 0$, and that $f$ has no nonconstant factors of degree less than a positive integer $r<n$. If for some positive integer $k<\frac{n}{r}$ we have 
\[
\max\limits _{1\leq i\leq n}|g(\theta _i)|<\left(\frac{|Res(f,g)|}{d_k|a_n|^m}\right)^{\frac{1}{n-kr}}\ \ or\ \ \ \max\limits _{1\leq i\leq n}
\left|\frac{g(\theta _i)}{\theta_i^m}\right|<\left(\frac{|Res(f,g)|}{d_k|a_0|^m}\right)^{\frac{1}{n-kr}},  
\]
then $f$ is a product of at most $k$ irreducible factors over $\mathbb{Q}$.
\end{lemma}

\begin{proof}
\ First of all we will prove that our assumption on the magnitude of $|g(\theta _i)|$ implies that $|Res(f,g)|>d_k|a_n|^m$. Indeed, since $|g(\theta _i)|<(\frac{|Res(f,g)|}{d_k|a_n|^m})^{\frac{1}{n-kr}}$ for each $i=1,\dots ,n$, we must have
\[
|Res(f,g)|<|a_n|^{m}\left(\frac{|Res(f,g)|}{d_k|a_n|^m}\right)^{\frac{n}{n-kr}}=\frac{|Res(f,g)|}{d_k}\left(\frac{|Res(f,g)|}{d_k|a_n|^m}\right)^{\frac{kr}{n-kr}},
\]
which yields
\begin{equation}\label{geqanm}
\frac{|Res(f,g)|}{d_k|a_n|^m}>d_k^{\frac{n-kr}{kr}}\geq 1,
\end{equation}
as claimed. Let us assume now that one may write $f$ as a product of more than $k$ irreducible factors over $\mathbb{Q}$, say
\[
f(X)=f_1(X)\cdots f_\ell(X)
\]
with $\ell>k$ and $f_i\in\mathbb{Z}[X]$, $f_i$ irreducible over $\mathbb{Q}$ for $i=1,\dots ,\ell$. Note that this implies in particular that $\ell\geq 2$, and since $\deg f_i\geq r$ for $i=1,\dots ,\ell$ and $\deg f_1+\cdots +\deg f_{\ell}=n$, we must have 
\begin{equation}\label{tmax}
\max\{\deg f_1,\dots ,\deg f_{\ell} \}\leq n-r(\ell-1).
\end{equation}
We also have $Res(f,g)=Res(f_1,g)\cdots Res(f_\ell,g)$, and since $|Res(f_1,g)|,\dots ,|Res(f_\ell,g)|$ are divisors of $|Res(f,g)|$, at least one of them, say $|Res(f_1,g)|$, must satisfy the inequality $|Res(f_1,g)|\leq\sqrt[\ell]{|Res(f,g)|}$.  This shows that we actually have $|Res(f_1,g)|\leq d_{\ell-1}$, which, if we denote $f_2(X)\cdots f_{\ell}(X)$ by $h(X)$, implies that
\[
|Res(h,g)|\geq\frac{|Res(f,g)|}{d_{\ell-1}},
\]
or, equivalently, if we assume that $h$ factorizes as $h(X)=c_t(X-\theta _1)\cdots (X-\theta _t)$,
\begin{equation}\label{InegRez(f1,g)Corectata}
|c_{t}|^{m}|g(\theta _{1})\cdots
g(\theta _{t})|\geq \frac{|Res(f,g)|}{d_{\ell-1}}.
\end{equation}
On the other hand, since $c_{t}\mid a_n$ and by our hypothesis each one of $|g(\theta _1)|,\dots ,|g(\theta _n)|$ is smaller than $(\frac{|Res(f,g)|}{d_k|a_n|^m})^{\frac{1}{n-kr}}$, we must actually have 
\begin{eqnarray*}
|c_{t}|^{m}|g(\theta _{1})\cdots
g(\theta _{t})| & < & |a_n|^m\biggl(\frac{|Res(f,g)|}{d_k|a_n|^m}\biggr)^{\frac{t}{n-kr}}\leq |a_n|^m\biggl(\frac{|Res(f,g)|}{d_k|a_n|^m}\biggr)^{\frac{n-r(\ell-1)}{n-kr}}\\
& \leq & \frac{|Res(f,g)|}{d_k}\leq \frac{|Res(f,g)|}{d_{\ell-1}},
\end{eqnarray*}
as $t\leq n-r(\ell-1)\leq n-kr$ (in view of (\ref{tmax})), $|Res(f,g)|>d_k|a_n|^m$ (by (\ref{geqanm})) and $k\leq \ell -1$ (implying that $d_k\geq d_{\ell -1}$, as $d_i$ is a decreasing sequence).
This contradicts (\ref{InegRez(f1,g)Corectata}) and completes the proof in the first case. 

Now let us prove that if $|g(\theta _i)|<|\theta _i^m|(\frac{|Res(f,g)|}{d_k|a_0|^m})^{\frac{1}{n-kr}}$ for each $i=1,\dots ,n$, then we must have $|Res(f,g)|>d_k|a_0|^m$. Indeed, our assumption implies that
\[
|Res(f,g)|<|a_n|^{m}|\theta _1^m|\cdots |\theta _n^m|\left(\frac{|Res(f,g)|}{d_k|a_0|^m}\right)^{\frac{n}{n-kr}}=\frac{|Res(f,g)|}{d_k}\left(\frac{|Res(f,g)|}{d_k|a_0|^m}\right)^{\frac{kr}{n-kr}},
\]
which leads to
\begin{equation}\label{geqa0m}
\frac{|Res(f,g)|}{d_k|a_0|^m}>d_k^{\frac{n-kr}{kr}}\geq 1,
\end{equation}
as claimed. 

Note that since $f(0)\neq 0$, we may also write $|c_{t}|^{m}|g(\theta _{1})\cdots
g(\theta _{t})|=|c_{0}|^{m}|\frac{g(\theta _{1})}{\theta _1^m}|\cdots |\frac{g(\theta _{t})}{\theta _t^m}|$, with $c_0\neq 0$ the free term of $h$, so if we assume that each one of $|\frac{g(\theta _{1})}{\theta _1^m}|,\dots ,|\frac{g(\theta _{n})}{\theta _n^m}|$ is smaller than $(\frac{|Res(f,g)|}{d_k|a_0|^m})^{\frac{1}{n-kr}}$, using (\ref{tmax}) and (\ref{geqa0m}) we obtain this time
\begin{eqnarray*}
|c_{t}|^{m}|g(\theta _{1})\cdots
g(\theta _{t})| & = & |c_{0}|^{m}\left|\frac{g(\theta _{1})}{\theta _1^m}\right|\cdots \left|\frac{g(\theta _{t})}{\theta _t^m}\right|  <  |a_0|^m\biggl(\frac{|Res(f,g)|}{d_k|a_0|^m}\biggr)^{\frac{t}{n-kr}}\\ 
& \leq & |a_0|^m\biggl(\frac{|Res(f,g)|}{d_k|a_0|^m}\biggr)^{\frac{n-r(\ell-1)}{n-kr}}
 \leq  \frac{|Res(f,g)|}{d_k}\leq \frac{|Res(f,g)|}{d_{\ell-1}},
\end{eqnarray*}
which contradicts (\ref{InegRez(f1,g)Corectata}) and completes the proof.
\end{proof}

When no information on the degrees of the irreducible factors of $f$ is available, one may use instead the following result for the cases that $r=1$ and $r=2$:

\begin{lemma}
\label{LemaParticularaCuRez/dk}Let $f,g\in \mathbb{Z}[X]$ be relatively prime polynomials of degrees $n$ and $m$, respectively, and suppose that $f$ factorizes as $f(X)=a_n(X-\theta _{1})\cdots(X-\theta _{n})$ with $\theta_1\cdots \theta _n\neq 0$. If for some positive integer $k<n$ we have 
\[
\max\limits _{1\leq i\leq n}|g(\theta _i)|<\left(\frac{|Res(f,g)|}{d_k|a_n|^m}\right)^{\frac{1}{n-k}}\ \ or\ \ \ \max\limits _{1\leq i\leq n}
\left|\frac{g(\theta _i)}{\theta_i^m}\right|<\left(\frac{|Res(f,g)|}{d_k|a_0|^m}\right)^{\frac{1}{n-k}},  
\]
then $f$ is a product of at most $k$ irreducible factors over $\mathbb{Q}$. 
The same conclusion holds with $k<\frac{n}{2}$ and exponents $\frac{1}{n-2k}$ instead of $\frac{1}{n-k}$, provided that $f$ has no rational roots.
\end{lemma}

In particular, from Lemma \ref{LemaGeneralaCuRez/dk} and Lemma \ref{LemaParticularaCuRez/dk} we obtain the following irreducibility criteria.
\begin{corollary}
\label{LemaGeneralaIreductibilitateCuRez/dk}Let $f,g\in \mathbb{Z}[X]$ be relatively prime polynomials of degrees $n$ and $m$, respectively, and suppose that $f$ factorizes as $f(X)=a_n(X-\theta _{1})\cdots(X-\theta _{n})$ with $\theta_1\cdots \theta _n\neq 0$, and that $f$ has no nonconstant factors of degree less than a positive integer $r<n$.  If  
$\max\limits _{1\leq i\leq n}|g(\theta _i)|<(\frac{|Res(f,g)|}{d_1|a_n|^m})^{\frac{1}{n-r}}$ or $\max\limits _{1\leq i\leq n}|\frac{g(\theta _i)}{\theta_i^m}|<(\frac{|Res(f,g)|}{d_1|a_0|^m})^{\frac{1}{n-r}}$, then $f$ is irreducible over $\mathbb{Q}$. 
\end{corollary}
\begin{corollary}
\label{LemaParticularaIreductibilitateCuRez/dk}Let $f,g\in \mathbb{Z}[X]$ be relatively prime polynomials of degrees $n$ and $m$, respectively, and suppose that $f$ factorizes as $f(X)=a_n(X-\theta _{1})\cdots(X-\theta _{n})$ with $\theta_1\cdots \theta _n\neq 0$. If  
$\max\limits _{1\leq i\leq n}|g(\theta _i)|<(\frac{|Res(f,g)|}{d_1|a_n|^m})^{\frac{1}{n-1}}$ or $\max\limits _{1\leq i\leq n}|\frac{g(\theta _i)}{\theta_i^m}|<(\frac{|Res(f,g)|}{d_1|a_0|^m})^{\frac{1}{n-1}}$, then $f$ is irreducible over $\mathbb{Q}$. The same conclusion holds with exponents $\frac{1}{n-2}$ instead of $\frac{1}{n-1}$, provided that $f$ has no rational roots.
\end{corollary}
The conditions on the magnitude of $g(\theta_i)$ in Corollary \ref{LemaGeneralaIreductibilitateCuRez/dk} and Corollary \ref{LemaParticularaIreductibilitateCuRez/dk} become more tractable if we know, for instance, that $Res(f,g)$ is divisible by a sufficiently large prime number, as one can see in the proof of the following result:
\begin{corollary}
\label{Coro1CuRez/dk}Let $f,g\in \mathbb{Z}[X]$ be relatively prime polynomials of degrees $n$ and $m$, respectively, and suppose that $f$ factorizes as $f(X)=a_n(X-\theta _{1})\cdots(X-\theta _{n})$ with $\theta_1\cdots \theta _n\neq 0$, and that $f$ has no nonconstant factors of degree less than a positive integer $r<n$. Assume that $|Res(f,g)|=pq$ with $p$ a prime number and $q$ a positive integer. If  
$\max\limits _{1\leq i\leq n}|g(\theta _i)|<(\frac{p}{|a_n|^m})^{\frac{1}{n-r}}$ or $\max\limits _{1\leq i\leq n}|\frac{g(\theta _i)}{\theta_i^m}|<(\frac{p}{|a_0|^m})^{\frac{1}{n-r}}$, then $f$ is irreducible over $\mathbb{Q}$. 
\end{corollary}
\begin{proof}
\ The irreducibility of $f$ follows trivially if $r>\frac{n}{2}$ without any assumption on the magnitude of $|g(\theta _i)|$, so we may assume that $r\leq\frac{n}{2}$. In the first case our assumptions that $|Res(f,g)|=pq$ and $|g(\theta _i)|<(\frac{p}{|a_n|^m})^{\frac{1}{n-r}}$ for $i=1,\dots ,n$ together imply that
\[
pq<|a_n|^{m}\left(\frac{p}{|a_n|^m}\right)^{\frac{n}{n-r}}=p\left(\frac{p}{|a_n|^m}\right)^{\frac{r}{n-r}},
\]
from which we deduce that $p>|a_n|^{m}q^{\frac{n-r}{r}}\geq q$,
as $|a_n|\geq 1$ and $r\leq\frac{n}{2}$. Thus $p>q$, so in this case $d_1=q$, and the irreducibility of $f$ follows by Lemma \ref{LemaGeneralaCuRez/dk} with $k=1$.

In the second case our assumptions that $|Res(f,g)|=pq$ and $|g(\theta _i)|<|\theta _i^m|(\frac{p}{|a_0|^m})^{\frac{1}{n-r}}$ for $i=1,\dots ,n$ will imply that
\[
pq<|a_n|^{m}|\theta _1^m|\cdots |\theta _n^m|\left(\frac{p}{|a_0|^m}\right)^{\frac{n}{n-r}}=|a_0|^m\left(\frac{p}{|a_0|^m}\right)^{\frac{n}{n-r}}
=p\left(\frac{p}{|a_0|^m}\right)^{\frac{r}{n-r}},
\]
from which we deduce that $p>|a_0|^{m}q^{\frac{n-r}{r}}\geq q$, as $|a_0|\geq 1$ and $r\leq\frac{n}{2}$. Thus $d_1=q$ in this case too, and the irreducibility of $f$ follows again by Lemma \ref{LemaGeneralaCuRez/dk} with $k=1$. 
\end{proof}

When no information on $r$ is known, instead of Corollary \ref{Coro1CuRez/dk} one may use the following simpler result:

\begin{corollary}
\label{Coro2CuRez/dk}Let $f,g\in \mathbb{Z}[X]$ be relatively prime polynomials of degrees $n$ and $m$, respectively, and suppose that $f$ factorizes as $f(X)=a_n(X-\theta _{1})\cdots(X-\theta _{n})$ with $\theta_1\cdots \theta _n\neq 0$, and that $|Res(f,g)|=pq$ with $p$ a prime number and $q$ a positive integer. If  
$\max\limits _{1\leq i\leq n}|g(\theta _i)|<(\frac{p}{|a_n|^m})^{\frac{1}{n-1}}$ or $\max\limits _{1\leq i\leq n}|\frac{g(\theta _i)}{\theta_i^m}|<(\frac{p}{|a_0|^m})^{\frac{1}{n-1}}$, then $f$ is irreducible over $\mathbb{Q}$. If $f$ has no rational roots, the same conclusion holds with exponents $\frac{1}{n-2}$ instead of $\frac{1}{n-1}$.
\end{corollary}
We mention that the particular case of Corollary \ref{Coro2CuRez/dk} when $g(X)=X-t$ for some integer $t$ was proved by Weisner \cite{Weisner} (see also Guersenzvaig \cite{Guersenzvaig}). We will end this section with the following irreducibility criterion for pairs of polynomials, whose proof relies on Corollary \ref{Coro2CuRez/dk}:
\begin{theorem}\label{TheoremPairCloseRoots}
Let $f,g\in \mathbb{Z}[X]$ be relatively prime polynomials, which factorize as $f(X)=a_n(X-\theta _{1})\cdots(X-\theta _{n})$ and $g(X)=b_m(X-\xi _{1})\cdots(X-\xi _{m})$, with $n\geq 2$, $m\geq 2$ and $\theta _{1}\cdots \theta _{n}\xi _{1}\cdots \xi _{m}\neq 0$. If $|Res(f,g)|=p\cdot q$, with $p$ a prime number and $q$ a positive integer, and 
\[
\max\limits_{i,j}|\theta_i-\xi_j|<\frac{p^{\frac{1}{mn-\min\{m,n\}}}}{|a_n|^{\frac{1}{n-1}}|b_m|^{\frac{1}{m-1}}}\quad or\quad \max\limits_{i,j}\left|\frac{1}{\theta_i}-\frac{1}{\xi_j}\right|<\frac{p^{\frac{1}{mn-\min\{m,n\}}}}{|a_0|^{\frac{1}{n-1}}|b_0|^{\frac{1}{m-1}}},
\]
then $f$ and $g$ are irreducible over $\mathbb{Q}$. If $n\geq 3$ and $m\geq 3$, the same conclusion holds if
\[
\max\limits_{i,j}|\theta_i-\xi_j|<\frac{p^{\frac{1}{mn-2\min\{m,n\}}}}{|a_n|^{\frac{1}{n-2}}|b_m|^{\frac{1}{m-2}}}\quad or\quad \max\limits_{i,j}\left|\frac{1}{\theta_i}-\frac{1}{\xi_j}\right|<\frac{p^{\frac{1}{mn-2\min\{m,n\}}}}{|a_0|^{\frac{1}{n-2}}|b_0|^{\frac{1}{m-2}}},
\]
provided that none of $f$ and $g$ has rational roots.
\end{theorem}
\begin{proof}
\ Let $\delta =\max\limits_{i,j}|\theta_i-\xi_j|$. For the condition $|g(\theta _i)|<(\frac{p}{|a_n|^m})^{\frac{1}{n-1}}$ to hold for $i=1,\dots ,n$, it suffices to have $|b_m|\delta ^m<(\frac{p}{|a_n|^m})^{\frac{1}{n-1}}$, or equivalently, to have
\begin{equation}\label{deltaslaba}
\delta<\frac{p^{\frac{1}{mn-m}}}{|a_n|^{\frac{1}{n-1}}|b_m|^{\frac{1}{m}}}.
\end{equation}
Thus, since $|Res(f,g)|=|Res(g,f)|$, if instead of (\ref{deltaslaba}) we ask
\[
\delta<\frac{p^{\frac{1}{mn-\min\{m,n\}}}}{|a_n|^{\frac{1}{n-1}}|b_m|^{\frac{1}{m-1}}},
\]
which is symmetric with respect to $f$ and $g$, we deduce by Corollary \ref{Coro2CuRez/dk} that both
$f$ and $g$ must be irreducible over $\mathbb{Q}$. 

Let now $\delta' =\max\limits_{i,j}|\frac{1}{\theta_i}-\frac{1}{\xi_j}|$, and observe that $|\frac{g(\theta _i)}{\theta _i^m}|=|b_0|\cdot |\frac{1}{\theta_i}-\frac{1}{\xi_1}|\cdots |\frac{1}{\theta_i}-\frac{1}{\xi_m}|$, so for the inequality $\max\limits_{1\leq i\leq n}|\frac{g(\theta _i)}{\theta _i^m}|<(\frac{p}{|a_0|^m})^{\frac{1}{n-1}}$ to hold, it suffices to have
$|b_0|\delta '^m<(\frac{p}{|a_0|^m})^{\frac{1}{n-1}}$, or equivalently, to have
\begin{equation}\label{delta'slaba}
\delta'<\frac{p^{\frac{1}{mn-m}}}{|a_0|^{\frac{1}{n-1}}|b_0|^{\frac{1}{m}}}.
\end{equation}
As before, since $|Res(f,g)|=|Res(g,f)|$, if instead of (\ref{delta'slaba}) we ask
\[
\delta'<\frac{p^{\frac{1}{mn-\min\{m,n\}}}}{|a_0|^{\frac{1}{n-1}}|b_0|^{\frac{1}{m-1}}},
\]
which is symmetric with respect to $f$ and $g$, we deduce by Corollary \ref{Coro2CuRez/dk} that both $f$ and $g$ must be irreducible. 
The case that none of $f$ and $g$ has rational roots follows in a similar way. 
\end{proof}

\section{Criteria that rely on the magnitude of the coefficients of $f$ and $g$}\label{magnitudeoffandg}

For the results in previous section to become effective, we have to consider pairs $f,g$ for which some information on their coefficients allows one to find information on the location of their roots. Such effective results are the following instances of Corollary \ref{coro1thm0}, that consider the case when the roots of $f$ are separated by the roots of $g$ by an annular region in the complex plane.

\begin{theorem}
\label{thm6Generala} Let $f(X)=\sum_{i=0}^{n}a_{i}X^{i},\
g(X)=\sum_{i=0}^{m}b_{i}X^{i}\in \mathbb{Z}[X]$, $a_{0}a_{n}b_{0}b_{m}\neq 0$. If $|Res(f,g)|=pq$ with $p$ a prime number and $q$ a positive integer, and for two positive real numbers $A$ and $B$ with $B\geq A+q^{1/\min\{m,n\}}$ we
have $|a_{n}|>\sum_{i=0}^{n-1}|a_{i}|A^{i-n}$ and $|b_{0}|>%
\sum_{i=1}^{m}|b_{i}|B^{i}$, then $f$ and $g$ are both
irreducible over $\mathbb{Q}$.
\end{theorem}
In particular, for $q=1$ and $B=A+1$ we have the following irreducibility criterion.
\begin{corollary}
\label{thm6} Let $f(X)=\sum_{i=0}^{n}a_{i}X^{i},\
g(X)=\sum_{i=0}^{m}b_{i}X^{i}\in \mathbb{Z}[X]$, $a_{0}a_{n}b_{0}b_{m}\neq 0$%
. If $|Res(f,g)|$ is a prime number and for a positive real $A$ we
have $|a_{n}|>\sum_{i=0}^{n-1}|a_{i}|A^{i-n}$ and $|b_{0}|>%
\sum_{i=1}^{m}|b_{i}|(A+1)^{i}$, then $f$ and $g$ are both
irreducible over $\mathbb{Q}$.
\end{corollary}
The following result considers the case when $f$ has a coefficient $a_j$ other then $a_n$ and $a_0$ with sufficiently large absolute value:
\begin{theorem}
\label{thm10Generala} Let $f(X)=\sum_{i=0}^{n}a_{i}X^{i},g(X)=%
\sum_{i=0}^{m}b_{i}X^{i}\in \mathbb{Z}[X]$, with $a_{0}a_{n}b_{0}b_{m}\neq 0$. If $|Res(f,g)|=pq$ with $p$ a prime number, $q$ a positive integer, and
\begin{eqnarray*}
|a_{j}| & > & \biggl(q^{\frac{1}{\min\{m,n\}}}+\sum\limits_{i=0}^{m-1}\Bigl|\frac{b_{i}}{b_{m}}\Bigr|\biggr)%
^{n-j}\cdot \sum_{i\neq j}|a_{i}|\quad \quad \text{and}  \\
|b_{0}| & > & \biggl(q^{\frac{1}{\min\{m,n\}}}+\Bigl(\sum\limits_{i\neq j}\Bigl|\frac{a_{i}}{a_{j}}\Bigr|%
\Bigr)^{\frac{1}{j}}\biggr)^{m}\cdot \sum\limits_{i=1}^{m}|b_{i}|
\end{eqnarray*}
for some index $j\in \{1,\dots ,n-1\}$, then $f$ and $g$\ are both irreducible over $\mathbb{Q}$.
\end{theorem}
In particular, for $q=1$ one obtains the following result.
\begin{corollary}
\label{thm10} Let $f(X)=\sum_{i=0}^{n}a_{i}X^{i},g(X)=%
\sum_{i=0}^{m}b_{i}X^{i}\in \mathbb{Z}[X]$, $a_{0}a_{n}b_{0}b_{m}\neq 0$. If $|Res(f,g)|$ is a prime number and for some index $j\in \{1,\dots
,n-1\}$ we have 
\[
|a_{j}|>\biggl(\sum\limits_{i=0}^{m}\Bigl|\frac{b_{i}}{b_{m}}\Bigr|\biggr)%
^{n-j}\sum_{i\neq j}|a_{i}|\ \text{ \ and \ } \ |b_{0}|>\biggl(1+\Bigl(\sum\limits_{i\neq j}\Bigl|\frac{a_{i}}{a_{j}}\Bigr|%
\Bigr)^{\frac{1}{j}}\biggr)^{m}\sum\limits_{i=1}^{m}|b_{i}|,
\]
then $f$ and $g$\ are both irreducible over $\mathbb{Q}$.
\end{corollary}

For the proof of our results we will need some lemmas on the location of polynomial roots. The first one, Lemma \ref{lemamica}, is quite elementary and is also well known as being an immediate consequence of Rouch\'e' s Theorem. Recall that by Rouch\'e's Theorem, if $f$ has a coefficient $a_j$ with $|a_{j}|>\sum_{i\neq j}|a_{i}|\delta ^{i-j}$ for some positive real $\delta $, then $f$ has $j$ roots with $|z|<\delta$ and $n-j$ roots with $|z|>\delta$. Lemma \ref{LemaALaRouche} will improve this result by providing sharper estimates on the location of the $j$ roots with $|z|<\delta$, and on the location of the $n-j$ remaining ones, and will also improve the estimates in Lemma \ref{lemamica} for polynomials $f$ that have at least three nonzero coefficients. Roughly speaking, Lemma \ref{LemaALaRouche} says that if a polynomial has one coefficient of sufficiently large absolute value, then its roots are forced to lie outside a certain annular region in the complex plane (which may degenerate to a disk or its complement in the complex plane). 
Our proofs will require a minimal use of Rouch\'e's Theorem.

\begin{lemma}
\label{lemamica} Let $f(X)=\sum_{i=0}^{n}a_{i}X^{i}\in \mathbb{C}[X]$, $%
a_{0}a_{n}\neq 0$ and let $\delta >0$ be an arbitrary real number.

i) If $|a_{0}|>\sum_{i=1}^{n}|a_{i}|\delta ^{i}$, then $f$ has no roots in
the disk $\{z:|z|\leq \delta \} $. Moreover, if $|a_{0}|\geq\sum_{i=1}^{n}|a_{i}|\delta ^{i}$, then $f$ has no roots in
the disk $\{z:|z|<\delta \} $;

ii) If $|a_{n}|>\sum_{i=0}^{n-1}|a_{i}|\delta ^{i-n}$, then all the roots of 
$f$ lie in the disk $\{ z:|z|<\delta \} $. Moreover, if $|a_{n}|\geq\sum_{i=0}^{n-1}|a_{i}|\delta ^{i-n}$, then all the roots of $f$ lie in the disk $\{ z:|z|\leq\delta \} $. 
\end{lemma}
\begin{proof}\  i) If $f$ would have a root $\theta $ with $|\theta |\leq
\delta $, then since $a_{0}=-\sum_{i=1}^{n}a_{i}\theta ^{i}$ we would
obtain $|a_{0}|\leq \sum_{i=1}^{n}|a_{i}|\cdot |\theta |^{i}\leq
\sum_{i=1}^{n}|a_{i}|\delta ^{i}$, a contradiction. Assume now that $|a_{0}|\geq\sum_{i=1}^{n}|a_{i}|\delta ^{i}$. If $f$ would have a root $\theta $ with $|\theta |<\delta $, then we would obtain $|a_{0}|\leq \sum_{i=1}^{n}|a_{i}|\cdot |\theta |^{i}<\sum_{i=1}^{n}|a_{i}|\delta ^{i}$, again a contradiction.

ii) If $f$ would have a root $\theta $ with $|\theta |\geq \delta $, then we
would obtain 
\[
0=\biggl| \sum\limits_{i=0}^{n}a_{i}\theta ^{i-n}\biggr| \geq
|a_{n}|-\sum\limits_{i=0}^{n-1}|a_{i}|\cdot |\theta |^{i-n}\geq
|a_{n}|-\sum\limits_{i=0}^{n-1}|a_{i}|\cdot \delta ^{i-n},
\]
a contradiction. Assume now that $|a_{n}|\geq\sum_{i=0}^{n-1}|a_{i}|\delta ^{i-n}$. If $f$ would have a root $\theta $ with $|\theta |>\delta $, then we
would obtain 
\[
0=\biggl| \sum\limits_{i=0}^{n}a_{i}\theta ^{i-n}\biggr| \geq
|a_{n}|-\sum\limits_{i=0}^{n-1}|a_{i}|\cdot |\theta |^{i-n}>
|a_{n}|-\sum\limits_{i=0}^{n-1}|a_{i}|\cdot \delta ^{i-n},
\]
again a contradiction. 
\end{proof}
\begin{lemma}
\label{LemaALaRouche} Let $f(X)=\sum_{i=0}^{n}a_{i}X^{i}\in \mathbb{C}[X]$, $a_{0}a_{n}\neq 0$, and assume that $f$ is not a binomial. If $|a_{j}|>\sum_{i\neq j}|a_{i}|\delta ^{i-j}$ for an index $j\in\{0,\dots ,n\}$ and a positive real number $\delta $, then $f$ has no roots in the annular region $\{ z:A\leq|z|\leq B\}$, with
\[
A=\delta \biggl( \frac{\sum_{i\neq j}|a_{i}|\delta ^{i}}{|a_{j}|\delta ^{j}}\biggr) ^{\frac{1}{j}}\text{\quad and \quad }B=\delta \biggl( \frac{|a_{j}|\delta ^{j}}{\sum_{i\neq j}|a_{i}|\delta ^{i}}\biggr) ^{\frac{1}{n-j}},
\]
and moreover, $f$ has $j$ roots with $|z|<A$ and $n-j$ roots with $|z|>B$. Here we use the convention that $A=0$ for $j=0$ and $B=\infty$ for $j=n$.
\end{lemma}
\begin{proof}\  First of all, note that as mentioned before, by Rouch\'e's Theorem the inequality $|a_{j}|>\sum_{i\neq j}|a_{i}|\delta ^{i-j}$ implies that $f$ has $j$ roots with $|z|<\delta$ and $n-j$ roots with $|z|>\delta$, so Lemma \ref{LemaALaRouche} gives sharper estimates on the location of the $j$ roots with absolute values less than $\delta$, and also on the location of the remaining $n-j$ ones. Indeed, observe that since $|a_{j}|>\sum_{i\neq j}|a_{i}|\delta ^{i-j}$, we have $A<\delta <B$, so $A<B$, and moreover, the ratio $\frac{\delta}{A}$ for $j\neq 0$ and the ratio $\frac{B}{\delta}$ for $j\neq n$ can be arbitrarily large.
\smallskip

{\bf Case 1:}\ $j=0$\quad In this case we have to prove that that all the roots of $f$ have absolute values exceeding $B=\delta \bigl(\frac{|a_{0}|}{\sum_{i\neq 0}|a_{i}|\delta ^{i}}\bigr) ^{1/n}$. Assume to the contrary that $f$ has a root $\theta $ with $|\theta |\leq B$. Then $|\theta |^n\sum_{i\neq 0}|a_{i}|\delta ^{i}\leq \delta ^n|a_{0}|$, which after multiplication by $|a_n|$ yields
\begin{equation}\label{Cuan}
|a_0|\cdot |a_n|\delta ^n\geq |a_n|\cdot |\theta|^n\sum\limits _{i=1}^{n}|a_i|\delta ^i.
\end{equation}
Since $B<\delta \bigl(\frac{|a_{0}|}{\sum_{i\neq 0}|a_{i}|\delta ^{i}}\bigr) ^{1/k}$ for each $k\in \{1,\dots ,n-1\}$, we also have $|\theta |^k\sum_{i\neq 0}|a_{i}|\delta ^{i}< \delta ^k|a_{0}|$ for each $k\in \{1,\dots ,n-1\}$. Adding together these inequalities multiplied by $|a_k|$, we obtain
\begin{equation}\label{Cuak}
|a_0|\cdot \sum\limits _{k=1}^{n-1}|a_k|\delta ^k> \sum\limits _{k=1}^{n-1}|a_k|\cdot |\theta|^k\sum\limits _{i=1}^{n}|a_i|\delta ^i,
\end{equation}
with the strict inequality guaranteed by our hypothesis that at least one of the coefficients $a_1,\dots ,a_{n-1}$ is nonzero, as $f$ is not a binomial. By (\ref{Cuan}) and (\ref{Cuak}) we then obtain
\[
|a_0|\cdot \sum\limits _{k=1}^{n}|a_k|\delta ^k> \sum\limits _{k=1}^{n}|a_k|\cdot |\theta|^k\sum\limits _{i=1}^{n}|a_i|\delta ^i,
\]
which after division by $\sum\limits _{i=1}^{n}|a_i|\delta ^i$ leads to $|a_0|> \sum\limits _{k=1}^{n}|a_k|\cdot |\theta|^k$. This contradicts the fact that $\theta $ is a root of $f$, and finishes the proof in this case.
\smallskip

{\bf Case 2:}\ $j\in\{1,\dots ,n-1\}$\quad Here we notice that
\begin{equation}\label{A<B}
A=\delta \cdot \max\limits _{k<j}\biggl( \frac{\sum_{i\neq j}|a_{i}|\delta ^{i}}{|a_{j}|\delta ^{j}}\biggr) ^{\frac{1}{j-k}}\text{\quad and
\quad }B=\delta \cdot \min\limits _{k>j}\biggl( \frac{|a_{j}|\delta ^{j}}{\sum_{i\neq j}|a_{i}|\delta ^{i}}\biggr) ^{\frac{1}{k-j}}.
\end{equation}

Let us assume to the contrary that $f$ has a root $\theta $ with $A\leq |\theta |\leq B$. Since $A<B$, we will either have $A<|\theta |\leq B$, or $A\leq |\theta |<B$.

In the first case, from $A<|\theta |$ we deduce by (\ref{A<B}) that
$\delta ^{k}|a_{j}\theta ^{j}|>|\theta |^{k}\sum_{i\neq j}|a_{i}|\delta ^{i}$ for each $k<j$, while from $|\theta |\leq B$ we obtain the inequality
$\delta ^{k}|a_{j}\theta ^{j}|\geq |\theta |^{k}\sum_{i\neq j}|a_{i}|\delta ^{i}$ for each $k>j$.
We observe now that since $a_{0}\neq 0$, the set of indices $k<j$ with $a_{k}\neq 0$ is not empty, and for all such indices we must have the strict inequality
$|a_{k}|\delta ^{k}|a_{j}\theta ^{j}|>|a_{k}\theta ^{k}|\sum_{i\neq j}|a_{i}|\delta ^{i}$, while for each index $k>j$ we have
$|a_{k}|\delta ^{k}|a_{j}\theta ^{j}|\geq |a_{k}\theta ^{k}|\sum_{i\neq j}|a_{i}|\delta ^{i}$.
Adding together these inequalities 
we deduce that
\[
 |a_{j}\theta ^{j}|\cdot \sum_{k\neq j}|a_{k}|\delta ^{k} >\sum_{k\neq j}|a_{k}\theta ^{k}|\cdot \sum_{i\neq j}|a_{i}|\delta ^{i}=\sum_{k\neq j}|a_{k}\theta ^{k}|\cdot \sum_{k\neq j}|a_{k}|\delta ^{k},
\]
which after cancellation of $\sum_{k\neq j}|a_{k}|\delta ^{k}$ on both sides yields
\begin{equation}\label{ecdelema}
|a_{j}\theta ^{j}|>\sum_{k\neq j}|a_{k}\theta ^{k}|.
\end{equation}

Let us assume now that $A\leq |\theta |<B$. Reasoning in a similar way, we deduce this time that since $a_{n}\neq 0$, we have at least one index $k>j$ with $|a_{k}|\delta ^{k}|a_{j}\theta ^{j}|>|a_{k}\theta ^{k}|\sum_{i\neq j}|a_{i}|\delta ^{i}$,
so (\ref{ecdelema}) will hold in this case too. 
On the other hand, since $\theta $ is a root of $f$, we must have
\[
|a_{j}\theta ^{j}|=\biggl| \sum_{k\neq j}a_{k}\theta ^{k}\biggr| \leq \sum_{k\neq j}|a_{k}\theta ^{k}|,
\]
which contradicts (\ref{ecdelema}) and completes the proof in this case.
\smallskip

{\bf Case 3:}\ $j=n$\quad In this case we have to prove that that all the roots of $f$ have absolute values less than $A=\delta \bigl(\frac{\sum_{i\neq n}|a_{i}|\delta ^{i}}{|a_{n}|\delta ^{n}}\bigr) ^{1/n}$. Assume towards a contradiction that $f$ has a root $\theta $ with $|\theta |\geq A$. Then $|a_n|\cdot|\theta |^n\geq\sum_{i\neq n}|a_{i}|\delta ^{i}$, which after multiplication by $|a_0|$ yields
\begin{equation}\label{CuanCaz3}
|a_n|\cdot |\theta|^n|a_0|\geq |a_0|\sum\limits _{i\neq n}|a_i|\delta ^i.
\end{equation}
Observe now that $|\theta|\geq A>\delta \bigl(\frac{\sum_{i\neq n}|a_{i}|\delta ^{i}}{|a_{n}|\delta ^{n}}\bigr) ^{1/(n-k)}$ for each $k\in\{1,\dots,n-1\}$, which implies that $|a_n|\cdot |\theta|^n\cdot \delta ^k> |\theta|^k\sum_{i\neq n}|a_i|\delta ^i$ for each $k\in\{1,\dots,n-1\}$. Adding together these inequalities multiplied by $|a_k|$, we obtain
\begin{equation}\label{CuakCaz3}
|a_n|\cdot |\theta|^n\sum\limits _{k=1}^{n-1}|a_k|\delta ^k> \sum\limits _{k=1}^{n-1}|a_k|\cdot |\theta|^k\sum\limits _{i\neq n}|a_i|\delta ^i,
\end{equation}
with the strict inequality guaranteed, as in (\ref{Cuak}), by our hypothesis that at least one of the coefficients $a_1,\dots ,a_{n-1}$ is nonzero, as $f$ is not a binomial. By (\ref{CuanCaz3}) and (\ref{CuakCaz3}) we deduce that
\[
|a_n|\cdot |\theta|^n\sum\limits _{k\neq n}|a_k|\delta ^k> \sum\limits _{k=0}^{n-1}|a_k|\cdot |\theta|^k\sum\limits _{i\neq n}|a_i|\delta ^i,
\]
which after divison by $\sum_{i\neq n}|a_i|\delta ^i$ finally leads us to $|a_n|\cdot |\theta|^n>\sum _{k=0}^{n-1}|a_k|\cdot |\theta|^k$, and this contradicts the fact that $\theta$ is a root of $f$.
\end{proof}
We mention that some variants of Lemma \ref{LemaALaRouche} have been used in \cite{Bonciocat1} and \cite{BBZ} to derive estimates for the location of the roots when some inequalities on the coefficients are known. Here we will use it in the following section in the proof of Theorem \ref{thmUnita}, which provides irreducibility conditions for polynomials having one large coefficient.

For the proof of Theorem \ref{thm10Generala} we will need the following technical lemma, which might be of independent interest and useful in other applications.

\begin{lemma}
\label{lemamare} Let $f(X)=\sum_{i=0}^{n}a_{i}X^{i}$ and $g(X)=\sum_{i=0}^{m}b_{i}X^{i}$ be complex polynomials and let $\delta $ be a positive real number.
If there exist three sequences of positive real numbers 
$\mu _{0},\dots ,\mu _{n}$, $\lambda _{1},\ldots ,\lambda _{m}$, $\gamma
_{0},\ldots ,\gamma _{m-1}$ and an index $j\in \{1,\dots ,n-1\}$ such that $a_j\neq 0$, $%
\sum_{k\neq j}\mu _{k}\leq 1$, $\sum_{k=1}^{m}\lambda _{k}\leq 1$, $%
\sum_{k=0}^{m-1}\gamma _{k}\leq 1$ and 
\begin{eqnarray}
\delta +\max\limits_{k<j}\biggl(\frac{|a_{k}|}{\mu _{k}|a_{j}|}\biggr)^{\frac{1}{%
j-k}}\  &<&\min\limits_{k>0}\biggl(\frac{\lambda _{k}|b_{0}|}{|b_{k}|}\biggr)^{%
\frac{1}{k}},  \label{urit1} \\
\delta +\max\limits_{k<m}\biggl(\frac{|b_{k}|}{\gamma _{k}|b_{m}|}\biggr)^{\frac{1%
}{m-k}} &<&\min\limits_{k>j}\biggl(\frac{\mu _{k}|a_{j}|}{|a_{k}|}\biggr)^{%
\frac{1}{k-j}},  \label{urit2}
\end{eqnarray}
then $|\theta -\xi |>\delta $ for every root $\theta $ of $f$ and every root $\xi $ of $g$.
\end{lemma}

Note that (\ref{urit1}) and (\ref{urit2}) are both satisfied if, for instance, $|a_{j}|$ and $|b_{0}|$ are sufficiently large. \medskip

\textit{Proof of Lemma \ref{lemamare}:} Assume that $f$ and $g$ factorize as $%
f(X)=a_{n}(X-\theta _{1})\cdots (X-\theta _{n})$ and $g(X)=b_{m}(X-\xi
_{1})\cdots (X-\xi _{m})$ with $\theta _{1},\ldots ,\theta _{n}$, $\xi
_{1},\ldots ,\xi _{m}\in \mathbb{C}$, let 
\begin{eqnarray*}
A &=&\max\limits_{k<j}\left( \frac{|a_{k}|}{\mu _{k}|a_{j}|}\right) ^{\frac{1%
}{j-k}},\quad B=\min\limits_{k>j}\left( \frac{\mu _{k}|a_{j}|}{|a_{k}|}%
\right) ^{\frac{1}{k-j}}, \\
C &=&\min\limits_{k>0}\left( \frac{\lambda _{k}|b_{0}|}{|b_{k}|}\right) ^{%
\frac{1}{k}},\quad \quad D=\max\limits_{k<m}\left( \frac{|b_{k}|}{\gamma
_{k}|b_{m}|}\right) ^{\frac{1}{m-k}},
\end{eqnarray*}
and note that according to our hypotheses $A+\delta <C$ and $D<B-\delta $. 

To find information on the location of the roots of our polynomials $f$ and $g$, we will adapt the method used in the proof of the following classical result of M. Fujiwara \cite{Fujiwara}:
\medskip

{\bf Theorem} (Fujiwara) \emph{Let $P(z)=\sum_{i=0}^{n}a_{i}z^{d_{i}}\in \mathbb{C}[z]$, with $%
0=d_{0}<d_{1}<\dots <d_{n}$ and $a_{0}a_{1}\dots a_{n}\neq 0$. Let also $\mu
_{0},\dots ,\mu _{n-1}\in (0,\infty )$ such that $\frac{1}{\mu _{0}}+\dots +%
\frac{1}{\mu _{n-1}}\leq 1$. Then all the roots of $P$ are contained in the
disk 
\[
|z|\leq\max_{0\leq j\leq n-1}\left( \mu _{j}\frac{|a_{j}|}{|a_{n}|}\right) ^{%
\frac{1}{d_{n}-d_{j}}}.
\]
}

We will first prove that all the roots of $g$ lie in the annular region $\{z:C\leq |z|\leq D\}$. To see this, assume first that $|\xi _{i}|<C$ for some
index $i\in \{1,\dots ,m\}$. Then we must have $\lambda
_{k}|b_{0}|>|b_{k}|\cdot |\xi _{i}|^{k}$ for $k=1,\ldots ,m$. Adding term by
term these inequalities and recalling the fact that $\sum_{k=1}^{m}\lambda
_{k}\leq 1$, we obtain $|b_{0}|>\sum_{k=1}^{m}|b_{k}|\cdot |\xi _{i}|^{k}$,
which can not hold, since $f(\xi _{i})=0$. If we assume now that $|\xi
_{i}|>D$ for some index $i\in \{1,\dots ,m\}$, we must have $\gamma
_{k}|b_{m}|\cdot |\xi _{i}|^{m}>|b_{k}|\cdot |\xi _{i}|^{k}$ for $k=0,\ldots
,m-1$. Using the fact that $\sum_{k=0}^{m-1}\gamma _{k}\leq 1$, we obtain
after summation $|b_{m}|\cdot |\xi _{i}|^{m}>\sum_{k=0}^{m-1}|b_{k}|\cdot
|\xi _{i}|^{k}$, again a contradiction.

Next, we will prove that $f$ has no roots in the annular region $\{z:A<|z|<B\}$. 
Assume that $A<|\theta _{i}|<B$ for some index $i\in \{1,\dots ,n\}$. From $A<|\theta _{i}|$ we then deduce that $\mu _{k}|a_{j}|\cdot |\theta
_{i}|^{j}>|a_{k}|\cdot |\theta _{i}|^{k}$ for each $k<j$, while from $%
|\theta _{i}|<B$ we find that $\mu _{k}|a_{j}|\cdot |\theta
_{i}|^{j}>|a_{k}|\cdot |\theta _{i}|^{k}$ for each $k>j$. Adding term by
term these inequalities and using the fact that $\sum_{k\neq j}\mu _{k}\leq
1 $, we obtain 
\begin{equation}
|a_{j}|\cdot |\theta _{i}|^{j}>{\textstyle\sum\limits_{k\neq j}}|a_{k}|\cdot
|\theta _{i}|^{k}.  \label{cinci}
\end{equation}
On the other hand, since $f(\theta _{i})=0$ we must have 
\[
0\geq |a_{j}|\cdot |\theta _{i}|^{j}-|{\textstyle\sum\limits_{k\neq j}}%
a_{k}\theta _{i}^{k}|\geq |a_{j}|\cdot |\theta _{i}|^{j}-{\textstyle%
\sum\limits_{k\neq j}}|a_{k}|\cdot |\theta _{i}|^{k},
\]
which contradicts (\ref{cinci}).
Finally, since $C-A>\delta $ and $B-D>\delta $, we must have $\min_{i,j}|\theta _{i}-\xi
_{j}|>\delta $, and this completes the proof of the lemma. \hfill $\square $
\medskip

{\it Proof of Theorem \ref{thm6Generala}.} By Lemma \ref{lemamica} we deduce that all the roots
of $f$ are situated in the disk $\{z:|z|<A\}$, while all the roots
of $g$ lie outside the disk $\{z:|z|\leq B\}$, so by Corollary \ref{coro1thm0} both 
$f$ and $g $ must be irreducible over $\mathbb{Q}$. \hfill $\square $

\medskip

{\it Proof of Theorem \ref{thm10Generala}.} \ For the proof we choose the parameters $\mu _{k}$, $\lambda _{k}$ and $\gamma _{k}$ in Lemma \ref{lemamare} to depend on the coefficients of $f$ and $g$. More precisely we take $\mu _{k}=|a_{k}|/\sum_{i\neq j}|a_{i}|$ for $k\neq j$, $\lambda _{k}=|b_{k}|/\sum_{i\neq 0}|b_{i}|$ for $k>0$, $\gamma
_{k}=|b_{k}|/\sum_{i\neq m}|b_{i}|$ for $k<m$, so for $\delta =q^{1/\min\{m,n\}}$ the inequalities (\ref{urit1}) and (\ref{urit2}) become 
\begin{eqnarray}
q^{\frac{1}{\min\{m,n\}}}+\max\limits_{k<j}\left( \frac{\sum_{i\neq j}|a_{i}|}{|a_{j}|}\right) ^{%
\frac{1}{j-k}}\ \thinspace &<&\min\limits_{k>0}\left( \frac{|b_{0}|}{\sum_{i\neq
0}|b_{i}|}\right) ^{\frac{1}{k}},  \label{ineg1} \\
q^{\frac{1}{\min\{m,n\}}}+\max\limits_{k<m}\left( \frac{\sum_{i\neq m}|b_{i}|}{|b_{m}|}\right) ^{%
\frac{1}{m-k}} &<&\min\limits_{k>j}\left( \frac{|a_{j}|}{\sum_{i\neq
j}|a_{i}|}\right) ^{\frac{1}{k-j}}.  \label{ineg2}
\end{eqnarray}
We have to check that our assumptions on $|a_j|$ and $|b_0|$ ensure us that inequalities (\ref{ineg1}) and (\ref{ineg2}) are satisfied. Since $|b_{0}|>\sum_{i\neq
0}|b_{i}|$, the minimum in the right member of (\ref{ineg1}) is greater than $%
1$, so it will be attained for $k=m$. On the other hand, since $|a_{j}|>\sum_{i\neq j}|a_{i}|$, the minimum in the right member of (\ref{ineg2}) is also greater than $1$,  so it will be attained for $k=n$. Besides, as $|a_{j}|>\sum_{i\neq j}|a_{i}|$, the maximum in the left member of (\ref{ineg1}) is attained for $k=0$. Now, since $|b_{0}|>\sum_{i\neq
0}|b_{i}|$ we also have $\sum_{i\neq m}|b_{i}|>|b_{m}|$, so the maximum in
the left member of (\ref{ineg2}) is attained for $k=m-1$. The inequalities (%
\ref{ineg1}) and (\ref{ineg2}) then read 
\begin{eqnarray*}
q^{\frac{1}{\min\{m,n\}}}+\left( \frac{\sum_{i\neq j}|a_{i}|}{|a_{j}|}\right) ^{%
\frac{1}{j}} &<& \left( \frac{|b_{0}|}{\sum_{i\neq
0}|b_{i}|}\right) ^{\frac{1}{m}},  \\
q^{\frac{1}{\min\{m,n\}}}+\frac{\sum_{i\neq m}|b_{i}|}{|b_{m}|}\ \ \   &<& \left( \frac{|a_{j}|}{\sum_{i\neq
j}|a_{i}|}\right) ^{\frac{1}{n-j}},
\end{eqnarray*}
which are easily seen to be equivalent to the inequalities on $|a_j|$ and $|b_0|$ in the statement of the theorem.
The proof follows now by Corollary \ref{coro1thm0}. \hfill $\square $

\begin{remark}\label{remarca2univariate}We note that one may obtain other similar irreducibility conditions by choosing in Lemma \ref{lemamare}
different sequences of positive real numbers $\mu _{0},\dots ,\mu _{n}$
satisfying $\sum_{k\neq j}\mu _{k}\leq 1$. For instance, one may simply
choose $\mu _{k}=1/n$ for $k\neq j$, or $\mu _{k}=2^{-n}{\binom{n}{k}} $ for 
$k\neq j$, or more generally $\mu _{k}=\beta ^{k}(1-\beta )^{n-k}{\binom{n}{k%
}}$ (with $0<\beta <1$) for $k\neq j$, and a similar choice may be obviously
done for $\lambda _{k}$ and $\gamma _{k}$ too. 
\end{remark}

\section{Irreducibility criteria in the case when $g$ is linear}\label{glinear}

The simplest irreducibility conditions that one may obtain using the results in previous sections appear for $g$ linear, and there are many irreducibility conditions in the literature that rely on the fact that 
$f$ takes a prime value $p$ at a suitable integral argument $m$. As already mentioned, in our framework this condition may be rephrased as $Res(f,g)=p$ with $g(X)=X-m$. For the sake of completeness, we will consider in this section the general case that $g$ is linear of the form $g(X)=bX-c$, even if this case too may be also studied without using resultants. Our main result in this respect is the following.

\begin{theorem}
\label{thmUnita} Let $f(X)=\sum_{i=0}^{n}a_{i}X^{i}\in \mathbb{Z}[X]$,
$a_{0}a_{n}\neq 0$, and suppose that for two nonzero integers $b,c$ we have
$|b^{n}f(c/b)|=p\cdot q$, where $p$ is a prime number and $q<|c|$ is a positive integer.
If for an index $j\in \{1,\dots ,n\} $ we have
\begin{equation}\label{conditiaa_j}
|a_{j}|>\sum\limits_{i\neq j}|a_{i}|\cdot |b|^{j-i}\cdot \sqrt[n]{\frac{(|c|+q)^{i(n-j)}}{(|c|-q)^{j(n-i)}}},
\end{equation}
then $f$ is irreducible over $\mathbb{Q}$. The same conclusion holds if we remove the condition $q<|c|$, and instead of (\ref{conditiaa_j}) we ask $a_0$ to satisfy $|a_0|>\sum_{i>0}|a_{i}|\cdot |b|^{-i}(|c|+q)^i$.
\end{theorem}
\begin{remark}\label{remarca1}
For $q<|c|$, condition $|a_0|>\sum_{i>0}|a_{i}|\cdot |b|^{-i}(|c|+q)^i$ may be seen as a particular case of (\ref{conditiaa_j}) for $j=0$.
\end{remark}

We will present here only three corollaries of Theorem \ref{thmUnita}, as follows.

\begin{corollary}
\label{corothm1} Let $f(X)=\sum_{i=0}^{n}a_{i}X^{i}\in \mathbb{Z}[X]$ with $a_i\in \{ -1,0,1\} $ for $i=0,\dots ,n$ and $a_0a_n\neq 0$.
If $|b^{n}f(c/b)|$ is a prime number for two nonzero integers $b,c$ with $|c|\geq 2|b|+1$ or $|b|\geq 2|c|+1$,
then $f$ is irreducible over $\mathbb{Q}$.
\end{corollary}

\begin{corollary}
\label{coro2thm1} Let $f(X)=\sum_{i=0}^{n}a_{i}X^{i}\in \mathbb{Z}[X]$ with $a_0a_n\neq 0$.
If $|b^{n}f(c/b)|$ is a prime number for two nonzero integers $b,c$ with $|c|\geq 2|b|+1$, and we have 
\[
|a_n|\geq \max \{ |a_{0}|,|a_{1}|,\cdots ,|a_{n-1}|\} \quad {\rm or}\quad |a_n|>\frac{|a_{0}|}{2^n}+\frac{|a_{1}|}{2^{n-1}}+\cdots +\frac{|a_{n-1}|}{2} ,
\]
then $f$ is irreducible over $\mathbb{Q}$.
\end{corollary}

\begin{corollary}
\label{corothm1,5} Let $f(X)=\sum_{i=0}^{n}a_{i}X^{i}\in \mathbb{Z}[X]$ with
$a_0a_n\neq 0$. If $|f(c)|$ is a prime number for an integer $c$ with $|c|\geq 2$, and 
\[
|a_{j}|\geq \sum_{i\neq j}|a_{i}|\cdot |c|^{1.6i}
\]
for some index $j\in \{1,\dots ,n-1\} $, then $f$ is irreducible over $\mathbb{Q}$.
\end{corollary}

{\it Proof of Theorem \ref{thmUnita}}. In our proof it will be more convenient to use the fact that 
\[
|Res(f(X),bX-c)|=|Res(f_b(X),g(X))|=|b^{n}f(c/b)|,
\]
with $f_{b}(X)=b^{n}f(X/b)=\sum_{i=0}^{n}b^{n-i}a_{i}X^{i}$ and $g(X)=X-c$.
\smallskip

{\bf Case 1: $j=0$}. Since our assumption on $|a_{0}|$ reads
$|b^{n}a_{0}|>\sum _{i=1}^{n}|b^{n-i}a_{i}|(q+|c|)^{i}$, we deduce by Lemma \ref{lemamica} that all the roots of $f_{b}(X)$ must be situated outside the disk $\{z:|z|\leq q+|c|\}$. Therefore, if $\theta $ is a root of $f_{b}(X)$, we have $|g(\theta )|=|\theta -c|\geq |\theta |-|c|>q$, so by Corollary \ref{corothm-1}, the polynomial
$f_{b}$ must be irreducible over $\mathbb{Q}$. One obtains the desired conclusion using the fact
that if $f_{b}$ is irreducible over $\mathbb{Q}$, then so is $f$.
\smallskip

{\bf Case 2: $j\in \{1,\dots ,n-1\} $}. Assume again that $\theta $ is a root of $f_{b}(X)$.
If we apply Lemma \ref{LemaALaRouche} to $f_b$, we deduce that if $|b^{n-j}a_{j}|>\sum_{i\neq j}|b^{n-i}a_{i}|\delta ^{i-j}$
for an index $j\in\{1,\dots ,n-1\}$ and a positive real number $\delta $, then $f_b$ will have no roots in the annular region
\[
\delta \biggl( \frac{\sum_{i\neq j}|b^{n-i}a_{i}|\delta ^{i}}{|b^{n-j}a_{j}|\delta ^{j}}\biggr) ^{\frac{1}{j}}=A_{\delta }\leq |z|\leq B_{\delta }=\delta \biggl( \frac{b^{n-j}|a_{j}|\delta ^{j}}{\sum_{i\neq j}|b^{n-i}a_{i}|\delta ^{i}}\biggr) ^{\frac{1}{n-j}},
\]
so we will either have $|\theta |<A_{\delta }$, or $|\theta |>B_{\delta }$, hence we will either have
\[
|g(\theta )|=|\theta -c|\geq |c|-|\theta |>|c|-A_{\delta },\qquad \text{if}\ |\theta |<A_{\delta },
\]
or
\[
|g(\theta )|=|\theta -c|\geq |\theta |-|c|>B_{\delta }-|c|,\qquad \text{if}\ |\theta |>B_{\delta }.
\]

In order to obtain $|g(\theta )|>q$ and to apply Corollary \ref{corothm-1} to our polynomial $f_b$, we will search for a suitable
$\delta $ such that $|b^{n-j}a_{j}|>\sum_{i\neq j}|b^{n-i}a_{i}|\delta ^{i-j}$ on the one
hand, and $|c|-A_{\delta }\geq q$ along with $B_{\delta }-|c|\geq q$, on the other hand. These three conditions are easily seen to be equivalent to
\begin{eqnarray*}
|b^{n-j}a_{j}|& >& E_1(\delta )=\frac{1}{\delta ^{j}}\sum_{i\neq j}|b^{n-i}a_{i}|\delta ^{i},\\
|b^{n-j}a_{j}|& \geq & E_2(\delta )=\frac{1}{(|c|-q)^{j}}\sum_{i\neq j}|b^{n-i}a_{i}|\delta ^{i},\quad q<|c|,\\
|b^{n-j}a_{j}|& \geq & E_3(\delta )=\frac{(|c|+q)^{n-j}}{\delta^{n}}\sum_{i\neq j}|b^{n-i}a_{i}|\delta ^{i},
\end{eqnarray*}
respectively. This shows that in order to obtain a sufficiently sharp condition on $|b^{n-j}a_{j}|$, we must search for a suitable $\delta $ that minimizes $\max\{E_{1}(\delta ),E_{2}(\delta ),E_{3}(\delta )\}$.
We notice now that for $\delta >0$, $E_2(\delta )$ is an increasing function, while $E_3(\delta )$ is a decreasing one, and that $E_{2}(\delta _{0})=E_{3}(\delta _{0})$
for $\delta _{0}=\sqrt[n]{(|c|+q)^{n-j}(|c|-q)^{j}}$, so a suitable candidate for our $\delta $ will be $\delta _{0}$, provided we check that $E_{1}(\delta _{0})<E_{2}(\delta _{0})$. Fortunately, this inequality holds, since $\delta _{0}>|c|-q$, so
if we ask $b^{n-j}a_j$ to satisfy $|b^{n-j}a_{j}|\geq E_{2}(\delta _{0})$, the polynomial $f_{b}$, and hence $f$ too, must be irreducible over $\mathbb{Q}$. Note that in general $E_{2}(\delta _{0})$ is not an integer, so we will actually use the strict inequality $|b^{n-j}a_{j}|>E_{2}(\delta _{0})$, which is also used in the cases $j=0$ and $j=n$.
The proof in this case finishes by observing that the inequality $|b^{n-j}a_{j}|>E_{2}(\delta _{0})$ is equivalent to the inequality on $|a_{j}|$ in the statement of the theorem.

We note here that one might obtain sharper conditions on $|a_j|$ by searching for estimates on the location of the roots of $f_{b}$, that are sharper than those given by Lemma \ref{lemamica}. 
\smallskip

{\bf Case 3: $j=n$}. Our assumption on $|a_{n}|$ reads
$|a_{n}|>\sum\nolimits_{i=0}^{n-1}|b^{n-i}a_{i}|(|c|-q)^{i-n}$, so by Lemma \ref{lemamica}, all the roots of $f_{b}(X)$ must be situated in the open disk $|z|<|c|-q$.
Assume now that $\theta $ is a root of $f_{b}$. Then $|g(\theta )|=|\theta -c|\geq |c|-|\theta |>q$, so by Corollary \ref{corothm-1},
$f_{b}(X)$, and hence $f(X)$ too, must be irreducible over $\mathbb{Q}$. This completes the proof. \hfill $\square $

\medskip 

{\it Proof of Corollary \ref{corothm1}.} Here we only need to observe
that if all the coefficients of a polynomial $f$ belong to $\{-1,0,1\}$, the condition
$|a_{n}|>\sum\nolimits_{i=0}^{n-1}|a_{i}|(|c/b|-1/|b|)^{i-n}$, corresponding to $q=1$ and $j=n$ in the statement
of Theorem \ref{thmUnita}, will hold for any two nonzero integers $b,c$ with $|c|\geq 2|b|+1$. Indeed, if $|c|\geq 2|b|+1$, we deduce that
\[
\sum\limits_{i=0}^{n-1}|a_{i}|\biggl( \frac{|b|}{|c|-1}\biggr) ^{n-i}\leq \sum\limits_{i=0}^{n-1}\biggl( \frac{|b|}{|c|-1}\biggr) ^{n-i}\leq \sum\limits_{i=0}^{n-1}\frac{1}{2^{n-i}}<\sum\limits_{i=1}^{\infty }\frac{1}{2^{i}}=1=|a_n|.
\]

Next, we notice that if we denote by $h$ the reciprocal of $f$, that is $h(X)=X^{\deg f}\cdot f(1/X)$, we obviously have
$b^{\deg f}f(c/b)=c^{\deg h}h(b/c)$, so the same conclusion
will hold if we replace the condition that $|c|\geq 2|b|+1$ by $|b|\geq 2|c|+1$. \hfill $\square $
\medskip

{\it Proof of Corollary \ref{coro2thm1}.} One may check that if we assume that $|a_n|\geq \max\limits _{0\leq i\leq n-1} |a_i| $, or that $|a_{n}|>\sum\nolimits_{i=0}^{n-1}|a_{i}|\cdot 2^{i-n}$, the condition
$|a_{n}|>\sum\nolimits_{i=0}^{n-1}|a_{i}|(|c/b|-1/|b|)^{i-n}$ corresponding to $q=1$ and $j=n$ in the statement
of Theorem \ref{thmUnita}, will hold for any two nonzero integers $b,c$ with $|c|\geq 2|b|+1$. \hfill $\square $
\medskip

{\it Proof of Corollary \ref{corothm1,5}.} We apply Theorem \ref{thmUnita} with $b=q=1$ and $j\in \{1,\dots ,n-1\} $, and observe that $x^{1.6}>x^{\log _23}\geq x+1$ for $x\geq 2$, so
\[
|c|^{1.6i}>(|c|+1)^i>\frac{(|c|+1)^{\frac{i(n-j)}{n}}}{(|c|-1)^{\frac{j(n-i)}{n}}} 
\]
for each $i\neq j$ and every integer $c$ with $|c|\geq 2$.\hfill $\square $

\section{Irreducibility criteria in the case when $g$ is quadratic}\label{gquadratic}

To find an explicit formula for the resultant of two arbitrary polynomials $f$ and $g$ of large degree, or to find information on its canonical decomposition, is in general quite a difficult task. We will content ourselves to present here formulas for the resultant only for polynomials $g$ of degree $2$. In the following two lemmas the resultant will be computed by two different methods,
the first one based on its invariance under a change of indeterminate, and the second one requiring the use of a suitable linear recurrence sequence (Binet's method). Since we couldn't find a proper reference for these two lemmas, we will include here their proofs.

\begin{lemma}\label{lemarez1}
Let $f(X)=\sum_{i=0}^{n}a_{i}X^{i}$ and $g(X)=aX^2+bX+c$ be complex polynomials, with $a\neq 0$, and let $D=\frac{b^{2}-4ac}{4a^{2}}$ and
$A_{i}=a_{i}+\sum\limits _{j>i}a_{j}\tbinom{j}{i}\bigl( \frac{-b}{2a}\bigr) ^{j-i}$ for $i=0,1,\dots ,n$. Then
\begin{equation*}
|Res(f,g)|=\left\lbrace
\begin{array}{ccc}
|a|^{n}\cdot \biggl| \biggl( \sum\limits _{i=0}^{\lfloor {n/2} \rfloor } A_{2i}D^{i}\biggr)^{2} -D\biggl( \sum\limits _{i=0}^{\lfloor {(n-1)/2} \rfloor }A_{2i+1}D^{i}\biggr)^{2} \biggr| , & \thinspace if & D\neq 0 \\ 
|a|^{n}\cdot \left| f\left( \frac{-b}{2a}\right) \right| ^{2},\hspace{5.8cm} & \thinspace if & D=0. 
\end{array}\right.
\end{equation*}
\end{lemma}

Alternatively, one may compute this resultant in the following way.

\begin{lemma}\label{lemarez2}
Let $f(X)=\sum_{i=0}^{n}a_{i}X^{i}$ and $g(X)=aX^2+bX+c$ be complex polynomials, with $a\neq 0$, and let us define the linear recurrence sequence
$(x_{k})_{k\geq 0}$ by $x_{0}=2, x_{1}=-\frac{b}{a}$ and $x_{k+2}=-\frac{b}{a}x_{k+1}-\frac{c}{a}x_{k}$ for $k\geq 0$. Then
\begin{equation*}
|Res(f,g)|=\left\lbrace
\begin{array}{ccc}
\biggl| \sum\limits _{i=0}^{n}c^{i}a^{n-i}a_{i}^{2}+\sum\limits _{0\leq i<j\leq n}c^{i}a^{n-i}x_{j-i}a_{i}a_{j}\biggr| , & \thinspace if & b^{2}\neq 4ac \\ 
|a|^{n}\cdot \left| f\left( \frac{-b}{2a}\right) \right| ^{2},\hspace{3.98cm} & \thinspace if & b^{2}=4ac. 
\end{array}\right.
\end{equation*}
\end{lemma}

The following results are obtained by considering several quadratic polynomials $g$, namely $X^{2}+1$, $X^{2}-m$ with $m$ a non-square positive integer,
$X^{2}+X+1$ and $X^{2}-X-1$.

\begin{theorem}
\label{thm3} Let $a_{0},a_{1}\dots ,a_{n}\in\mathbb{Z}$, with $a_{0}a_{n}\neq 0$, and suppose that
\[
(a_{0}-a_{2}+a_{4}-\dots
)^{2}+(a_{1}-a_{3}+a_{5}-\dots )^{2}=p\cdot q,
\]
where $p$ is a prime number and $q$ is a positive integer. If $|a_{0}|>\sum_{i=1}^{n}|a_{i}|\sqrt{(1+q)^i}$, then the polynomial $\sum_{i=0}^{n}a_{i}X^{i}$
is irreducible over $\mathbb{Q}$.
\end{theorem}
For $q=1$ one obtains the following corollary, stated as Theorem B in the Introduction.
\begin{corollary}
\label{corolthm3} If we write a prime number as $(a_{0}-a_{2}+a_{4}-\dots
)^{2}+(a_{1}-a_{3}+a_{5}-\dots )^{2}$ with $a_{i}\in \mathbb{Z}$ and $%
|a_{0}|>\sum_{i> 0}|a_{i}|\sqrt{2^i}$, then the polynomial $\sum_{i}a_{i}X^{i}$
is irreducible over $\mathbb{Q}$.
\end{corollary}

\begin{theorem}
\label{thm4}  Let $a_{0},a_{1}\dots ,a_{n}\in\mathbb{Z}$, with $a_{0}a_{n}\neq 0$, and suppose that
\[
(a_{0}+a_{2}m+a_{4}m^{2}+a_{6}m^{3}+%
\dots)^2-m(a_{1}+a_{3}m+a_{5}m^{2}+a_{7}m^{3}+\dots)^2=p\cdot q
\]
where $p$ is a prime number, $q$ is a positive integer and $m$ is a non-square positive integer. If 
$|a_{0}|>\sum_{i=1}^{n}|a_{i}|\sqrt{m+q}\thinspace ^{i}$, or if $q<m$ and $|a_{n}|>\sum_{i=0}^{n-1}|a_{i}|\sqrt{m-q}\thinspace ^{i-n}$,
then the polynomial $\sum_{i=0}^{n}a_{i}X^{i}$ is irreducible over $\mathbb{Q}$.
\end{theorem}

\begin{corollary}
\label{corolthm4} If $(a_{0}+a_{2}m+a_{4}m^{2}+a_{6}m^{3}+%
\dots)^2-m(a_{1}+a_{3}m+a_{5}m^{2}+a_{7}m^{3}+\dots)^2$ is a prime number
for a non-square positive integer $m$ and integers $a_{0},a_{1},\dots,a_{n}$
satisfying 
\[
|a_{0}|>\sum_{i=1}^{n}|a_{i}|\sqrt{m+1}\thinspace ^{i} \quad {\rm or} \quad |a_{n}|>\sum_{i=0}^{n-1}|a_{i}|\sqrt{m-1}\thinspace ^{i-n},
\]
then the polynomial $\sum_{i=0}^{n}a_{i}X^{i}$ is irreducible over $\mathbb{Q}$.
\end{corollary}

\begin{theorem}
\label{thm4,5} Let $a_{0},a_{1},\dots,a_{n}$ be integers with $a_0a_n\neq 0$, let $S_{j}=\sum _{i\equiv j\thinspace ({\rm mod}\thinspace 3)}a_{i}$ for $j\in\{0,1,2\}$, and suppose that $S_0^2+S_1^2+S_2^2-S_0S_1-S_0S_2-S_1S_2=p\cdot q$, where $p$ is a prime number and $q$ is a positive integer.
If $|a_0|>\sum_{i=1}^{n}|a_{i}|\bigl( \frac{1+\sqrt{4q+5}}{2}\bigr) ^i$, then the polynomial $\sum_{i=0}^{n}a_{i}X^{i}$ is irreducible over $\mathbb{Q}$.
\end{theorem}
For $q=1$ we obtain the following result, stated as Theorem C in the Introduction.
\begin{corollary}
\label{corolthm4,5} Let $a_{0},a_{1},\dots,a_{n}$ be integers with $a_0a_n\neq 0$ and $|a_0|>\sum_{i=1}^{n}|a_{i}|2^i$, and let
$S_{j}=\sum _{i\equiv j\thinspace ({\rm mod}\thinspace 3)}a_{i}$ for $j\in\{0,1,2\}$.
If $S_0^2+S_1^2+S_2^2-S_0S_1-S_0S_2-S_1S_2$ is a prime number, then the polynomial $\sum_{i=0}^{n}a_{i}X^{i}$ is irreducible over $\mathbb{Q}$.
\end{corollary}

\begin{theorem}
\label{thm5} Let $a_{0},a_{1},\dots,a_{n}\in \mathbb{Z}$, $a_n\neq 0$, and let $(L_k)_{k\geq 0}$ denote the sequence of Lucas numbers.
If
\[
\biggl| \sum\limits _{i=0}^{n}(-1)^{i}a_{i}^{2}+\sum\limits _{0\leq i<j\leq n}(-1)^{i}L_{j-i}a_{i}a_{j}\biggr| =p\cdot q,
\]
where $p$ is a prime number and $q$ is a positive integer, and
$|a_{0}|>\sum_{i=1}^{n}|a_{i}|\bigl( \frac{1+\sqrt{4q+5}}{2}\bigr) ^{i}$, then the polynomial $\sum_{i=0}^{n}a_{i}X^{i}$ is irreducible over $\mathbb{Q}$.
\end{theorem}

\begin{corollary}
\label{corolthm5} Let $a_{0},a_{1},\dots,a_{n}\in \mathbb{Z}$ with $a_n\neq 0$ and $|a_0|>\sum_{i=1}^{n}|a_{i}|2^i$, and let
$(L_k)_{k\geq 0}$ denote the sequence of Lucas numbers.
If up to a sign,
\[
\sum\limits _{i=0}^{n}(-1)^{i}a_{i}^{2}+\sum\limits _{0\leq i<j\leq n}(-1)^{i}L_{j-i}a_{i}a_{j}
\]
is a prime number, then the polynomial $\sum_{i=0}^{n}a_{i}X^{i}$ is irreducible over $\mathbb{Q}$.
\end{corollary}

{\it Proof of Lemma \ref{lemarez1}.} We will use here the invariance of the resultant under a change of indeterminate, that is the fact that
\[
Res(f(X),g(X))=Res(f(X+\delta ),g(X+\delta ))
\]
for arbitrary $\delta \in\mathbb{C}$.
In particular, if we choose $\delta =-\frac{b}{2a}$, and express $f(X-\frac{b}{2a})$ as $A_{0}+A_{1}X+\cdots +A_{n}X^{n}$ with
\[
A_{i}=a_{i}+\sum\limits _{j>i}a_{j}\tbinom{j}{i}\biggl( \frac{-b}{2a}\biggr) ^{j-i}\qquad  i=0,1,\dots ,n,
\]
we successively obtain
\begin{eqnarray*}
|Res(f,g)|& =& \biggl| Res\biggl( f\biggl( X-\frac{b}{2a}\biggr) ,g\biggl( X-\frac{b}{2a}\biggr) \biggr) \biggr| \\
& =& \biggl| Res\biggl( f\biggl( X-\frac{b}{2a}\biggr) ,a(X-\sqrt{D})(X+\sqrt{D})\biggr) \biggr| \\
& =& |a|^{n}\cdot \biggl| \sum_{i=0}^{n}A_{i}\sqrt{D}\thinspace ^{i}\cdot \sum_{i=0}^{n}A_{i}(-\sqrt{D})^{i}\biggr| \\
& =& |a|^{n}\cdot \biggl| \biggl( \sum\limits _{i=0}^{\lfloor {n/2} \rfloor } A_{2i}D^{i}\biggr)^{2} -D\biggl( \sum\limits _{i=0}^{\lfloor {(n-1)/2} \rfloor }A_{2i+1}D^{i}\biggr)^{2} \biggr| ,
\end{eqnarray*}
and this completes the proof. \hfill $\square $
\medskip

{\it Proof of Lemma \ref{lemarez2}.} Let $\theta _{1}=\frac{-b+\sqrt{b^{2}-4ac}}{2a}$ and $\theta _{2}=\frac{-b-\sqrt{b^{2}-4ac}}{2a}$ be the roots of $g$, and assume that $b^{2}\neq 4ac$.
Using Binet's method, one way to compute the powers of $\theta _{1}$ and $\theta _{2}$ is to use the linear recurrence sequence $(y_{k})_{k\geq 0}$
given by $y_{k+2}=-\frac{b}{a}y_{k+1}-\frac{c}{a}y_{k}$, $k\geq 0$, with $y_{0}=0$ and $y_{1}=1$. One easily checks that $\theta _{1}$ and $\theta _{2}$ satisfy the relation
\begin{equation}\label{thetapower}
\theta ^{k}=y_{k}\theta -\frac{c}{a}y_{k-1}\qquad \text{for}\ k\geq 1,
\end{equation}
and that $y_{k}$ may be expressed in the closed form
\begin{equation}\label{yk}
y_{k}=a\frac{\theta _{1}^{k}-\theta _{2}^{k}}{\sqrt{b^{2}-4ac}}\qquad \text{for}\ k\geq 0.
\end{equation}

Letting now $y_{-1}:=0$ and using (\ref{thetapower}), one further deduces that $f(\theta _{1})=\theta _1S_{1}-\frac{c}{a}S_{2}$ and $f(\theta _{2})=\theta _2S_{1}-\frac{c}{a}S_{2}$,
with $S_{1}=\sum_{i=0}^{n}a_{i}y_{i}$ and $S_{2}=\sum_{i=0}^{n}a_{i}y_{i-1}$, which yields
\begin{eqnarray*}
|Res(f,g)|& =& |a^n|\cdot |f(\theta _{1} )\cdot f(\theta _{2})|\\
& =& |a^n|\cdot \bigl| \bigl(\theta _1S_{1}-\frac{c}{a}S_{2}\bigr) \bigl( \theta _2S_{1}-\frac{c}{a}S_{2}\bigr) \bigr| \\
& =& |ca^{n-2}|\cdot |aS_{1}^{2}+bS_{1}S_{2}+cS_{2}^{2}|\\
& =& |ca^{n-2}|\cdot \biggl| \sum\limits _{i=0}^{n}\beta _{i}a_{i}^2+\sum\limits _{0\leq i<j\leq n}\gamma _{ij}a_{i}a_{j}\biggr| ,
\end{eqnarray*}
with 
\begin{eqnarray*}
\beta _{i}& =& ay_{i}^{2}+by_{i}y_{i-1}+cy_{i-1}^{2}\quad \text{and}\\
\gamma _{ij}& =& 2ay_{i}y_{j}+b(y_{i}y_{j-1}+y_{i-1}y_{j})+2cy_{i-1}y_{j-1}.
\end{eqnarray*}

Using now (\ref{yk}), the recurrence relation for $y_{i}$, and the relations $\theta _{1}+\theta _{2}=-b/a$ and $\theta _{1}\theta _{2}=c/a$, we deduce that
$\beta _{i}=a\bigl( \frac{c}{a}\bigr) ^{i-1}$ and $\gamma _{ij}=a\bigl( \frac{c}{a}\bigr) ^{i-1}(\theta _{1}^{j-i}+\theta _{2}^{j-i})$.
We observe now that the sequences $(x_{k})_{k\geq 0}$ and $(\theta_{1}^{k}+\theta _{2}^{k})_{k\geq 0}$ satisfy the same recurrence relation and have the same two initial terms,
so they must coincide, hence $\gamma _{ij}=a\bigl( \frac{c}{a}\bigr) ^{i-1}x_{j-i}$.

Finally, using these expressions for $\beta _{i}$ and $\gamma _{ij}$ we conclude that
\[
|Res(f,g)|=\biggl| \sum\limits _{i=0}^{n}c^{i}a^{n-i}a_{i}^{2}+\sum\limits _{0\leq i<j\leq n}c^{i}a^{n-i}x_{j-i}a_{i}a_{j}\biggr| ,
\]
which completes the proof. \hfill $\square $
\medskip

{\it Proof of Theorem \ref{thm3}.} Let us take $g(X)=x^{2}+1$. By Lemma \ref{lemarez1} with $D=-1$ and $A_{i}=a_{i}$ for each $i$ we obtain
\[
|Res(f,g)|=(a_{0}-a_{2}+a_{4}-\dots
)^{2}+(a_{1}-a_{3}+a_{5}-\dots )^{2}=p\cdot q,
\]
where $p$ is a prime number and $q$ is a positive integer. By Lemma \ref{lemamica}, all the roots of $f$ are situated outside the
disk $\{ z:|z|\leq \sqrt{1+q}\} $, so if $\theta $ is a root of $f$, we must have
\[
|g(\theta )|=|\theta ^2+1|\geq |\theta |^2-1>(\sqrt{1+q})^2-1=q,
\]
and the conclusion follows now by Corollary \ref{corothm-1}. \hfill $\square $

\medskip 

{\it Proof of Theorem \ref{thm4}.} In this case we consider the polynomial $g(X)=X^{2}-m$,
so by Lemma \ref{lemarez1} with $D=m$ and $A_{i}=a_{i}$ for each $i$ we see that $|Res(f,g)|$ is equal to 
\[
(a_{0}+a_{2}m+a_{4}m^{2}+a_{6}m^{3}+\dots
)^{2}-m(a_{1}+a_{3}m+a_{5}m^{2}+a_{7}m^{3}+\dots )^{2}.
\]

Let us suppose first that $|a_{0}|>\sum_{i=1}^{n}|a_{i}|\sqrt{m+q}\thinspace ^{i}$. Then, according to Lemma \ref{lemamica}, $f$ will have no roots in the disk $\{ z:|z|\leq \sqrt{m+q}\} $, so if $\theta $ is a root of $f$, we must have
\[
|g(\theta )|=|\theta ^2-m|\geq |\theta |^2-m>(\sqrt{m+q})^2-m=q,
\]
and the conclusion follows by Corollary \ref{corothm-1}.

Assume now that $q<m$ and $|a_{n}|>\sum_{i=0}^{n-1}|a_{i}|\sqrt{m-q}\thinspace ^{i-n}$.
In this case, we deduce from Lemma \ref{lemamica} that all the roots of $f$ will belong to the
open disk $\{ z:|z|<\sqrt{m-q}\} $, so for a root $\theta $ of $f$, we must have
\[
|g(\theta )|=|\theta ^2-m|\geq m-|\theta |^2>m-(\sqrt{m-q})^2=q,
\]
and the conclusion follows again by Corollary \ref{corothm-1}. \hfill $\square $

\medskip 

{\it Proof of Theorem \ref{thm4,5}.} Here we consider the polynomial $g(X)=X^{2}+X+1$, with roots $\xi =\frac{-1+i\sqrt{3}}{2}$ and $\xi ^2=\frac{-1-i\sqrt{3}}{2}$,
so
\[
|Res(f,g)|=|f(\xi )\cdot f(\xi ^2)|=S_0^2+S_1^2+S_2^2-S_0S_1-S_0S_2-S_1S_2=p\cdot q.
\]
If we assume now that $|a_{0}|>\sum_{i=1}^{n}|a_{i}|\bigl( \frac{1+\sqrt{4q+5}}{2}\bigr) ^{i}$, then by Lemma \ref{lemamica}, $f$ will have no roots in the disk $\{ z:|z|\leq
\frac{1+\sqrt{4q+5}}{2}\} $, so for each root $\theta $ of $f$ we deduce that
\[
|g(\theta )|=|\theta ^2+\theta +1|\geq |\theta |^2-|\theta |-1>q,
\]
and one applies again Corollary \ref{corothm-1}. \hfill $\square $
\medskip 

{\it Proof of Theorem \ref{thm5}.} We consider now the polynomial $g(X)=X^{2}-X-1$, with roots $\varphi =\frac{1+\sqrt{5}}{2}$, the golden ratio, and $\psi=\frac{1-\sqrt{5}}{2}=1-\varphi $.
We recall that the Lucas numbers $L_{k}$ and the Fibonacci numbers $F_{k}$ satisfy the same recurrence relation $x_{k+2}=x_{k+1}+x_{k}$, with $L_{0}=2,\ L_{1}=1$, while $F_{0}=0,\ F_{1}=1$.
It is easy to check that one may express the Lucas numbers in the following closed form
\[
L_{k}=\varphi ^{k}+\psi ^{k},\quad k=0,1,2,\dots
\]
By Lemma \ref{lemarez2} with $a=1$ and $b=c=-1$ we find that
\[
|Res(f,g)|=\biggl| \sum\limits _{i=0}^{n}(-1)^{i}a_{i}^{2}+\sum\limits _{0\leq i<j\leq n}(-1)^{i}L_{j-i}a_{i}a_{j}\biggr| =p\cdot q.
\]
Again, since $|a_{0}|>\sum_{i=1}^{n}|a_{i}|\bigl( \frac{1+\sqrt{4q+5}}{2}\bigr) ^{i}$,
we deduce by Lemma \ref{lemamica} that $f$ will have no roots in the disk $\{ z:|z|\leq \frac{1+\sqrt{4q+5}}{2}\} $, so for each root $\theta $ of $f$ we obtain
\[
|g(\theta )|=|\theta ^2-\theta -1|\geq |\theta |^2-|\theta |-1>q.
\]
The desired conclusion follows again by Corollary \ref{corothm-1}. \hfill $\square $

\section{Irreducibility criteria for linear combinations of $f$ and $g$}\label{LinearCombinations}

Once we have found a pair $(f,g)$ of polynomials whose resultant is a prime number and such that the roots of $f$ are more than $1$ away from the roots of $g$, we may also obtain some irreducibility conditions for linear combinations of the form $Mf+Ng$, with $M,N\in\mathbb{Z}$ having absolute values bounded by certain numbers depending on the coefficients of $f$ and $g$. The following results give such irreducibility conditions that depend on two suitable parameters $A$ and $B$, and show that this kind of {\it irreducibility in pairs} extends locally to some linear combinations $(Mf+Ng,g)$ and $(f,Mf+Ng)$. We will analyse separately the cases $\deg f<\deg g$, $\deg f=\deg g$ and $\deg f>\deg g$, as they lead to different outcomes.
\medskip

{\bf Case 1:}\ $\deg f<\deg g$\quad In this case we will prove the following result:
\begin{theorem}
\label{gradf<gradg} Let $f(X)=\sum_{i=0}^{n}a_{i}X^{i}$,$\
g(X)=\sum_{i=0}^{m}b_{i}X^{i}\in \mathbb{Z}[X]$, $a_{0}a_{n}b_{0}b_{m}\neq 0$%
, $n<m$, and suppose that $|Res(f,g)|=pq$ with $p$ a prime number, and $q$ a positive integer. Suppose also that for two positive real numbers $A,B$ with $B\geq A+q^{\frac{1}{n}}$ we have 
$|a_{n}|>\sum_{i=0}^{n-1}|a_{i}|A^{i-n}$ and $|b_{0}|\geq\sum_{i=1}^{m}|b_{i}|B^{i}$. Let $M$ and $N$ be integers, not both zero, satisfying the inequalities
\vspace{-0.6cm}

\begin{equation}\label{MNPrimaDeTot}
|N| \leq\frac{B-A}{q^{\frac{1}{n}}} \quad and \quad |M|\leq|N|\cdot \frac {|b_{0}|-\sum\limits _{i=1}^{m}|b_{i}|(A+q^{\frac{1}{n}}|N|)^{i}} {\sum\limits_{i=0}^{n}|a_{i}|(A+q^{\frac{1}{n}}|N|)^{i}}.
\end{equation}
Then the polynomial $Mf(X)+Ng(X)$ is irreducible over $\mathbb{Q}$. Moreover, the same conclusion holds if we ask $a_n$ and $b_0$ to satisfy $|a_{n}|\geq\sum_{i=0}^{n-1}|a_{i}|A^{i-n}$ and $|b_{0}|>\sum_{i=1}^{m}|b_{i}|B^{i}$, provided that the second inequality in (\ref{MNPrimaDeTot}) is a strict one.
\end{theorem}

A simple instance of Theorem \ref{gradf<gradg} appears in the following result.

\begin{corollary}
\label{Coro1gradf<gradg} Let $g(X)=\sum_{i=0}^{m}b_{i}X^{i}\in \mathbb{Z}[X]$ with $m\geq 2$, $b_{0}b_{m}\neq 0$, and suppose that $|a^mg(\frac{1}{a})|$ is a prime number for an integer $a$ with $|a|>1$. If for a real number $B\geq 2$ we have 
$|b_{0}|\geq 1+|a|B+\sum_{i=1}^{m}|b_{i}|B^{i}$, then for all integers $M,N$ not both zero and such that
\[
|M|\leq |N|\leq B-1,
\]
the polynomial $M(aX-1)+Ng(X)$ is irreducible over $\mathbb{Q}$. The same conclusion also holds for $a=\pm 1$ and integers $M,N$ satisfying $|M|<|N|\leq B-1$.
\end{corollary}
In particular, for $a=1$ we have:
\begin{corollary}
\label{Coro2gradf<gradg} If we write a prime number as a sum of integers $b_0,\dots,b_m$ such that $|b_{0}|\geq 1+B+\sum_{i=1}^{m}|b_{i}|B^{i}$ for some real number $B\geq 2$, then for all integers $M,N$ with $|M|<|N|\leq B-1$,
the polynomial  $M(X-1)+N\sum_{i=0}^mb_iX^i$ is irreducible over $\mathbb{Q}$.
\end{corollary}

{\it Proof of Theorem \ref{gradf<gradg}.} Assume that $M$ and $N$ satisfy inequalities (\ref{MNPrimaDeTot}). First of all we notice that for $|N|\leq \frac{B-A}{q^{\frac{1}{n}}}$ and $|b_0|\geq\sum_{i=1}^{m}|b_{i}|B^{i}$, the numerator in the second inequality in (\ref{MNPrimaDeTot}) is nonnegative, so this inequality makes sense.
Assume now that $f$ and $g$
factorize as $f(X)=a_{n}(X-\theta _{1})\cdots (X-\theta _{n})$ and $g(X)=b_{m}(X-\xi _{1})\cdots (X-\xi _{m})$ with $\theta _{1},\ldots ,\theta_{n}$, $\xi _{1},\ldots ,\xi _{m}\in \mathbb{C}$. By Lemma \ref{lemamica} we see that $|\theta _{i}|<A$ for $i=1,\ldots ,n$ and $|\xi _{j}|\geq B$ for
$j=1,\ldots ,m$, so $\min\limits_{i,j}|\theta _i-\xi _j|>q^{\frac{1}{n}}$ and by Corollary \ref{coro1thm0} we conclude that both $f$
and $g$ must be irreducible over $\mathbb{Q}$. We may therefore assume that none of $M$ and $N$ is zero.
Suppose now that $Mf+Ng$ decomposes as $Nb_m(X-\lambda _1)\cdots (X-\lambda _m)$ with $\lambda_1,\dots ,\lambda _m\in\mathbb{C}$, and let us notice first that the second inequality in (\ref{MNPrimaDeTot}) implies that $|Nb_0|>|Ma_0|$, as $a_nb_m\neq 0$, so the constant term $Ma_0+Nb_0$ of the polynomial $Mf+Ng$ can not vanish for our choice of $M$ and $N$. Next, since $m>n$, we have $\deg (Mf+Ng)=\deg (g)=m$, so we deduce that
\begin{eqnarray*}
\bigl|Res\bigl(f,Mf+Ng\bigr)\bigr| & = & |a_{n}|^{m}\prod%
\limits_{i=1}^{n}\bigl|\bigl(Mf+Ng\bigr)(\theta _{i})\bigr|=|a_{n}|^{m}|N|^n\prod%
\limits_{i=1}^{n}\bigl|g(\theta _{i})\bigr|\\
& = & |N|^n|Res(f,g)|=pq|N|^n.
\end{eqnarray*}
To apply Corollary \ref{coro1thm0} to the pair of polynomials $(f,Mf+Ng)$, it suffices to check that our assumptions on $|M|$ and $|N|$ will force $\lambda _i$ and $\theta _j$ to satisfy 
\[
\min\limits _{i,j}|\lambda _i-\theta_j|>(q|N|^n)^{\frac{1}{\min \{m,n\}}}=q^{\frac{1}{n}}|N|.
\]
As $|\theta _j|<A$ for each $j$, it suffices to ask $\lambda _i$ to satisfy $|\lambda _i|\geq A+q^{\frac{1}{n}}|N|$ for each $i$. Since 
\[
Mf(X)+Ng(X)=\sum_{i=0}^{n}\bigl(Ma_{i}+Nb_{i}\bigr)X^{i}+\sum_{i=n+1}^{m}Nb_{i}X^{i},
\]
by Lemma \ref{lemamica} it will be sufficient to prove that 
\[
|Ma_0+Nb_0|\geq\sum_{i=1}^{n}|Ma_i+Nb_i|(A+q^{\frac{1}{n}}|N|)^i+\sum_{i=n+1}^{m}|Nb_i|(A+q^{\frac{1}{n}}|N|)^i.
\]
By the triangle inequality, it will be sufficient to prove that
\begin{equation}\label{intermediary}
|Nb_0|-|Ma_0|\geq\sum_{i=1}^{m}|Nb_i|(A+q^{\frac{1}{n}}|N|)^i+\sum_{i=1}^{n}|Ma_i|(A+q^{\frac{1}{n}}|N| )^i,
\end{equation}
which obviously holds, in view of (\ref{MNPrimaDeTot}). 
In our second case, by Lemma \ref{lemamica} we have $|\theta _j|\leq A$ for each $j$, so it will be sufficient to ask $\lambda _i$ to satisfy $|\lambda _i|>A+q^{\frac{1}{n}}|N|$ for each $i$, which will hold if we ask inequality (\ref{intermediary}) to be a strict one.
This completes the proof. \hfill $\square$
\medskip

{\it Proof of Corollary \ref{Coro1gradf<gradg}.} We will apply Theorem \ref{gradf<gradg} for the pair $(f,g)$ with $f(X)=aX-1$ and $A=q=1$. Thus condition $B\geq A+q^{\frac{1}{n}}$ reduces to $B\geq 2$, condition $|a_{n}|>\sum_{i=0}^{n-1}|a_{i}|A^{i-n}$ reduces to $|a|>1$, and condition $|b_{0}|\geq 1+|a|B+\sum_{i=1}^{m}|b_{i}|B^{i}$ implies that
\[
\frac {|b_{0}|-\sum\limits _{i=1}^{m}|b_{i}|(A+q^{\frac{1}{n}}|N|)^{i}} {\sum\limits_{i=0}^{n}|a_{i}|(A+q^{\frac{1}{n}}|N|)^{i}}=\frac {|b_{0}|-\sum\limits _{i=1}^{m}|b_{i}|(1+|N|)^{i}} {1+|a|(1+|N|)}\geq\frac {|b_{0}|-\sum\limits _{i=1}^{m}|b_{i}|B^{i}} {1+|a|B}\geq 1
\]
for $|N|\leq B-1$, so instead of the second inequality in (\ref{MNPrimaDeTot}) it suffices to ask here $|M|\leq |N|$, while for $|a|=1$ it suffices to ask $|M|<|N|$. This completes the proof.
\hfill $\square$

\medskip

{\bf Case 2:}\ $\deg f=\deg g$\quad 
Here we will first prove the following result.
\begin{theorem}
\label{thm7EqualDegreesGeneralaDeTot} Let $f(X)=\sum_{i=0}^{n}a_{i}X^{i}$,$\
g(X)=\sum_{i=0}^{n}b_{i}X^{i}\in \mathbb{Z}[X]$, $a_{0}a_{n}b_{0}b_{n}\neq 0$%
, and suppose that $|Res(f,g)|=pq$ with $p$ a prime number, and $q$ a positive integer. Suppose also that for two positive real numbers $A,B$ with $B\geq A+q^{\frac{1}{n}}$ we have 
$|a_{n}|>\sum_{i=0}^{n-1}|a_{i}|A^{i-n}$ and $|b_{0}|>\sum_{i=1}^{n}|b_{i}|B^{i}$. Let $M$ and $N$ be integers with  $Ma_n+Nb_n\neq 0$ and $Ma_0+Nb_0\neq 0$, and satisfying one of the following two pairs of inequalities:
\vspace{-0.6cm}

\begin{eqnarray}
|M| & \leq & \frac{B-A}{q^{\frac{1}{n}}}\quad and \quad |N|\leq|M|\cdot \frac {|a_{n}|-\sum\limits _{i=0}^{n-1}|a_{i}|(B-q^{\frac{1}{n}}|M|)^{i-n}} {\sum\limits_{i=0}^{n}|b_{i}|(B-q^{\frac{1}{n}}|M|)^{i-n}}\label{MNLungaEqualDegree} \\
|N| & \leq & \frac{B-A}{q^{\frac{1}{n}}} \quad and \quad |M|\leq|N|\cdot \frac {|b_{0}|-\sum\limits _{i=1}^{n}|b_{i}|(A+q^{\frac{1}{n}}|N|)^{i}} {\sum\limits_{i=0}^{n}|a_{i}|(A+q^{\frac{1}{n}}|N|)^{i}}\label{MNLunga2EqualDegree}
\end{eqnarray}
Then the polynomial $Mf(X)+Ng(X)$ is irreducible over $\mathbb{Q}$.
\end{theorem}
\begin{remark}\label{remrem1}
In the statement of Theorem \ref{thm7EqualDegreesGeneralaDeTot} we may allow equality in only one of the inequalities $|a_{n}|>\sum_{i=0}^{n-1}|a_{i}|A^{i-n}$ and $|b_{0}|>\sum_{i=1}^{n}|b_{i}|B^{i}$. Thus, if we allow $|a_{n}|$ to be equal to $\sum_{i=0}^{n-1}|a_{i}|A^{i-n}$, we must ask the second inequality in (\ref{MNLunga2EqualDegree}) to be a strict one, while if we allow $|b_{0}|$ to be equal to $\sum_{i=1}^{n}|b_{i}|B^{i}$, we must ask the second inequality in (\ref{MNLungaEqualDegree}) to be a strict one.
\end{remark}
In particular, one obtains simpler irreducibility conditions for $q=1$ and $B=A+1$:

\begin{corollary}
\label{thm7EqualDegrees} Let $f(X)=\sum_{i=0}^{n}a_{i}X^{i}$,$\
g(X)=\sum_{i=0}^{n}b_{i}X^{i}\in \mathbb{Z}[X]$, $a_{0}a_{n}b_{0}b_{n}\neq 0$, and suppose that for a positive real number $A$ we have 
$|a_{n}|>\sum_{i=0}^{n-1}|a_{i}|A^{i-n}$ and $|b_{0}|>\sum_{i=1}^{n}|b_{i}|(A+1)^{i}$. 
If $|Res(f,g)|$ is a prime number, then for all integers $M,N$ with $Ma_n+b_n\neq 0$ and $Ma_0+b_0\neq 0$, and such that
\[
|M|\leq\frac {|b_{0}|-\sum\limits _{i=1}^{n}|b_{i}|(A+1)^{i}} {\sum\limits_{i=0}^{n}|a_{i}|(A+1)^{i}}\quad and\quad |N|\leq\frac {|a_{n}|-\sum\limits _{i=0}^{n-1}|a_{i}|A^{i-n}} {\sum\limits_{i=0}^{n}|b_{i}|A^{i-n}},
\]
the polynomials $Mf(X)\pm g(X)$ and $\pm f(X)+Ng(X)$ are irreducible over $\mathbb{Q}$.
\end{corollary}

{\it Proof of Theorem \ref{thm7EqualDegreesGeneralaDeTot}.} Here we first mention that the left inequalities in (\ref{MNLungaEqualDegree}) and (\ref{MNLunga2EqualDegree}) combined with the inequalities $|a_{n}|>\sum_{i=0}^{n-1}|a_{i}|A^{i-n}$ and $|b_0|>\sum_{i=1}^{m}|b_{i}|B^{i}$, respectively, show that the numerators in the corresponding right inequalities are positive, so these inequalities make sense too.
Assume first that $M$ and $N$ satisfy inequalities (\ref{MNLungaEqualDegree}), and that $f$ and $g$ factorize as $f(X)=a_{n}(X-\theta _{1})\cdots (X-\theta _{n})$ and $%
g(X)=b_{n}(X-\xi _{1})\cdots (X-\xi _{n})$ with $\theta _{1},\ldots ,\theta
_{n}$, $\xi _{1},\ldots ,\xi _{n}\in \mathbb{C}$.
We will apply Corollary \ref{coro1thm0} to the pair $(Mf+Ng,g)$. In this case we will prove that our assumption on the magnitudes of $M$ and $N$ actually prevents $Ma_{n}+Nb_{n}$ from vanishing, which will ensure us that $\deg(Mf+Ng)=n$ (this will no longer be the case if $M$ and $N$ satisfy (\ref{MNLunga2EqualDegree}), when we will have to assume that $Ma_{n}+Nb_{n}\neq 0$). Indeed, (\ref{MNLungaEqualDegree}) implies that
\[
\left| Ma_{n}\right| \geq|M|\sum_{i=0}^{n-1}\left| a_{i}\right| (B-q^{\frac{1}{n}}|M|)^{i-n}+|N|\sum_{i=0}^{n} \left| b_{i}\right| (B-q^{\frac{1}{n}}|M|)^{i-n}>|Nb_n|,
\]
so $Ma_{n}+Nb_{n}$ can not be zero. This further implies that
\[
\bigl|Res\bigl(Mf+Ng,g\bigr)\bigr|=|b_{n}|^{n}\prod%
\limits_{i=1}^{n}\bigl|\bigl(Mf+Ng\bigr)(\xi _{i})\bigr|%
=|M|^n|Res(f,g)|=pq|M|^n.
\]

In order to check that the roots $\lambda _i$
of $Mf+Ng$ lie in the disk $\{z:|z|\leq B-q^{\frac{1}{n}}|M|\}$, it suffices to prove that
\[
\left| Ma_{n}+Nb_{n}\right| \geq\sum_{i=0}^{n-1}\bigl|Ma_{i}+Nb_{i}\bigr|(B-q^{\frac{1}{n}}|M|) ^{i-n}.
\]
By the triangle inequality it will be sufficient to prove that
\[
\left| Ma_{n} \right| - \left| Nb_{n}\right| \geq\sum_{i=0}^{n-1}\left| Ma_{i}\right|(B-q^{\frac{1}{n}}|M|) ^{i-n} +|N|\sum_{i=0}^{n-1}\left| b_{i}\right| (B-q^{\frac{1}{n}}|M|) ^{i-n},
\]
and this obviously holds, according to (\ref{MNLungaEqualDegree}).

Assume now that $M$ and $N$ satisfy the pair of inequalities (\ref{MNLunga2EqualDegree}).  Here our assumption that $Ma_{n}+Nb_{n}\neq 0$ implies that we have $\deg(Mf+Ng)=n$, and we deduce that
\[
\bigl|Res\bigl(f,Mf+Ng\bigr)\bigr|=|a_{n}|^{n}\prod%
\limits_{i=1}^{n}\bigl|\bigl(Mf+Ng\bigr)(\theta _{i})\bigr|%
=|N|^n|Res(f,g)|=pq|N|^n.
\]
To apply Corollary \ref{coro1thm0} to the pair of polynomials $(f,Mf+Ng)$, it suffices to check that all the roots $\lambda _i$ of $Mf+Ng$ satisfy $|\lambda _i|\geq A+q^{\frac{1}{n}}|N|$. By Lemma \ref{lemamica} it suffices to prove that 
\[
|Ma_0+Nb_0|\geq\sum_{i=1}^{n}|Ma_i+Nb_i|(A+q^{\frac{1}{n}}|N|)^i.
\]
Using again the triangle inequality, it will be sufficient to prove that
\[
|Nb_0|-|Ma_0|\geq\sum_{i=1}^{n}(|Ma_i|+|Nb_i|)(A+q^{\frac{1}{n}}|N|)^i,
\]
which obviously holds, in view of (\ref{MNLunga2EqualDegree}).
This completes the proof. \hfill $\square $
\medskip

We will also prove here another result for the case when $\deg f=\deg g$, which provides simpler conditions on $M$ and $N$, at the cost of enlarging the required distance between the roots of $f$ and those of $g$:
\medskip

\begin{theorem}
\label{thmCuMarden} Let $f(X)=\sum_{i=0}^{n}a_{i}X^{i}$,$\
g(X)=\sum_{i=0}^{n}b_{i}X^{i}\in \mathbb{Z}[X]$, $a_{0}a_{n}b_{0}b_{n}\neq 0$%
, and suppose that $|Res(f,g)|=pq$ with $p$ a prime number, and $q$ a positive integer. Suppose also that for two positive real numbers $A,B$ with $B\geq 3A+2q^{\frac{1}{n}}$ we have 
$|a_{n}|>\sum_{i=0}^{n-1}|a_{i}|A^{i-n}$ and $|b_{0}|>\sum_{i=1}^{n}|b_{i}|B^{i}$. Then the polynomial $Mf(X)+Ng(X)$ is irreducible over $\mathbb{Q}$ for each integers $M$ and $N$ with   $|M|\cdot\frac{|a_n|}{|b_n|}<|N|\leq\frac{B-3A}{2q^{\frac{1}{n}}}$.
\end{theorem}

In particular, for $q=1$ and $|a_n|=|b_n|$ one obtains the following simpler result:
\begin{corollary}
\label{corothmCuMarden} Let $f(X)=\sum_{i=0}^{n}a_{i}X^{i}$,$\
g(X)=\sum_{i=0}^{n}b_{i}X^{i}\in \mathbb{Z}[X]$, $a_{0}a_{n}b_{0}b_{n}\neq 0$, with $|a_n|=|b_n|$, and suppose that $|Res(f,g)|$ is a prime number. Suppose also that for two positive real numbers $A,B$ with $B\geq 3A+2$ we have 
$|a_{n}|>\sum_{i=0}^{n-1}|a_{i}|A^{i-n}$ and $|b_{0}|>\sum_{i=1}^{n}|b_{i}|B^{i}$. Then the polynomial $Mf(X)+Ng(X)$ is irreducible over $\mathbb{Q}$ for all integers $M$ and $N$ with $|M|<|N|\leq\frac{B-3A}{2}$.
\end{corollary}

{\it Proof of Theorem \ref{thmCuMarden}.} First of all, the irreducibility of $f$ and $g$ follows by Corollary \ref{coro1thm0} with $q=1$, so we may assume that $M\neq 0$. Next, we observe that the hypotheses in the statement imply the following sequence of inequalities:
\[
|a_0|<|a_n|A^n;\ \ |Ma_n|<|Nb_n|;\ \  A<B;\ \ |b_n|B^n<|b_0|,
\]
from which we successively deduce that
\[
|Ma_0|<|Ma_n|A^n<|Nb_n|A^n<|Nb_n|B^n<|Nb_0|,
\]
so the constant term of the polynomial $Mf+Ng$ cannot vanish when  $M$ and $N$ are not both zero. In our proof we will need the following result of Marden on the location of the roots of a linear combination of two polynomials:
\medskip

{\bf Theorem} (Marden \cite[Th. 17,2b]{Marden}) {\em If each zero of $f_1(z)=z^n+a_1z^{n-1}+\cdots +a_n$ lies on or outside circle $C_1$ of center $c_1$ and radius $r_1$, if each zero of $f_2(z)=z^n+b_1z^{n-1}+\cdots +b_n$ lies in or on circle $C_2$ of center $c_2$ and radius $r_2$, and if $r_1>r_2|\lambda |^{\frac{1}{n}}$, then each zero of $h(z)=f_1(z)-\lambda f_2(z)$, $\lambda \neq 1$ lies on or outside at least one of the circles $\Gamma_k$ with center $\gamma _k=\frac{c_1-\omega _kc_2}{1-\omega _k}$ and radius $\gamma _k=\frac{r_1-|\omega _k|r_2}{|1-\omega _k|}$, where $\omega _k$ ($k=1,\dots ,n$) are the $n$th roots of $\lambda $.
}
\medskip

Here $f_1$ and $f_2$ have complex coefficients, and $c_1,c_2,\gamma_k$ and $\lambda$ are complex numbers. To use this theorem in our case, we write $Mf+Ng=Nb_n(f_1-\lambda f_2)$ with 
\[
f_1(X)=\prod_{i=1}^n(X-\xi_i),\ f_2(X)=\prod_{i=1}^n(X-\theta_i)\ \ \text{and} \ \ \lambda=-\frac{Ma_n}{Nb_n},
\]
with $\theta _i$ and $\xi _i$ the roots of $f$ and $g$, respectively.
Since $|Ma_n|<|Nb_n|$ we have $|\lambda|<1$, so in particular $\lambda\neq 1$ and $Ma_n+Nb_n\neq 0$. By Lemma \ref{lemamica} we deduce that $|\theta _i|< A$ and $|\xi_i|> B$ for $i=1,\dots,n$, so by Marden's Theorem with $c_1=c_2=0$ we conclude that each root $\lambda _i$ of $Mf+Ng$ lies on or outside at least one of the circles $\Gamma_k$ with center $0$ and radius $\gamma _k=\frac{r_1-|\omega _k|r_2}{|1-\omega _k|}$, where $\omega _k$ ($k=1,\dots ,n$) are the $n$th roots of $\lambda $. In particular, all the $\lambda _i$'s lie on or outside the circle centered at $0$ with radius $\min\limits_{k}\gamma _k$. To avoid the use of trigonometric functions in the denominators of $\gamma_k$, we will content ourselves with the conclusion that all the $\lambda _i$'s satisfy
\begin{equation}\label{ModulLambdai}
|\lambda_i|\geq \frac{B-(\frac{|Ma_n|}{|Nb_n|})^{\frac{1}{n}}A}{1+(\frac{|Ma_n|}{|Nb_n|})^{\frac{1}{n}}}.
\end{equation}
Now, since $\deg(Mf+Ng)=n$, we deduce that
\[
\bigl|Res\bigl(f,Mf+Ng\bigr)\bigr|=|a_{n}|^{n}\prod%
\limits_{i=1}^{n}\bigl|\bigl(Mf+Ng\bigr)(\theta _{i})\bigr|%
=|N|^n|Res(f,g)|=pq|N|^n,
\]
so to apply Corollary \ref{coro1thm0} to the pair of polynomials $(f,Mf+Ng)$, it suffices to check that all the roots $\lambda _i$ of $Mf+Ng$ have absolute values exceeding $A+q^{\frac{1}{n}}|N|$. It is therefore sufficient to ask the quotient in (\ref{ModulLambdai}) to exceed $A+q^{\frac{1}{n}}|N|$, or equivalently, to have
\[
B>\left(1+2\left(\frac{|Ma_n|}{|Nb_n|}\right)^{\frac{1}{n}}\right)A+q^{\frac{1}{n}}|N|\left(1+\left(\frac{|Ma_n|}{|Nb_n|}\right)^{\frac{1}{n}}\right).
\]
Since $|Ma_n|<|Nb_n|$, it suffices to ask $B\geq 3A+2q^{\frac{1}{n}}|N|$, which completes the proof. \hfill $\square $
\medskip

{\bf Case 3:}\ $\deg f>\deg g$\quad 
For this remaining case we will prove the following result:

\begin{theorem}
\label{thm7GeneralaDeTot} Let $f(X)=\sum_{i=0}^{n}a_{i}X^{i}$,$\
g(X)=\sum_{i=0}^{m}b_{i}X^{i}\in \mathbb{Z}[X]$, $a_{0}a_{n}b_{0}b_{m}\neq 0$%
, $n>m$, and suppose that $|Res(f,g)|=pq$ with $p$ a prime number, and $q$ a positive integer. Suppose also that for two positive real numbers $A,B$ with $B\geq A+q^{\frac{1}{m}}$ we have 
$|a_{n}|>\sum_{i=0}^{n-1}|a_{i}|A^{i-n}$ and $|b_{0}|>\sum_{i=1}^{m}|b_{i}|B^{i}$. Let $M$ and $N$ be integers with $Ma_0+Nb_0\neq 0$ and satisfying one of the following two pairs of inequalities:
\vspace{-0.6cm}

\begin{eqnarray}
|M| & \leq & \frac{B-A}{q^{\frac{1}{m}}}\quad \quad \ \thinspace and \quad |N|\leq|M|\cdot \frac {|a_{n}|-\sum\limits _{i=0}^{n-1}|a_{i}|(B-q^{\frac{1}{m}}|M|)^{i-n}} {\sum\limits_{i=0}^{m}|b_{i}|(B-q^{\frac{1}{m}}|M|)^{i-n}}\label{MNLunga} \\
|N|^{\frac{n}{m}} & \leq & \frac{B-A}{q^{\frac{1}{m}}|a_n|^{\frac{n-m}{m}}} \quad and \quad |M|\leq|N|\cdot \frac {|b_{0}|-\sum\limits _{i=1}^{m}|b_{i}|(A+q^{\frac{1}{m}}|a_n|^{\frac{n-m}{m}}|N|^{\frac{n}{m}})^{i}} {\sum\limits_{i=0}^{n}|a_{i}|(A+q^{\frac{1}{m}}|a_n|^{\frac{n-m}{m}}|N|^{\frac{n}{m}})^{i}}.\label{MNLunga2}
\end{eqnarray}
\vspace{-0.4cm}

\noindent
Then the polynomial $Mf(X)+Ng(X)$ is irreducible over $\mathbb{Q}$.
\end{theorem}

\begin{remark}\label{remrem2}
As in the case of Theorem \ref{thm7EqualDegreesGeneralaDeTot}, in  Theorem \ref{thm7GeneralaDeTot} too we may allow equality in only one of the inequalities $|a_{n}|>\sum_{i=0}^{n-1}|a_{i}|A^{i-n}$ and $|b_{0}|>\sum_{i=1}^{n}|b_{i}|B^{i}$. Thus, if we allow $|a_{n}|$ to be equal to $\sum_{i=0}^{n-1}|a_{i}|A^{i-n}$, we must ask the second inequality in (\ref{MNLunga2}) to be a strict one, while if we allow $|b_{0}|$ to be equal to $\sum_{i=1}^{n}|b_{i}|B^{i}$, we must ask the second inequality in (\ref{MNLunga}) to be a strict one.
\end{remark}
A simple and useful instance of Theorem \ref{thm7GeneralaDeTot} is the following result:
\begin{corollary}
\label{cevacuf(b)} Let $f(X)=\sum_{i=0}^{n}a_{i}X^{i}\in \mathbb{Z}[X]$, with $a_{0}a_{n}\neq 0$, and assume that $|a_{n}|\geq 1+|b|+\sum_{i=0}^{n-1}|a_{i}|$ for some integer $b$ with $|b|\geq 2$. If $|f(b)|$ is a prime number, then for each integers $M,N$ with $Ma_0-Nb\neq 0$ and
\[
|N|<|M|\leq |b|-1,
\]
the polynomial $Mf(X)+N(X-b)$ is irreducible over $\mathbb{Q}$. 
\end{corollary}

Conditions (\ref{MNLunga}) and (\ref{MNLunga2}) drastically simplify for $B=A+1$, $q=1$ and $|a_n|=1$, as in the following two corrolaries.

\begin{corollary}
\label{thm7} Let $f(X)=\sum_{i=0}^{n}a_{i}X^{i}$,$\
g(X)=\sum_{i=0}^{m}b_{i}X^{i}\in \mathbb{Z}[X]$, $a_{0}a_{n}b_{0}b_{m}\neq 0$%
, $n>m$. Suppose that $|Res(f,g)|$ is a prime number, and that for a real $A>0$ we have 
$|a_{n}|>\sum_{i=0}^{n-1}|a_{i}|A^{i-n}$ and $|b_{0}|>\sum_{i=1}^{m}|b_{i}|(A+1)^{i}$. 
Then for each integer $N\neq -\frac{a_0}{b_0}$ with 
\[
|N|\leq\frac {|a_{n}|-\sum\limits _{i=0}^{n-1}|a_{i}|A^{i-n}} {\sum\limits_{i=0}^{m}|b_{i}|A^{i-n}},
\]
the polynomial $f(X)+Ng(X)$ is irreducible over $\mathbb{Q}$. 
\end{corollary}

\begin{corollary}
\label{thm7split} Let $f(X)=\sum_{i=0}^{n}a_{i}X^{i}$,$\
g(X)=\sum_{i=0}^{m}b_{i}X^{i}\in \mathbb{Z}[X]$, with $|a_n|=1$, $a_{0}b_{0}b_{m}\neq 0$%
, $n>m$. Suppose that $|Res(f,g)|$ is a prime number, and that for a real $A>0$ we have 
$1>\sum_{i=0}^{n-1}|a_{i}|A^{i-n}$ and $|b_{0}|>\sum_{i=1}^{m}|b_{i}|(A+1)^{i}$. 
Then for each integer $M$ with 
\[
|M|\leq\frac {|b_{0}|-\sum\limits _{i=1}^{m}|b_{i}|(A+1)^{i}} {\sum\limits_{i=0}^{n}|a_{i}|(A+1)^{i}},
\]
the polynomial $Mf(X)+g(X)$ is irreducible over $\mathbb{Q}$. 
\end{corollary}

{\it Proof of Theorem \ref{thm7GeneralaDeTot}.} Assume that $M$ and $N$ satisfy inequalities (\ref{MNLunga}). First of all we notice that since $|a_{n}|>\sum_{i=0}^{n-1}|a_{i}|A^{i-n}$, the numerator in the second inequality in (\ref{MNLunga}) is positive for $|M|\leq \frac{B-A}{q^{\frac{1}{m}}}$, so the inequality makes sense.
Assume now that $f$ and $g$
factorize as $f(X)=a_{n}(X-\theta _{1})\cdots (X-\theta _{n})$ and $g(X)=b_{m}(X-\xi _{1})\cdots (X-\xi _{m})$ with $\theta _{1},\ldots ,\theta_{n}$, $\xi _{1},\ldots ,\xi _{m}\in \mathbb{C}$. By Lemma \ref{lemamica} we see that $|\theta _{i}|<A$ for $i=1,\ldots ,n$ and $|\xi _{j}|>B$ for
$j=1,\ldots ,m$, so by Corollary \ref{coro1thm0} we deduce that both $f$
and $g$ must be irreducible over $\mathbb{Q}$. 
We may thus assume that none of $M$ and $N$ is zero.
Now, since $m<n$, we have $\deg
(Mf+Ng)=\deg (f)=n$, so 
\begin{equation}\label{NoulRezultant}
\bigl|Res\bigl(Mf+Ng,g\bigr)\bigr|=|b_{m}|^{n}\prod%
\limits_{i=1}^{m}\bigl|\bigl(Mf+Ng\bigr)(\xi _{i})\bigr|%
=|M|^m|Res(f,g)|=pq|M|^m.
\end{equation}
Let us assume that $Mf+Ng$ decomposes as $Ma_n(X-\lambda _1)\cdots (X-\lambda _n)$ with $\lambda_1,\dots ,\lambda _n\in\mathbb{C}$. In view of (\ref{NoulRezultant}), to apply Corollary \ref{coro1thm0} to our pair of polynomials $(Mf+Ng,g)$, all that remains to prove
is that our assumptions on $|M|$ and $|N|$ will force $\lambda _i$ to satisfy 
\[
\min\limits _{i,j}|\lambda _i-\xi_j|>(q|M|^m)^{\frac{1}{\min \{m,n\}}}=q^{\frac{1}{m}}|M|.
\]
As $|\xi _j|>B$ for each $j$, it suffices to ask $\lambda _i$ to satisfy $|\lambda _i|\leq B-q^{\frac{1}{m}}|M|$ for each $i$. Now, since 
\[
Mf(X)+Ng(X)=\sum_{i=0}^{m}\bigl(Ma_{i}+Nb_{i}\bigr)X^{i}+\sum_{i=m+1}^{n}Ma_{i}X^{i},
\]
by Lemma \ref{lemamica} it will be sufficient to prove that 
\[
\left| Ma_{n}\right| \geq\sum_{i=0}^{m}\bigl|Ma_{i}+Nb_{i}\bigr|(B-q^{\frac{1}{m}}|M|)^{i-n}+\sum_{m<i<n}|Ma_{i}|(B-q^{\frac{1}{m}}|M|) ^{i-n},
\]
the right-most sum being zero if $m=n-1$.
In particular, by the triangle inequality this will be satisfied if 
\[
\left|M a_{n}\right| \geq|M|\sum_{i=0}^{n-1}\left| a_{i}\right|(B-q^{\frac{1}{m}}|M|)^{i-n} +|N|\sum_{i=0}^{m}\left| b_{i}\right|(B-q^{\frac{1}{m}}|M|)^{i-n},
\]
and this last inequality obviously holds, according to (\ref{MNLunga}).

Assume now that $M$ and $N$ satisfy inequalities (\ref{MNLunga2}). Here, since $|b_0|>\sum_{i=1}^{m}|b_{i}|B^{i}$, we observe that the numerator in the second inequality in (\ref{MNLunga2}) is positive for $|N|^{\frac{n}{m}}\leq \frac{B-A}{q^{\frac{1}{m}}|a_n|^{\frac{n-m}{m}}}$, so this inequality makes sense too.
We then deduce that
\begin{eqnarray*}
\bigl|Res\bigl(f,Mf+Ng\bigr)\bigr| & = & |a_{n}|^{n}\prod%
\limits_{i=1}^{n}\bigl|\bigl(Mf+Ng\bigr)(\theta _{i})\bigr|=|a_{n}|^{n}|N|^n\prod%
\limits_{i=1}^{n}\bigl|g(\theta _{i})\bigr|\\
& = & |a_{n}|^{n-m}|N|^n|Res(f,g)|=pq|a_{n}|^{n-m}|N|^n.
\end{eqnarray*}
To apply Corollary \ref{coro1thm0} to the pair of polynomials $(f,Mf+Ng)$, it suffices to check that all the roots $\lambda _i$ of $Mf+Ng$ satisfy $|\lambda_i|\geq A+\delta$ with $\delta=q^{\frac{1}{m}}|a_{n}|^{\frac{n-m}{m}}|N|^{\frac{n}{m}}$. By Lemma \ref{lemamica} it suffices to prove that 
\[
|Ma_0+Nb_0|\geq\sum_{i=1}^{m}|Ma_i+Nb_i|(A+\delta )^i+\sum_{i=m+1}^{n}|Ma_i|(A+\delta )^i.
\]
Using the triangle inequality, it will be sufficient to prove that
\[
|Nb_0|-|Ma_0|\geq\sum_{i=1}^{m}(|Ma_i|+|Nb_i|)(A+\delta )^i+\sum_{i=m+1}^{n}|Ma_i|(A+\delta )^i,
\]
which obviously holds, in view of (\ref{MNLunga2}).
This completes the proof. \hfill $\square$
\medskip 

{\it Proof of Corollary \ref{cevacuf(b)}.} We will apply Theorem \ref{thm7GeneralaDeTot} and Remark \ref{remrem2} for the pair $(f,g)$ with $g(X)=X-b$ and $A=q=1$, and $B=|b|$. Thus condition $B\geq A+q^{\frac{1}{n}}$ reduces to $|b|\geq 2$, condition $|b_{0}|\geq \sum_{i=1}^{m}|b_{i}|B^{i}$ is obviously satisfied, and condition $|a_{n}|\geq 1+|b|+\sum_{i=0}^{n-1}|a_{i}|$ implies that for $|M|\leq |b|-1$ we have
\[
\frac {|a_{n}|-\sum\limits _{i=0}^{n-1}|a_{i}|(B-q^{\frac{1}{m}}|M|)^{i-n}} {\sum\limits_{i=0}^{m}|b_{i}|(B-q^{\frac{1}{m}}|M|)^{i-n}}=\frac {|a_{n}|-\sum\limits _{i=0}^{n-1}|a_{i}|(|b|-|M|)^{i-n}} {|b|(|b|-|M|)^{-n}+(|b|-|M|)^{1-n}}\geq\frac {|a_{n}|-\sum\limits _{i=0}^{n-1}|a_{i}|} {1+|b|}\geq 1,
\]
so for the second inequality in (\ref{MNLunga}) to hold with $``<``$ instead of $``\leq``$, it suffices to have $|N|<|M|$. This completes the proof.
\hfill $\square$

\medskip

As we shall see in Section \ref{multivariate}, in the case of bivariate polynomials $(f,g)$ over a field $K$, under certain conditions on the coefficients of $f$ and $g$, the irreducibility in pairs extends to all linear combinations of the form $(f,\alpha f+\beta g)$ with $\alpha,\beta\in K$, not both zero.

\section{Irreducibility criteria in the multivariate case}\label{multivariate}

In this section we will establish several similar irreducibility conditions for pairs of multivariate polynomials $f,g$ over an arbitrary field $K$, and for some of their linear combinations, that will be deduced from more general results on the number of irreducible factors of $f$ and $g$. Here we will use the following definitions:
\begin{definition}\label{ADouaDefinitie}i) Let $K$ be a field. For a nonzero polynomial $f(X)\in K[X]$ we will denote by $\Omega(f)$ the number of its irreducible factors, counted with their multiplicities, up to multiplication by a nonzero constant in $K$.

ii) Let $f(X,Y)=\sum\nolimits_{i=0}^{n}a_{i}Y^{i}$
and $g(X,Y)=\sum\nolimits_{i=0}^{m}b_{i}Y^{i}\in K[X,Y]$, with coefficients $a_{0},\dots
,a_{n}$, $b_{0},\dots ,b_{m}\in K[X]$, $a_{0}a_{n}b_{0}b_{m}\neq 0$, and assume that $f$ and $g$ as polynomials in $Y$ are relatively prime, so $Res_{Y}(f,g)$ is a nonzero polynomial in $K[X]$. Let also $k$ be a positive integer, and let us define
\[
\delta_k=\max\bigg\lbrace \deg d:d\mid Res_{Y}(f,g)\ \ and \ \ \deg d\leq \frac{\deg Res_{Y}(f,g)}{k+1}\bigg\rbrace .
\]
Note that $\delta _k$ defines a decreasing, eventually constant sequence, that may be considered as the analogue for the bivariate case of the sequence $d_k$ in Definition \ref{PrimaDefinitie}.
\end{definition}
The first result that we will prove relies on the properties of $\delta _k$, and provides conditions on the degrees of the coefficients of $f$ and $g$ that guarantee us that each one of $f$ and $g$ is a product of at most $k$ irreducible factors in $K[X,Y]$. Taking also into account Remark \ref{RemarcaMultivar}, we may regard this result as a non-Archimedean version of Theorem \ref{thm0dk}.
\begin{theorem}
\label{thm11generala} Let $K$ be a field, $f(X,Y)=\sum\nolimits_{i=0}^{n}a_{i}Y^{i},%
\ g(X,Y)=\sum\nolimits_{i=0}^{m}b_{i}Y^{i}\in K[X,Y]$, with $a_{0},\dots
,a_{n},b_{0},\dots ,b_{m}\in K[X]$, $a_{0}a_{n}b_{0}b_{m}\neq 0$, and assume that $f$ and $g$ have no nonconstant factors in $K[X]$. Let $k$ be a positive integer, and let
\begin{eqnarray*}
A & = & \deg a_0-\max\{\deg a_1,\dots, \deg a_n\},\ \  and\\
B & = & \max\{\deg b_0,\dots, \deg b_{m-1}\}-\deg b_m.
\end{eqnarray*}
If $\min\{A,\frac{A}{n}\}>\max\{B,\frac{B}{m}\}$ and $\min\{\deg a_0, \deg b_m\}> \delta_k$, 
then each one of the polynomials $f$ and $g$ is a product of at most $k$ irreducible factors in $K[X,Y]$.
\end{theorem}
\begin{remark}\label{RemarcaMultivar} As we shall see in the proof, the condition $\min\{A,A/n\}>\max\{B,B/m\}$ in  Theorem \ref{thm11generala} forces $f$ and $g$ to share no common root in $\overline{K(X)}$, and hence to be relatively prime as polynomials in $Y$, as required in the definition of $\delta_k$. 
Let $\overline{f}=Y^nf(X,\frac{1}{Y})=\sum\nolimits_{i=0}^{n}a_{n-i}Y^{i}$ and $\overline{g}=Y^ng(X,\frac{1}{Y})=\sum\nolimits_{i=0}^{m}b_{m-i}Y^{i}$, the reciprocals of $f$ and $g$ with respect to $Y$. Note that the irreducible factors of $\overline{f}$ and $\overline{g}$ are precisely the reciprocals with respect to $Y$ of the irreducible factors of $f$ and $g$, respectively, and we also have $Res_Y(\overline{f},\overline{g})=(-1)^{mn}Res_Y(f,g)$, which implies that $\delta _k$ will not change if we replace $(f,g)$ by $(\overline{f},\overline{g})$. Thus the same conclusion in Theorem \ref{thm11generala} will hold if we replace the conditions $\min\{A,A/n\}>\max\{B,B/m\}$ and $\min\{\deg a_0, \deg b_m\}> \delta_k$ with
$\min\{\overline{A},\overline{A}/n\}>\max\{\overline{B},\overline{B}/m\}$ and $\min\{\deg a_n, \deg b_0\}> \delta_k$, respectively, where
\begin{eqnarray*}
\overline{A} & = & \deg a_n-\max\{\deg a_0,\dots, \deg a_{n-1}\},\ \  and\\
\overline{B} & = & \max\{\deg b_1,\dots, \deg b_{m}\}-\deg b_0.
\end{eqnarray*}
\end{remark}

In the proofs of the results in this section we will use, as in \cite{CVZ2}, a non-Archimedean absolute value $|\cdot |_\rho$ on $K(X)$, introduced as follows. We first fix an arbitrarily chosen real number $\rho >1$, and for any polynomial $f(X)\in K[X]$ we define $|f(X)|_\rho$ as
\[
|f(X)|_\rho=\rho ^{\deg f(X)},
\]
which in particular shows that for any nonzero $f\in K[X]$ we have $|f(X)|_\rho\geq 1$.
We then extend our absolute value $|\cdot |_\rho$ to $K(X)$ by multiplicativity.
Thus for any $h(X)\in K(X)$, $h(X)=\frac{f(X)}{g(X)}$, with $f(X),g(X)\in
K[X]$, $g(X)\neq 0$, we let $|h(X)|_\rho=\frac{|f(X)|_\rho}{|g(X)|_\rho}$. Let then $\overline{K(X)}$ be a fixed algebraic closure of $K(X)$, and let us fix an extension of our absolute value $|\cdot |_\rho$ to $\overline{K(X)}$ (see \cite{EnglerPrestel}, for instance), which will be also denoted by $|\cdot |_\rho$.

To study the location of the roots $\theta \in \overline{K(X)}$ of a bivariate polynomial $f(X,Y)\in K[X,Y]$, regarded as a polynomial in $Y$ with coefficients in $K[X]$, we will use this absolute value $|\cdot |_\rho$, and we will prove the following lemma, that provides non-Archimedean versions of Cauchy's bounds on the roots of univariate complex polynomials. 
\begin{lemma}\label{CauchyNonArchimedean}
{\em ({\em Cauchy - the non-Archimedean case})}\ Let $f(X,Y)=\sum\nolimits_{i=0}^{n}a_{i}Y^{i}\in K[X,Y]$, with $a_{0},\dots
,a_{n}\in K[X]$, $a_{0}a_{n}\neq 0$, and let $\theta \in \overline{K(X)}$ be a root of $f$. Then
\begin{equation}\label{Cauchy}
|\theta|_\rho\leq \max\bigg\lbrace \frac{\max\{|a_0|_\rho,\dots,|a_{n-1}|_\rho\}}{|a_n|_\rho},\left(\frac{\max\{|a_0|_\rho,\dots,|a_{n-1}|_\rho\}}{|a_n|_\rho}\right)^{\frac{1}{n}}\bigg\rbrace 
\end{equation}
and 
\begin{equation}\label{CauchyReversed}
|\theta|_\rho\geq \min\bigg\lbrace \frac{|a_0|_\rho}{\max\{|a_1|_\rho,\dots,|a_{n}|_\rho\}},\left(\frac{|a_0|_\rho}{\max\{|a_1|_\rho,\dots,|a_{n}|_\rho\}}\right)^{\frac{1}{n}}\bigg\rbrace .
\end{equation}
\end{lemma}
\begin{proof}
Since $\theta $ is a root of $f$, we deduce that
\begin{eqnarray*}
|a_n|_\rho|\theta |_\rho^{n} & = & |a_0+a_1\theta +\cdots +a_{n-1}\theta ^{n-1}|_\rho \\
 & \leq & \max\{ |a_0|_\rho, |a_1|_\rho |\theta|_\rho,\dots ,|a_{n-1}|_\rho|\theta|_\rho^{n-1}\}\\
  & \leq & \max\{ |a_0|_\rho, |a_1|_\rho,\dots ,|a_{n-1}|_\rho\}
\end{eqnarray*}
if $|\theta|_\rho \leq 1$, which yields $|\theta|_\rho\leq \left(\frac{\max\{|a_0|_\rho,\dots,|a_{n-1}|_\rho\}}{|a_n|_\rho}\right)^{\frac{1}{n}}$, while for $|\theta|_\rho >1$ we obtain 
\[
|a_n|_\rho|\theta |_\rho^{n}\leq \max\{ |a_0|_\rho, |a_1|_\rho,\dots ,|a_{n-1}|_\rho\}\cdot |\theta |_\rho^{n-1},
\]
which leads us to $|\theta|_\rho \leq \frac{\max\{|a_0|_\rho,\dots,|a_{n-1}|_\rho\}}{|a_n|_\rho}$. 

If we apply now inequality (\ref{Cauchy}) to $Y^nf(X,\frac{1}{Y})$, the reciprocal of $f$ with respect to $Y$, whose roots in $\overline{K(X)}$ are precisely the inverses of the roots of $f$, we deduce that $\theta $ also satisfies the inequality (\ref{CauchyReversed}). This completes the proof of the lemma.
\end{proof}

{\it Proof of Theorem \ref{thm11generala}.} \ Recall that for $i<j$ we have $\delta_i\geq \delta_j$, so $\delta_i$ is a decreasing sequence (not necessarily a strictly decreasing one). 
Let us assume to the contrary that $f$ decomposes as a product of more than $k$ irreducible factors over $K[X]$, say $f(X,Y)=f_{1}(X,Y)\cdots f_{\ell}(X,Y)$, with $\ell>k$, $f_{1},\dots,f_{\ell}\in K[X,Y]$, $f_{1},\dots,f_{\ell}$ irreducible over $K[X]$. Consider the factorizations of $f$ and $g$, say 
\begin{align*}
f(X,Y)& =a_{n}(Y-\theta _{1})\cdots (Y-\theta _{n}), \\
g(X,Y)& =b_{m}(Y-\xi _{1})\cdots (Y-\xi _{m}),
\end{align*}
with $\theta _{1},\ldots ,\theta _{n},\xi _{1},\ldots ,\xi _{m}\in \overline{%
K(X)}$. Since 
\[
Res_{Y}(f,g)=Res_{Y}(f_{1},g)\cdots Res_{Y}(f_{\ell},g),
\]
it follows that at least one of the polynomials $Res_{Y}(f_{1},g),\dots, Res_{Y}(f_{\ell},g)\in
K[X]$ must have degree at most $\frac{\deg Res_{Y}(f,g)}{\ell}$, say 
\[
\deg Res_{Y}(f_{1},g)\leq \frac{\deg Res_{Y}(f,g)}{\ell}.
\]
In particular this shows that $\deg Res_{Y}(f_{1},g)\leq \delta_{\ell -1}$, so 
\begin{equation}\label{PrincipiulCutiei}
|Res_{Y}(f_{1},g)|_\rho \leq \rho^{\delta_{\ell -1}}.
\end{equation}
Assume that $f_1=C_0+C_1Y+\cdots +C_tY^t$ with $t\geq 1$ and $C_0,\dots ,C_t\in K[X]$, $C_{t}\neq 0$. Since $f_{1}$ is a factor of $f$, it will
factorize over $\overline{K(X)}$ as 
$f_{1}(X,Y)=C_{t}(Y-\theta _{1})\cdots
(Y-\theta _{t})$, say, so we have 
\vspace{-1mm} 
\begin{equation}\label{formagenerala}
|Res_Y(f_{1},g)|_\rho=|C_{t}|_\rho^{m}|b_{m}|_\rho^{t}\prod\limits_{i=1}^{t}\prod%
\limits_{j=1}^{m}|\theta _{i}-\xi _{j}|_\rho.
\end{equation}
If we use now (\ref{CauchyReversed}) for the roots $\theta _i$ of $f$ together with (\ref{Cauchy}) for the roots $\xi _j$ of $g$, we observe that our assumption that $\min\{A,A/n\}>\max\{B,B/m\}$ simply implies that $|\theta_i|_\rho>|\xi_j|_\rho$ for each $i=1,\dots ,n$ and each $j=1,\dots ,m$. Since our absolute value $|\cdot|_\rho$ is non-Archimedean, we conclude that $|\theta _{i}-\xi _{j}|_\rho=|\theta _{i}|_\rho$ for each $i=1,\dots ,n$ and each $j=1,\dots ,m$. In view of (\ref{formagenerala}), and also using Vieta's formula for the product $\theta _1\cdots \theta _t$ we then obtain
\begin{eqnarray*}
|Res_Y(f_{1},g)|_\rho & = & |C_{t}|_\rho^{m}|b_{m}|_\rho^{t}\prod\limits_{i=1}^{t}\prod%
\limits_{j=1}^{m}|\theta _{i}|_\rho=|C_{t}|_\rho^{m}|b_{m}|_\rho^{t}\prod\limits_{i=1}^{t}|\theta _{i}|_\rho^{m}\\
 & = & |C_{t}|_\rho^{m}|b_{m}|_\rho^{t}\frac{|C_0|_\rho^m}{|C_t|_\rho^m}=|b_{m}|_\rho^{t}|C_0|_\rho^m.
\end{eqnarray*}
Since $C_0$ is a factor of $a_0$, we have $C_0\neq 0$, so $|C_0|_\rho\geq 1$, which implies that 
\begin{equation}\label{cevageneral}
|Res_Y(f_{1},g)|_\rho \geq |b_{m}|_\rho^{t}\geq |b_{m}|_\rho,
\end{equation}
as $t\geq 1$.
Let us recall now that one of our hypotheses was that $\min\{\deg a_0, \deg b_m\}> \delta_k$, so $\deg b_m> \delta_k$, which also reads $|b_{m}|_\rho>\rho^{\delta_k}$, and yields $|b_{m}|_\rho>\rho^{\delta_{\ell-1}}$, as $\ell -1\geq k$ and $\delta_i$ is a decreasing sequence. We have thus reached the conclusion that $|Res_Y(f_{1},g)|_\rho > \rho^{\delta_{\ell-1}}$, which contradicts (\ref{PrincipiulCutiei}), showing that $f$ is a product of at most $k$ irreducible factors over $K[X]$. 

Assuming to the contrary that $g$ decomposes as a product of more than $k$ irreducible factors over $K[X]$, say $g(X,Y)=g_{1}(X,Y)\cdots g_{\ell}(X,Y)$, with $\ell>k$, $g_{1},\dots,g_{\ell}\in K[X,Y]$, $g_{1},\dots,g_{\ell}$ irreducible over $K[X]$, we deduce as before that one of the factors, say $g_1$ satisfies
\begin{equation}\label{PrincipiulCutieiPentrug}
|Res_{Y}(f,g_1)|_\rho \leq \rho^{\delta_{\ell -1}},
\end{equation}
and if $g_{1}$ decomposes over $\overline{K(X)}$ as $g_{1}(X,Y)=D_{s}(Y-\xi _{1})\cdots (Y-\xi _{s})$, say, with $s\geq 1$, $D_{s}\in K[X]$, $D_{s}\neq 0$%
, we have \vspace{-1mm} 
\begin{eqnarray*}
|Res_Y(f,g_1)|_\rho & = & |D_{s}|_\rho^{n}|a_{n}|_\rho^{s}\prod\limits_{i=1}^{n}\prod%
\limits_{j=1}^{s}|\theta _{i}-\xi _{j}|_\rho=|D_{s}|_\rho^{n}|a_{n}|_\rho^{s}\prod\limits_{i=1}^{n}\prod%
\limits_{j=1}^{s}|\theta _{i}|_\rho\\
 & = & |D_{s}|_\rho^{n}|a_{n}|_\rho^{s}\prod\limits_{i=1}^{n}|\theta _{i}|_\rho^s=|D_{s}|_\rho^{n}|a_{n}|_\rho^{s}\frac{|a_0|_\rho^s}{|a_n|_\rho^s}=|D_{s}|_\rho^{n}|a_0|_\rho^s,
\end{eqnarray*}
again by Vieta's formula, so 
\begin{equation}\label{cevageneral2}
|Res_Y(f,g_1)|_\rho \geq |a_0|_\rho^s\geq |a_0|_\rho,
\end{equation}
since $|D_{s}|_\rho\geq 1$ and $s\geq 1$. Recalling that one of our hypotheses was that $\deg a_0>\delta_k$, which also reads $|a_{0}|_\rho>\rho^{\delta_k}$, we conclude that $|Res_Y(f,g_1)|_\rho > \rho^{\delta_k}\geq \rho^{\delta_{\ell-1}}$, which contradicts (\ref{PrincipiulCutieiPentrug}) and completes the proof.  \hfill $\square$

\begin{remark}\label{remarcamultivarAdditional}
i) Note that if we know, as an additional information, that $|C_0|_\rho>1$, then instead of (\ref{cevageneral}) we obtain $|Res_Y(f_{1},g)|_\rho >|b_{m}|_\rho$, and we will still reach the same contradiction $|Res_Y(f_{1},g)|_\rho > \rho^{\delta_{\ell-1}}$ if we allow $\deg b_m$ to be equal to $\delta_k$. So instead of condition $\min\{\deg a_0,\deg b_m\}>\delta_k$, we may ask  $\deg a_0>\delta_k$ and  $\deg b_m\geq \delta_k$. As we shall see, this situation will appear in the proof of Corollary \ref{corothm11generala}.

ii) Similarly, if we know (as an additional information) that $|D_{s}|_\rho> 1$, then instead of (\ref{cevageneral2}) we obtain $|Res_Y(f,g_1)|_\rho >|a_0|_\rho$, and we will still reach the same contradiction $|Res_Y(f,g_1)|_\rho > \rho^{\delta_{\ell-1}}$ if we allow $\deg a_0$ to be equal to $\delta_k$, so in this case, instead of the inequality $\min\{\deg a_0,\deg b_m\}>\delta_k$, we may ask $\deg a_0\geq \delta_k$ and $\deg b_m>\delta_k$. This situation will appear in the proof of Corollary \ref{corothm11generala}, when we
replace the two inequalities in its statement by $\deg a_{0}\geq\max \{\deg a_{1},\ldots ,\deg a_{n}\}$ and $\deg
b_{m}> \max \{\deg b_{0},\ldots ,\deg b_{m-1}\}$.
\end{remark}
As a corollary of the proof of Theorem \ref{thm11generala}, we obtain the following irreducibility criterion, that was stated as Theorem D in the first section of the paper.
\begin{corollary}
\label{corothm11generala} Let $K$ be a field, $f(X,Y)=\sum\nolimits_{i=0}^{n}a_{i}Y^{i},%
\ g(X,Y)=\sum\nolimits_{i=0}^{m}b_{i}Y^{i}\in K[X,Y]$, with $a_{0},\dots
,a_{n},b_{0},\dots ,b_{m}\in K[X]$, $a_{0}a_{n}b_{0}b_{m}\neq 0$, and assume that $f$ and $g$ have no nonconstant factors in $K[X]$, and that $Res_{Y}(f,g)$ is irreducible over $K$. If 
\[
\deg a_{0}>\max \{\deg a_{1},\ldots ,\deg a_{n}\}\quad and\quad \deg
b_{m}\geq \max \{\deg b_{0},\ldots ,\deg b_{m-1}\},
\]
then both $f$ and $g$ are irreducible in $K[X,Y]$. The same conclusion will also hold if we interchange the signs $>$ and $\geq$ in these two inequalities.
\end{corollary}
\begin{proof}
\ Observe that our assumption that $\deg a_{0}>\max \{\deg a_{1},\ldots ,\deg a_{n}\}$ and $\deg b_{m}\geq \max \{\deg b_{0},\ldots ,\deg b_{m-1}\}$  implies that $A>0$ and $B\leq 0$, so 
\begin{equation}\label{inegalitateaAB}
\min\{A,A/n\}=A/n>B/m=\max\{B,B/m\},
\end{equation}
and since $Res_Y(f,g)$ is irreducible over $K$, we must have $\delta_1=0$. In view of (\ref{CauchyReversed}), our assumption that $\deg a_{0}>\max \{\deg a_{1},\ldots ,\deg a_{n}\}$ implies that all the roots $\theta $ of $f$ satisfy $|\theta |_\rho >1$, so in the proof of Theorem (\ref{thm11generala}), all the roots $\theta_1,\dots ,\theta _t$ of $f_1(X,Y)$ satisfy $|\theta _i|_\rho >1$, and hence by Vieta's formula, we must have $|C_0|_\rho>|C_t|_\rho\geq 1$. Thus $|C_0|_\rho>1$, and by Remark \ref{remarcamultivarAdditional} i), to reach the conclusion in Theorem \ref{thm11generala} it suffices to have $\deg a_0>\delta_1$ and $\deg b_m\geq \delta_1$. Since $\delta_1=0$, these last two inequalities are obviously satisfied.

Consider now the case that 
\begin{eqnarray*}
\deg a_{0} & \geq & \max \{\deg a_{1},\ldots ,\deg a_{n}\}\quad {\rm and}\\
\deg b_{m} & > & \max \{\deg b_{0},\ldots ,\deg b_{m-1}\}.
\end{eqnarray*}
These inequalities imply that $A\geq 0$ and $B< 0$, and in this case (\ref{inegalitateaAB}) still holds. If we use (\ref{Cauchy}) for $g$ instead of $f$, we deduce that inequality $\deg b_{m}>\max \{\deg b_{0},\ldots ,\deg b_{m-1}\}$ forces all the roots $\xi_j$ of $g$ to satisfy $|\xi_j|_\rho<1$, so in the proof of Theorem (\ref{thm11generala}), all the roots $\xi_1,\dots ,\xi _s$ of $g_1(X,Y)$ satisfy $|\theta _i|_\rho <1$, and hence by Vieta's formula, we must have $1\leq |D_0|_\rho<|D_s|_\rho$. Thus $|D_s|_\rho>1$, and by Remark \ref{remarcamultivarAdditional} ii), to apply Theorem \ref{thm11generala} it suffices to check that $\deg a_0\geq \delta_1$ and $\deg b_m>\delta_1$, which obviously hold, since $\delta_1=0$. 
\end{proof}

We mention here that by taking $g(X,Y)=Y-1$ in Corollary \ref{corothm11generala} we obtain as a corollary the following result proved in \cite{BZ}:

\begin{corollary}
\label{coro1thm11} {\rm (}\cite[Theorem 2]{BZ}{\rm )} If we write an irreducible polynomial $f\in K[X]$ as a
sum of polynomials $f_{0},\dots,f_{n}\in K[X]$ with $f_{n}\neq 0$ and $%
\deg f_{0}>\max \{\deg f_{1},\ldots ,\deg f_{n}\}$, then the polynomial $%
F(X,Y)=\sum\nolimits_{i=0}^{n}f_{i}(X)Y^{i}\in K[X,Y]$ is irreducible in $%
K[X,Y]$.
\end{corollary}

\begin{remark}\label{remarca2}
The condition $\deg f_{0}>\max \{\deg f_{1},\ldots ,\deg f_{n}\}$
in the statement of Corollary \ref{coro1thm11} is best possible, in the sense that
there exist polynomials $f,f_{0},f_{1},\dots ,f_{n}\in K[X]$ with $f_{n}\neq 0$, $f$
irreducible over $K$, $f=f_{0}+\cdots +f_{n}$, $\deg f_{0}=\max \{\deg
f_{1},\ldots ,\deg f_{n}\}$, and such that the polynomial $%
F(X,Y)=\sum\nolimits_{i=0}^{n}f_{i}(X)Y^{i}\in K[X,Y]$ is reducible over $%
K(X)$. Indeed, let $K=\mathbb{Q}$, $m\geq 2$, and write the Eisensteinian polynomial $%
X^{m}+5X+5$ as $f_{0}+f_{1}+f_{2}+f_{3}$ with $%
f_{0}(X)=(X^{m}+1)/2$, $f_{1}(X)=5X/2+1$, $f_{2}(X)=5X/2+2$ and $%
f_{3}(X)=(X^{m}+3)/2$. Here $\deg f_{0}=\max \{\deg f_{1},\deg f_{2},\deg
f_{3}\}$, while the polynomial $F(X,Y)=\sum\nolimits_{i=0}^{3}f_{i}(X)Y^{i}$
is reducible over $\mathbb{Q}(X)$, being divisible by $Y+1$.
\end{remark}

The following result is a non-Archimedean version of Theorem \ref{thm0}.

\begin{theorem}
\label{thmGeneralaNoua} Let $K$ be a field, $f(X,Y)=\sum\nolimits_{i=0}^{n}a_{i}Y^{i},%
\ g(X,Y)=\sum\nolimits_{i=0}^{m}b_{i}Y^{i}\in K[X,Y]$, with $a_{0},\dots
,a_{n},b_{0},\dots ,b_{m}\in K[X]$, $a_{0}a_{n}b_{0}b_{m}\neq 0$, and assume that $f$ and $g$ have no nonconstant factors in $K[X]$. Let
\begin{eqnarray*}
A & = & \deg a_0-\max\{\deg a_1,\dots, \deg a_n\},\ \  and\\
B & = & \max\{\deg b_0,\dots, \deg b_{m-1}\}-\deg b_m.
\end{eqnarray*}
If $\min\{A,\frac{A}{n}\}>\max\{B,\frac{B}{m}\}$ and $\min\{\deg a_0, \deg b_m\}> \deg d$ for some divisor $d$ of $Res_Y(f,g)$, then each one of $f$ and $g$ is a product of at most $\Omega(\frac{Res_Y(f,g)}{d})$ irreducible factors in $K[X,Y]$.
\end{theorem}
\begin{proof} \ Let $f$ and $g$ be as in the statement of our theorem, and let us assume to the contrary that $f$ decomposes as a product of more than $\Omega(\frac{Res_Y(f,g)}{d})$ irreducible factors over $K[X]$, say $f(X,Y)=f_{1}(X,Y)\cdots f_{\ell}(X,Y)$, with $\ell>\Omega(\frac{Res_Y(f,g)}{d})$, $f_{1},\dots,f_{\ell}\in K[X,Y]$, $f_{1},\dots,f_{\ell}$ irreducible over $K[X]$. Consider again the factorizations of $f$ and $g$, say 
\begin{align*}
f(X,Y)& =a_{n}(Y-\theta _{1})\cdots (Y-\theta _{n}), \\
g(X,Y)& =b_{m}(Y-\xi _{1})\cdots (Y-\xi _{m}),
\end{align*}
with $\theta _{1},\ldots ,\theta _{n},\xi _{1},\ldots ,\xi _{m}\in \overline{%
K(X)}$. Since $\ell>\Omega(\frac{Res_Y(f,g)}{d})$ and
\begin{equation}\label{dfractie}
d\cdot \frac{Res_{Y}(f,g)}{d}=Res_{Y}(f_{1},g)\cdots Res_{Y}(f_{\ell},g),
\end{equation}
we deduce that at least one of the polynomials $Res_{Y}(f_{1},g),\dots, Res_{Y}(f_{\ell},g)\in
K[X]$, say $Res_{Y}(f_{1},g)$, will not be affected after division by  $\frac{Res_Y(f,g)}{d}$ in both sides of (\ref{dfractie}). In particular, this will imply that $Res_{Y}(f_{1},g)$ must be a divisor of $d(X)$, and hence 
$\deg Res_{Y}(f_{1},g)\leq \deg d$, which may be also written as
\begin{equation}\label{gradulRezmaimic}
|Res_{Y}(f_{1},g)|_\rho \leq \rho^{\deg d}.
\end{equation}
The proof continues now as in the case of Theorem \ref{thm11generala}, and we obtain again the inequality $|Res_Y(f_{1},g)|_\rho \geq |b_{m}|_\rho$. One of our hypotheses was that $\min\{\deg a_0, \deg b_m\}> \deg d$, so $\deg b_m> \deg d$, which also reads $|b_{m}|_\rho>\rho^{\deg d}$. We have thus reached the conclusion that $|Res_Y(f_{1},g)|_\rho > \rho^{\deg d}$, which contradicts (\ref{gradulRezmaimic}), showing that our polynomial $f$ is a product of at most $\Omega(\frac{Res_Y(f,g)}{d})$ irreducible factors over $K[X]$.

Assume now that $g$ decomposes as a product of more than $\Omega(\frac{Res_Y(f,g)}{d})$ irreducible factors over $K[X]$, say $g(X,Y)=g_{1}(X,Y)\cdots g_{\ell}(X,Y)$, with $\ell>\Omega(\frac{Res_Y(f,g)}{d})$, $g_{1},\dots,g_{\ell}\in K[X,Y]$, $g_{1},\dots,g_{\ell}$ irreducible over $K[X]$. We deduce as before that one of the factors, say $g_1$, satisfies
\begin{equation}\label{PrincipiulCutieiPentrugNou}
|Res_{Y}(f,g_1)|_\rho \leq \rho^{\deg d}.
\end{equation}
As in the proof of Theorem \ref{thm11generala}, we deduce that $|Res_Y(f,g_1)|_\rho \geq |a_0|_\rho$, and since one of our hypotheses was that $\deg a_0>\deg d$, which also reads $|a_{0}|_\rho>\rho^{\deg d}$, we conclude that $|Res_Y(f,g_1)|_\rho > \rho^{\deg d}$, which contradicts (\ref{PrincipiulCutieiPentrugNou}) and completes the proof.  
\end{proof}

\begin{remark}\label{RemarcaMultivar2} We mention that Remark \ref{RemarcaMultivar} also applies in the case of Theorem \ref{thmGeneralaNoua} with $\deg d$ instead of $\delta_k$.
\end{remark}

The following two results show that under some additional assumptions on the coefficients, the conclusion on the irreducibility of $f$ and $g$ in Corollary \ref{corothm11generala} extends to linear combinations of $f$ and $g$ with scalars in $K$.
\begin{theorem}
\label{Globala1Multivar} Let $K$ be a field, $f(X,Y)=\sum\nolimits_{i=0}^{n}a_{i}Y^{i}$, $g(X,Y)=\sum\nolimits_{i=0}^{m}b_{i}Y^{i}\in K[X,Y]$, with $a_{0},\dots
,a_{n},b_{0},\dots ,b_{m}\in K[X]$, $a_{0}a_{n}b_{0}b_{m}\neq 0$, $n<m$. Assume
that $f$ has no nonconstant factors in $K[X]$, that $\gcd (b_{n+1},\dots ,b_{m})\in K$, and that $Res_{Y}(f,g)$ is irreducible over $K$. If 
\begin{eqnarray*}
\deg a_{0} & > & \max \{\deg a_{1},\ldots ,\deg a_{n}\},\\
\deg b_{m} & \geq & \max \{\deg b_{0},\ldots ,\deg b_{m-1}\},\\
\deg a_i & \leq & \deg b_i \quad for\ \ i=0,\dots ,n,
\end{eqnarray*}
then $\alpha f+\beta g$ is irreducible in $K[X,Y]$ for any $\alpha ,\beta \in K$, not both zero. 
\end{theorem}
\begin{proof}
\ First of all we note that our condition that $\gcd (b_{n+1},\dots ,b_{m})\in K$ prevents $g$ and $\alpha f+\beta g$ to have nonconstant factors in $K[X]$. Next, we see by Corollary \ref{corothm11generala} that both $f$ and $g$ are irreducible in $K[X,Y]$, so the conclusion follows trivially in the case when one of $\alpha$ and $\beta$ is zero. We may therefore assume that $\alpha\neq 0$ and $\beta\neq 0$. It is easy to check that since $\deg _Y(\alpha f+\beta g)=m$, we have
\[
Res_Y(f,\alpha f+\beta g)=\beta^nRes_Y(f,g),
\]
which is irreducible over $K$, as $\beta$ is a nonzero constant. Now, since
\[
\alpha f+\beta g=(\alpha a_0+\beta b_0)+(\alpha a_1+\beta b_1)Y+\cdots +(\alpha a_n+\beta b_n)Y^n+\sum_{n+1\leq i\leq m}\beta b_iY^i,
\]
to apply Corollary \ref{corothm11generala} to our pair $(f,\alpha f+\beta g)$,  it remains to check that
\[
\deg \beta b_m\geq \max\thinspace  \{ \deg\thinspace (\alpha a_0+\beta b_0), \dots ,\deg\thinspace (\alpha a_n+\beta b_n), \max_{n+1\leq i\leq m-1}\deg \beta b_{i}\},
\]
if $m\geq n+2$, or to check that
\[
\deg \beta b_m\geq \max\thinspace  \{ \deg\thinspace  (\alpha a_0+\beta b_0), \dots ,\deg\thinspace  (\alpha a_n+\beta b_n)\},
\]
if $m=n+1$. In both cases the conclusion is an immediate consequence of our assumptions that $\deg b_{m} \geq \max\thinspace  \{\deg b_{0},\ldots ,\deg b_{m-1}\}$ and $\deg a_i \leq \deg b_i$ for $i=0,\dots ,n$. 
\end{proof}

\begin{theorem}
\label{Globala2Multivar} Let $K$ be a field, $f(X,Y)=\sum\nolimits_{i=0}^{n}a_{i}Y^{i}$, $g(X,Y)=\sum\nolimits_{i=0}^{n}b_{i}Y^{i}\in K[X,Y]$, with $a_{0},\dots
,a_{n},b_{0},\dots ,b_{n}\in K[X]$, $a_{0}a_{n}b_{0}b_{n}\neq 0$, and such that $a_j\in K\setminus\{0\}$ and $b_j\in K\setminus\{0\}$ for some index $j\in\{1,\dots,n-1\}$. If $Res_{Y}(f,g)$ is irreducible over $K$ and
\begin{eqnarray*}
\deg a_{0} & > & \max \{\deg a_{1},\ldots ,\deg a_{n}\},\\
\deg b_{n} & \geq & \max \{\deg b_{0},\ldots ,\deg b_{n-1}\},\\
\deg a_i & \leq & \deg b_i \quad for \ \ i=0,\dots ,n-1,
\end{eqnarray*}
then $\alpha f+\beta g$ is irreducible in $K[X,Y]$ for any $\alpha ,\beta \in K$ with $\alpha a_j+\beta b_j\neq 0$. 
\end{theorem}
\begin{proof}
\ We first notice that the fact that $a_j\in K\setminus\{0\}$ and $b_j\in K\setminus\{0\}$ for some index $j\in\{1,\dots,n-1\}$ prevents $f$ and $g$ to have nonconstant factors in $K[X]$, and the same will also hold for any linear combination $\alpha f+\beta g$ for which $\alpha a_j+\beta b_j\neq 0$. Again, by Corollary \ref{corothm11generala}, both $f$ and $g$ must be irreducible in $K[X,Y]$, so the conclusion follows trivially in the case when one of $\alpha$ and $\beta$ is zero. We will therefore assume that $\alpha\neq 0$ and $\beta\neq 0$. It is worth-mentioning here that $\alpha f+\beta g\neq 0$ for any nonzero $\alpha$ and $\beta$, since $Res_Y(f,g)$ was assumed to be irreducible over $K$. We notice now that by our hypotheses one may successively deduce that
\[
\deg a_n\leq\max \thinspace \{\deg a_1,\dots ,\deg a_n\}<\deg a_0\leq \deg b_0\leq \deg b_n,
\]
so $\deg a_n<\deg b_n$, which shows that $\deg\thinspace (\alpha a_n+\beta b_n)=\deg b_n$. In particular, this also shows that $\deg _Y(\alpha f+\beta g)=n$, so $Res_Y(f,\alpha f+\beta g)=\beta^nRes_Y(f,g)$, which is irreducible over $K$ for any nonzero constant $\beta $. Now, since
$\alpha f+\beta g=(\alpha a_0+\beta b_0)+(\alpha a_1+\beta b_1)Y+\cdots +(\alpha a_n+\beta b_n)Y^n$,
to apply Corollary \ref{corothm11generala} to our pair $(f,\alpha f+\beta g)$,  it remains to check that
\[
\deg\thinspace (\alpha a_n+\beta b_n)\geq \max\thinspace  \{ \deg\thinspace  (\alpha a_0+\beta b_0), \dots ,\deg\thinspace  (\alpha a_{n-1}+\beta b_{n-1})\}.
\]
Taking into account the fact that $\deg\thinspace (\alpha a_n+\beta b_n)=\deg b_n$, this inequality will obviously hold, since $\deg b_{n} \geq \max\thinspace  \{\deg b_{0},\ldots ,\deg b_{n-1}\}$ and $\deg a_i \leq \deg b_i$ for $i=0,\dots ,n-1$. \end{proof}

A bivariate analogue of Corollary \ref{thm10} is the following irreducibility criterion.
\begin{theorem}
\label{thmCuAj} Let $K$ be a field, $f(X,Y)=\sum\nolimits_{i=0}^{n}a_{i}Y^{i},%
\ g(X,Y)=\sum\nolimits_{i=0}^{m}b_{i}Y^{i}\in K[X,Y]$, with $a_{0},\dots
,a_{n},b_{0},\dots ,b_{m}\in K[X]$, $a_{0}a_{n}b_{0}b_{m}\neq 0$, and assume that $f$ and $g$ have no nonconstant factors in $K[X]$, and that $Res_{Y}(f,g)$ is irreducible over $K$. Assume also that
\[
\Delta:=\deg b_0-\max\{\deg b_1,\dots ,\deg b_m\}>0.
\]
If for an index $j\in \{1,\dots ,n\}$ we have
\begin{eqnarray}
\deg a_{j} & > & \max _{k<j}\{\deg a_k+(k-j)\Delta/m\}\quad {\rm and}\label{aj1}\\
\deg a_{j} & > & \max _{k>j}\{\deg a_k+(k-j)(\deg b_0-\deg b_m)\},\label{aj2}
\end{eqnarray}
then both $f$ and $g$ are irreducible in $K[X,Y]$. 
\end{theorem}
\begin{proof}
\ Let $f$ and $g$ be as in the statement of our theorem, and let us assume to the contrary that $f$ decomposes as a product of two nonconstant factors in $K[X,Y]$, say $f(X,Y)=f_{1}(X,Y)f_{2}(X,Y)$, with $f_{1},f_{2}\in K[X,Y]$, $\deg _Yf_{1}\geq 1$ and $\deg _Yf_{2}\geq 1$. Consider as before the factorizations of $f$ and $g$, say 
\begin{align*}
f(X,Y)& =a_{n}(Y-\theta _{1})\cdots (Y-\theta _{n}), \\
g(X,Y)& =b_{m}(Y-\xi _{1})\cdots (Y-\xi _{m}),
\end{align*}
with $\theta _{1},\ldots ,\theta _{n},\xi _{1},\ldots ,\xi _{m}\in \overline{%
K(X)}$. Since $Res_{Y}(f,g)$ is irreducible over $K$, and
\begin{equation}\label{dfractieCuAj}
Res_{Y}(f,g)=Res_{Y}(f_{1},g)\cdot Res_{Y}(f_{2},g),
\end{equation}
we deduce that one of the polynomials $Res_{Y}(f_{1},g)$ and $Res_{Y}(f_{2},g)$ must be a constant in $K$, say $Res_{Y}(f_{1},g)\in K$. In particular, this will imply that 
$\deg Res_{Y}(f_{1},g)=0$, which may be also written as
$|Res_{Y}(f_{1},g)|_\rho =1$.
Since $f_{1}$ is a factor of $f$, it will
factorize over $\overline{K(X)}$ as 
$f_{1}(X,Y)=C_{t}(Y-\theta _{1})\cdots
(Y-\theta _{t})$, say, with $t\geq 1$, $C_{t}\in K[X]$, $C_{t}\neq 0$, so 
\begin{equation}\label{formageneralaNouaCuAj}
1=|Res_Y(f_{1},g)|_\rho=|C_{t}|_\rho^{m}|b_{m}|_\rho^{t}\prod\limits_{i=1}^{t}\prod%
\limits_{j=1}^{m}|\theta _{i}-\xi _{j}|_\rho.
\end{equation}
Let us denote now
\[
\mathcal{A}=\max\limits_{k<j}\frac{\deg a_k-\deg a_j}{j-k}\quad {\rm and}\quad \mathcal{B}=\min\limits_{k>j}\frac{\deg a_k-\deg a_j}{j-k}.
\]
With this notation, we deduce by (\ref{aj1}) and (\ref{aj2}) that
\begin{equation}\label{AsiB}
\mathcal{A}<\frac{\Delta}{m}\leq \Delta\leq \deg b_0-\deg b_m<\mathcal{B},
\end{equation}
which in particular shows that $\mathcal{A}<\mathcal{B}$. 
We will prove now that for each root $\theta_i$ of $f$ we must either have $|\theta_i|_\rho\leq \rho^\mathcal{A}$, or $|\theta_i|_\rho\geq \rho^\mathcal{B}$. To see this, assume to the contrary that $\rho ^{\mathcal{A}}<|\theta_i|_\rho<\rho ^{\mathcal{B}}$ for some index $i\in\{1,\dots ,n\}$. Since $a_j\neq 0$ we see from $\rho ^{\mathcal{A}}<|\theta_i|_\rho$ that $|a_j|_\rho|\theta _i|_\rho ^j>|a_k|_\rho|\theta _i|_\rho ^k$ for each $k<j$, while from $|\theta_i|_\rho<\rho ^{\mathcal{B}}$ we deduce that we also have $|a_j|_\rho|\theta _i|_\rho ^j>|a_k|_\rho|\theta _i|_\rho ^k$ for each $k>j$. If we take the maximum with respect to $k$ in all these inequalities, we obtain:
\begin{equation}\label{ajmaimare}
|a_j|_\rho|\theta _i|_\rho ^j>\max\limits _{k\neq j}|a_k|_\rho|\theta _i|_\rho ^k.
\end{equation}
On the other hand, since $f(X,\theta _i)=0$ and our absolute value also satisfies the triangle inequality, we must have
\[
0\geq |a_j|_\rho|\theta _i|_\rho ^j-|\sum_{k\neq j}a_k\theta_i^k|_\rho\geq |a_j|_\rho|\theta _i|_\rho ^j-\max\limits _{k\neq j}|a_k|_\rho|\theta _i|_\rho ^k,
\]
which contradicts (\ref{ajmaimare}).
Therefore, for an index $i\in\{1,\dots ,n\}$ and an index $j\in\{1,\dots ,m\}$ we either have 
\begin{eqnarray*}
|\theta _{i}-\xi _{j}|_\rho & \geq & |\theta _{i}|_\rho-|\xi _{j}|_\rho\geq \rho^\mathcal{B}-\max\limits_{j}|\xi _{j}|_\rho,\quad {\rm if}\quad |\theta_i|_\rho\geq \rho^\mathcal{B},\quad {\rm or}\\
|\theta _{i}-\xi _{j}|_\rho & \geq & |\xi _{j}|_\rho-|\theta _{i}|_\rho\geq \min\limits _{j}|\xi _{j}|_\rho-\rho^\mathcal{A},\quad \ {\rm if}\quad |\theta_i|_\rho\leq \rho^\mathcal{A}.
\end{eqnarray*}
To estimate the terms $\max\limits_{j}|\xi _{j}|_\rho$ and $\min\limits _{j}|\xi _{j}|_\rho$ in these inequalities, we will use Lemma \ref{CauchyNonArchimedean}. Our assumption that $\Delta>0$ implies that $\max\{|b_1|_\rho,\dots ,|b_m|_\rho \}<|b_0|_\rho$, so by (\ref{Cauchy}) and (\ref{CauchyReversed}) applied to $g$ we see that all the roots $\xi_j$ of $g$ satisfy the inequalities
\[
\left(\frac{|b_0|_\rho}{\max\{|b_1|_\rho,\dots,|b_{m}|_\rho\}}\right)^{\frac{1}{m}}\leq |\xi_j|_\rho\leq \frac{|b_0|_\rho}{|b_m|_\rho},
\]
which further leads us to
\begin{eqnarray*}
|\theta _{i}-\xi _{j}|_\rho & \geq & \rho^\mathcal{B}-\rho^{\deg b_0-\deg b_m},\quad {\rm if}\quad |\theta_i|_\rho\geq \rho^\mathcal{B},\quad {\rm and}\\
|\theta _{i}-\xi _{j}|_\rho & \geq & \rho ^{\Delta/m}-\rho^\mathcal{A},\quad \quad \quad \ \ \thinspace \thinspace {\rm if}\quad |\theta_i|_\rho\leq \rho^\mathcal{A}.
\end{eqnarray*}
In view of (\ref{AsiB}), all we need now is to observe that for a sufficiently large $\rho $ both $\rho^\mathcal{B}-\rho^{\deg b_0-\deg b_m}$ and $\rho ^{\Delta/m}-\rho^\mathcal{A}$ will be larger than $1$, implying that $|\theta _{i}-\xi _{j}|_\rho >1$ for a sufficiently large $\rho$.
This will contradict (\ref{formageneralaNouaCuAj}), as $|C_{t}|_\rho\geq 1$ and $|b_{m}|_\rho\geq 1$. Thus $f$ must be irreducible in $K[X,Y]$. The fact that $g$ too is irreducible follows in a similar way.
\end{proof}

Some immediate consequences of our results in the bivariate case are similar results for polynomials in $r\geq 2$ variables $X_{1},X_{2},\ldots ,X_{r}$
over $K$.
 For any polynomial $f\in K[X_{1},\ldots ,X_{r}]$ and any $j\in
\{1,\ldots ,r\}$ we denote by $\deg _{j}f$ the degree of $f$ as a polynomial
in $X_{j}$ with coefficients in $K[X_{1},\dots ,\widehat{X}_{j},\dots
,X_{r}] $, where the hat stands for the fact that the corresponding variable
is missing. For instance, using this notation, one has the following extension of Corollary \ref{corothm11generala} to polynomials in $r\geq 2$ variables $X_{1},X_{2},\ldots ,X_{r}$
over $K$, obtained by
writing $Y$ for $X_{r}$ and $X$ for a suitable $X_{j}$, and by replacing the
field $K$ with $K(X_{1},\dots ,\widehat{X}_{j},\dots ,X_{r-1})$.

\begin{corollary}
\label{coro2thm11} Let $K$ be a field, fix integers $m,n,r\geq 2$, and let $%
f(X_{1},\ldots ,X_{r})=\sum\nolimits_{i=0}^{n}a_{i}X_{r}^{i}$ and $%
g(X_{1},\ldots ,X_{r})=\sum\nolimits_{i=0}^{m}b_{i}X_{r}^{i}$ with $%
a_{0},\ldots ,a_{n},b_{0},\ldots ,b_{m}\in K[X_{1},\ldots ,X_{r-1}]$, $%
a_{0}a_{n}b_{0}b_{m}\neq 0$, and assume that $f$ and $g$ have no nonconstant factors in $K[X_1,\dots ,X_{r-1}]$, and that for an index $j\in \{1,\ldots
,r-1\}$, $Res_{X_{r}}(f,g)$ as a polynomial in $X_{j}$ is irreducible over $%
K(X_{1},\dots ,\widehat{X}_{j},\dots ,X_{r-1})$. If 
$\deg _{j}a_{0}>\max \{\deg _{j}a_{1},\ldots ,\deg _{j}a_{n}\}$ and $\deg _{j}b_{m}\geq \max \{\deg _{j}b_{0},\ldots ,\deg _{j}b_{m-1}\}$, then both $f$ and $g$ are irreducible in $K[X_{1},\dots ,X_{r}]$. The same conclusion will also hold if we interchange the signs $>$ and $\geq$ in these two inequalities.
\end{corollary}

\section{Some numerical examples.}

1) Let $f(X)=-1+X-2X^{2}+3X^{3}+9X^{4}$
and $g(X)=35+3X-X^{2}-X^{3}+X^{4}$. Since $|Res(f,g)|=9\,794\,181\,403$, a prime
number, and for $A=1$ we have $|a_{4}|>\sum_{i=0}^{3}|a_{i}|\cdot A^{i-4}$ and $|b_{0}|>\sum_{i=1}^{4}|b_{i}|(A+1)^{i} $, we conclude by Corollary \ref{thm6} that both $f$ and $g$ are irreducible.

2) Let us take $f(X)=X^{5}+X^{4}-X^{3}-2X^{2}-3X+15$ and apply Theorem \ref
{thmUnita} with $b=11$ and $c=2$. We have $11^{5}f(\frac{2}{11})=2\,316\,511$,
which is a prime number, and $15=|a_{0}|>\sum_{i=1}^{5}|a_{i}|(|1/b|+|c/b|)^{i}=%
\frac{160\,128}{161\,051}$, so $f$ must be irreducible over $\mathbb{Q}$.

3) Consider the Littlewood polynomial
$f(X)=X^{6}-X^{5}-X^{4}+X^{3}+X^{2}-X+1 $ and apply Corollary \ref{corothm1}
with $b=4$ and $c=15$. Since $4^{6}f(\frac{15}{4})=7\,805\,461$, a prime
number, $f$ must be irreducible over $\mathbb{Q}$.

4) Consider now the polynomials 
\begin{align*}
f_{1}(X) & =92+X-3X^{2}+3X^{3}-X^{4}+X^{5}, \\
f_{2}(X) & =93+2X-2X^{2}+2X^{3}-X^{4}+X^{5}, \\
f_{3}(X) & =-100+X-5X^{2}+X^{3}+X^{4}+X^{5}, \\
f_{4}(X) & =-104-X-9X^{2}+X^{3}+X^{4}+X^{5},
\end{align*}
and apply Theorem \ref{thm3}. In all these four cases we have $%
|a_{0}|>\sum_{i=1}^{5}|a_{i}|\sqrt{2^i}$ and 
\[
(a_{0}-a_{2}+a_{4})^{2}+(a_{1}-a_{3}+a_{5})^{2}=94^{2}+1=8\,837,
\]
which is a prime number, so all these polynomials are irreducible.

5) Let $f(X)=361+8X+X^{2}-X^{3}+X^{4}-X^{5}+X^{6}$ and apply Corollary \ref
{corolthm4} with $m=2$. Simple computations show that 
\[
(a_{0}+2a_{2}+4a_{4}+8a_{6})^{2}-2(a_{1}+2a_{3}+4a_{5})^{2}=375^{2}-2\cdot
2^{2}=140\,617, 
\]
which is a prime number, and $361=|a_{0}|>\sum_{i=1}^{6}|a_{i}|\cdot (\sqrt{3})^i=39+20\sqrt{3}$, so $f$ must be an irreducible polynomial.

6) Let us take now the Littlewood polynomial $%
f(X)=-1+X-X^{2}+X^{3}-X^{4}+X^{5}+X^{6}$ and apply again Corollary \ref{corolthm4}, this time with $m=13$. We have 
\[
(a_{0}+a_{2}m+a_{4}m^{2}+a_{6}m^{3})^{2}-m(a_{1}+a_{3}m+a_{5}m^{2})^{2}=3%
\,620\,839,
\]
a prime number, and the condition $|a_{6}|>\sum_{i=0}^{5}|a_{i}|\cdot (\sqrt{12})^{i-6}$ is obviously satisfied, so $f$ must be irreducible.

7) For any positive integer $k$ and any integer $a> 2^{3k+2}-2$ such that $a^2-a+1$ is a prime number, the polynomial $f_{k,a}(X)=a+X+X^2+\cdots +X^{3k+1}$ is irreducible. The conclusion follows by Corollary \ref{corolthm4,5} with $n=3k+1$, $S_0=a+k$, $S_1=k+1$ and $S_2=k$, since $S_0^2+S_1^2+S_2^2-S_0S_1-S_0S_2-S_1S_2=a^2-a+1$ and the condition $a_0>\sum_{i=1}^{n}|a_{i}|2^i$ reduces to $a>2^{3k+2}-2$.

8) Let $n\geq 3$ and write the Eisensteinian polynomial $f(X)=X^{n}+5X+5$ as 
\[
f(X)=(X^{n}+2)+(X^{n-1}-X^{n-2})+(X^{n-2}-X^{n-3})+(X^{n-3}-X^{n-1})+(5X+3). 
\]
We may use then Corollary \ref{coro1thm11} to conclude that for each $n\geq 3$ the polynomial
\[
(X^{n}+2)+(X^{n-1}-X^{n-2})Y+(X^{n-2}-X^{n-3})Y^{2}+(X^{n-3}-X^{n-1})Y^{3}+(5X+3)Y^{4} 
\]
is irreducible in $\mathbb{Q}[X,Y]$. \medskip

{\bf Acknowledgements} This work was done in the frame of the GDRI ECO-Math.

\end{document}